\documentclass[11pt]{article}

\usepackage{amsmath}          
\usepackage{amssymb}          
\usepackage{latexsym}         
\usepackage{amsxtra}          
\usepackage[latin1]{inputenc} 
\usepackage{eucal}            
\usepackage{bbm}              
\usepackage{longtable}        
\usepackage{exscale}          
\usepackage{theorem}          
\usepackage{graphicx}         
\usepackage[sort]{cite}       
\usepackage{mathrsfs}         
\usepackage{array}            
\usepackage{stmaryrd}         
\usepackage{chemarr}          
\usepackage{paralist}         
\usepackage{mathtools}        

\title{Involutions and Representations for Reduced Quantum Algebras}
\author{\textbf{Simone Gutt}\thanks{\texttt{sgutt@ulb.ac.be}} \\[0.2cm]
  \begin{minipage}{6cm}
      \begin{center}{\small{
          D{\'e}partement de Math{\'e}matique \\
          Universit{\'e} Libre de Bruxelles \\
          Campus Plaine, C. P. 218 \\
          Boulevard du Triomphe \\
          B-1050 Bruxelles \\
          Belgium}}
      \end{center}
  \end{minipage}
  \qquad
  \begin{minipage}{6cm}
      \begin{center}
         {\small{ Universit\'e de Metz\\
          D{\'e}partement de Math{\'e}matique\\
          Ile du Saulcy \\
          F-57045~Metz Cedex 01\\
          France}}
      \end{center}
  \end{minipage}
  \vspace{0.2cm} \\
  and
  \\[0.2cm]
  \addtocounter{footnote}{5}
  \textbf{Stefan Waldmann}\thanks{\texttt{Stefan.Waldmann@physik.uni-freiburg.de}}
  \\[0.2cm]
  \begin{minipage}{8cm}
      \begin{center}{\small{
          Fakult{\"a}t f{\"u}r Mathematik und Physik \\
          Albert-Ludwigs-Universit{\"a}t Freiburg \\
          Physikalisches Institut \\
          Hermann Herder Strasse 3 \\
          D 79104 Freiburg \\
          Germany}}
      \end{center}
  \end{minipage}
}
\date{\today}

\newcounter{comment}[section] 

\renewcommand{\mathbb}[1]{\mathbbm{#1}}          
\allowdisplaybreaks                              
\renewcommand{\arraystretch}{1.2}                

          
\newcommand{\refitem}[1] {~\textit{\ref{#1}.)}}

\theoremheaderfont{\normalfont\bfseries}
\theorembodyfont{\itshape}
\newtheorem{lemma}{Lemma}[section]
\newtheorem{proposition}[lemma]{Proposition}
\newtheorem{theorem}[lemma]{Theorem}
\newtheorem{corollary}[lemma]{Corollary}
\newtheorem{definition}[lemma]{Definition}

\theorembodyfont{\rmfamily}

\newtheorem{remark}[lemma]{Remark}

\makeatletter
\newcommand\qedsymbol{\hbox{$\boxempty$}}
\newcommand\qed{\relax\ifmmode\boxempty\else
  {\unskip\nobreak\hfil\penalty50\hskip1em\null\nobreak\hfil\qedsymbol
  \parfillskip=\z@\finalhyphendemerits=0\endgraf}\fi}

\newcommand\subqedsymbol{\hbox{$\triangledown$}}
\newcommand\subqed{\relax\ifmmode\triangledown\else
  {\unskip\nobreak\hfil\penalty50\hskip1em\null\nobreak\hfil\subqedsymbol
  \parfillskip=\z@\finalhyphendemerits=0\endgraf}\fi}
\makeatother

\newenvironment{proof}[1][{}]{\par\noindent Proof{#1}. }{\qed}

\newcommand{\Lie}        {\operatorname{\mathscr{L}\!}}    
\newcommand{\cc}[1]      {\overline{{#1}}}              
\newcommand{\id}         {\operatorname{\mathsf{id}}}   
\newcommand{\supp}       {\operatorname{\mathrm{supp}}}  
\newcommand{\Pol}        {\operatorname{\mathrm{Pol}}}   
\newcommand{\tr}         {\operatorname{\mathsf{tr}}}    
\newcommand{\ad}         {\operatorname{\mathrm{ad}}}    
\newcommand{\Ad}         {\operatorname{\mathrm{Ad}}}    
\newcommand{\Hom}        {\operatorname{\mathsf{Hom}}}   
\newcommand{\End}        {\operatorname{\mathsf{End}}}   
\newcommand{\SP}[1]      {\left\langle{#1}\right\rangle} 
\newcommand{\Unit}       {\mathbb{1}}

\newcommand{\cl}         {\mathrm{cl}}

\newcommand{\I}          {\mathrm{i}}

\newcommand{\D}          {\operatorname{\mathrm{d}}}

\newcommand{\Log}        {\operatorname{\mathrm{Log}}}

\newcommand{\Anti}       {\Lambda}
\newcommand{\AntiC}      {\Lambda_{\mathbb{C}}}
\newcommand{\Sym}        {\mathrm{S}}

\newcommand{\Cinfty}     {\mathcal{C}^\infty}

\newcommand{\ins}        {\operatorname{\mathrm{i}}}

\newcommand{\lie}[1]     {\mathfrak{#1}}
\newcommand{\red}        {\mathrm{red}}
\newcommand{\can}        {\mathrm{can}}
\newcommand{\prol}       {\operatorname{\mathrm{prol}}}
\newcommand{\pr}         {\operatorname{\mathrm{pr}}}
\newcommand{\image}      {\operatorname{\mathrm{im}}}
\newcommand{\opp}        {\mathrm{opp}}
\newcommand{\ring}[1]    {\mathsf{#1}}
\newcommand{\hor}        {\mathrm{hor}}                
\newcommand{\ver}        {\mathrm{ver}}                

\newcommand{\Secinfty}   {\Gamma^\infty}
\newcommand{\Density}[1][{}] {|\Anti^{\mathrm{top}}|^{#1}}
\newcommand{\Dleft}      {\D^{\mathrm{left}}}

\newcommand{\codim}      {\operatorname{\mathrm{codim}}}

\newcommand{\koszul}     {\partial}
\newcommand{\nice}       {\mathrm{nice}}
\newcommand{\deform}[1]  {\boldsymbol{#1}}

\newcommand{\starred}    {\mathbin{\star_\red}}
\newcommand{\bulletred}  {\mathbin{\bullet_\red}}
\newcommand{\conv}       {\mathbin{*}}
\newcommand{\starredk}   {\mathbin{\star_\red^{(\kappa)}}}
\newcommand{\bulletredk} {\mathbin{\bullet_\red^{(\kappa)}}}
\newcommand{\deformiotak}{\deform{\iota}^{\deform{*}}_\kappa}
\newcommand{\bulletk}    {\mathbin{\bullet_\kappa}}
\newcommand{\bulletcan}  {\mathbin{\bullet_\can}}
\newcommand{\starG}      {\mathbin{\star_{\mathrm{G}}}}
\newcommand{\starstd}    {\mathbin{\star_{\mathrm{Std}}}}
\newcommand{\stdrep}     {\mathop{\varrho_{\mathrm{Std}}}}
\newcommand{\weylrep}    {\mathop{\varrho_{\mathrm{Weyl}}}}

\newcommand{\at}[1]      {\big|_{#1}}
\newcommand{\At}[1]      {\Big|_{#1}}
\newcommand{\Diffop}     {\operatorname{\mathrm{DiffOp}}}

\newcommand{\Bounded}[1][{}] {\operatorname{\mathfrak{B}}_{\scriptscriptstyle{#1}}}
\newcommand{\Finite}[1][{}]  {\operatorname{\mathfrak{F}}_{\scriptscriptstyle{#1}}}

\newcommand{\Cinftycf}   {\mathcal{C}^\infty_{\mathrm{cf}}}

\newcommand{\Bimod}[5] {\sideset{^{\scriptscriptstyle{#1}}_{\scriptscriptstyle{#2}}}{^{\scriptscriptstyle{#4}}_{\scriptscriptstyle{#5}}}{\operatorname{#3}}}

\newcommand{\EA}   {\Bimod{}{}{\mathcal{E}}{}{\mathcal{A}}}
\newcommand{\EpA}  {\Bimod{}{}{\mathcal{E}}{\prime}{\mathcal{A}}}
\newcommand{\EppA} {\Bimod{}{}{\mathcal{E}}{\prime\prime}{\mathcal{A}}}
\newcommand{\BEA}  {\Bimod{}{\mathcal{B}}{\mathcal{E}}{}{\mathcal{A}}}
\newcommand{\AtEB} {\Bimod{}{\mathcal{A}}{\widetilde{\mathcal{E}}}{}{\mathcal{B}}}
\newcommand{\AccEB}{\Bimod{}{\mathcal{A}}{\cc{\mathcal{E}}}{}{\mathcal{B}}}
\newcommand{\CFB}  {\Bimod{}{\mathcal{C}}{\mathcal{F}}{}{\mathcal{B}}}
\newcommand{\AEAred}  {\Bimod{}{\mathcal{A}}{\mathcal{E}}{}{\mathcal{A}_\red}}
\newcommand{\AredHD}  {\Bimod{}{\mathcal{A}_\red}{\mathcal{H}}{}{\mathcal{D}}}

\newcommand{\IP}[4]{{\,}_{\scriptscriptstyle{#2}\!\!}\left\langle{{#1}}\right\rangle^{\scriptscriptstyle{#3}}_{\scriptscriptstyle{#4}}}

\newcommand{\SPA}[1]     {\IP{{#1}}{}{}{\mathcal{A}}}
\newcommand{\SPEA}[1]    {\IP{{#1}}{}{\mathcal{E}}{\mathcal{A}}}
\newcommand{\SPEpA}[1]   {\IP{{#1}}{}{\mathcal{E}'}{\mathcal{A}}}
\newcommand{\BSP}[1]     {\IP{{#1}}{\mathcal{B}}{}{}}

\newcommand{\SPFB}[1]    {\IP{{#1}}{}{\mathcal{F}}{\mathcal{B}}}
\newcommand{\SPFEA}[1]   {\IP{{#1}}{}{\mathcal{F} \tensor \mathcal{E}}{\mathcal{A}}}

\newcommand{\tensor}[1][{}] {\mathbin{\otimes_{\scriptscriptstyle{#1}}}}
\newcommand{\itensor}[1][{}]{\mathbin{\widehat{\otimes}_{\scriptscriptstyle{#1}}}}
\newcommand{\extensor}      {\mathbin{\otimes_{\scriptscriptstyle{\mathsf{ext}}}}}

\newcommand{\rep}[1][{}]  {\sideset{^*}{_{#1}}{\operatorname{\textrm{-}\mathsf{rep}}}}
\newcommand{\Rep}[1][{}]  {\sideset{^*}{_{#1}}{\operatorname{\textrm{-}\mathsf{Rep}}}}
\newcommand{\smod}[1][{}] {\sideset{^*}{_{#1}}{\operatorname{\textrm{-}\mathsf{mod}}}}
\newcommand{\sMod}[1][{}] {\sideset{^*}{_{#1}}{\operatorname{\textrm{-}\mathsf{Mod}}}}

\begin{document}

\maketitle

\begin{abstract}
    In the context of deformation quantization, there exist various
    procedures to deal with the quantization of a reduced space
    $M_\red$.  We shall be concerned here mainly with the classical
    Marsden-Weinstein reduction, assuming that we have a proper action
    of a Lie group $G$ on a Poisson manifold $M$, with a moment map
    $J$ for which zero is a regular value.  For the quantization, we
    follow \cite{bordemann.herbig.waldmann:2000a} (with a simplified
    approach) and build a star product $ \starred$ on $M_\red$ from a
    strongly invariant star product $\star$ on $M$. The new questions
    which are addressed in this paper concern the existence of natural
    $^*$-involutions on the reduced quantum algebra and the
    representation theory for such a reduced $^*$-algebra.

    We assume that $\star$ is Hermitian and we show that the choice of
    a formal series of smooth densities on the embedded coisotropic
    submanifold $C=J^{-1}(0)$, with some equivariance property,
    defines a $^*$-involution for $\starred$ on the reduced space.
    Looking into the question whether the corresponding
    $^*$-involution is the complex conjugation (which is a
    $^*$-involution in the Marsden-Weinstein context) yields a new
    notion of quantized unimodular class.

    We introduce a left $(\Cinfty(M)[[\lambda]], \star)$-submodule and
    a right $(\Cinfty(M_\red)[[\lambda]], \starred)$-submodule
    $\Cinftycf(C)[[\lambda]]$ of $C^\infty(C)[[\lambda]]$; we define
    on it a $\Cinfty(M_\red)[[\lambda]]$-valued inner product and we
    establish that this gives a strong Morita equivalence bimodule
    between $\Cinfty(M_\red)[[\lambda]]$ and the finite rank operators
    on $\Cinftycf(C)[[\lambda]]$. The crucial point is here to show
    the complete positivity of the inner product. We obtain a Rieffel
    induction functor from the strongly non-degenerate
    $^*$-representations of $ \left(\Cinfty(M_\red)[[\lambda]],
        \starred\right)$ on pre-Hilbert right $\mathcal{D}$-modules to
    those of $\left(\Cinfty(M)[[\lambda]], \star \right),$ for any
    auxiliary coefficient $^*$-algebra $\mathcal{D}$ over
    $\mathbb{C}[[\lambda]]$.
\end{abstract}

\medskip
\medskip
\medskip
\medskip

\noindent
MSC Classification (2000): 53D55 (primary), 53D20, 16D90, 81S10 (secondary)

\newpage

\tableofcontents

\newpage

%
%

\section{Introduction}
\label{sec:Introduction}

Some mathematical formulations of quantizations are based on the
algebra of observables and consist in replacing the classical algebra
of observables $\mathcal{A}$ (typically complex-valued smooth
functions on a Poisson manifold $M$) by a non commutative one
$\deform{\mathcal{A}}$.  Formal deformation quantization was 
introduced in \cite{bayen.et.al:1978a}; it constructs the quantum
observable algebra by means of a formal deformation (in the sense of
Gerstenhaber) of the classical algebra. Given a Poisson manifold $M$
and the classical algebra $\mathcal{A} = C^\infty (M)$ of
complex-valued smooth functions, a star product on $M$ is a
$\mathbb{C}[[\lambda]]$-bilinear associative multiplication on
$\Cinfty(M)[[\lambda]]$ with
\begin{equation}
    f \star g = \sum_{r=0}^\infty \lambda^r C_r (f, g),
\end{equation}
where $C_0(f, g) = fg$ and $C_1 (f, g) - C_1(g, f) = \I \{f, g\}$,
where the $C_r$ are bidifferential operators so that $1 \star f = f =
f \star 1$ for all $f \in \Cinfty(M)[[\lambda]]$. The algebra of
quantum observables is $\deform{\mathcal{A}} =
(\Cinfty(M)[[\lambda]],\star)$.

An important classical tool to ``reduce the number of variables'',
i.e. to start from a ``big'' Poisson manifold $M$ and construct a
smaller one $M_{\red}$, is given by reduction: one considers an
embedded coisotropic submanifold in the Poisson manifold, $\iota: C
\hookrightarrow M$ and the canonical foliation of $C$ which we assume
to have a nice leaf space $M_\red$. In this case one knows that
$M_\red$ is a Poisson manifold in a canonical way.

We shall consider here the particular case of the Marsden-Weinstein
reduction: let $\mathsf{L}: G \times M \longrightarrow M$ be a smooth
left action of a connected Lie group $G$ on $M$ by Poisson
diffeomorphisms and assume we have an $\ad^*$-equivariant momentum
map. The constraint manifold $C$ is now chosen to be the level surface
of $J$ for momentum $0 \in \lie{g}^*$ (thus we assume, for simplicity,
that $0$ is a regular value). Then $ C = J^{-1}(\{0\})$ is an embedded
submanifold which is coisotropic. The group $G$ acts on $C$ and the
reduced space is the orbit space of this group action of $G$ on $C$
(in order to guarantee a good quotient we assume that $G$ acts freely
and properly).

Given a mathematical formulation of quantization, one studies then a
quantized version of reduction and how ``quantization commutes with
reduction''. This has been done in the framework of deformation
quantization by various authors \cite{bordemann.herbig.waldmann:2000a,
  cattaneo.felder:2007a, cattaneo.felder:2004a}.  We shall use here
the approach proposed by Bordemann \cite{bordemann:2005a} .  Since the
emphasis is put in our quantization scheme on the observable algebra,
recall that at the classical level if $\iota: C \hookrightarrow M$ is
an embedded coisotropic submanifold, one considers $\mathcal{J}_C =
\{f \in \Cinfty(M) \; | \; \iota^*f = 0\} = \ker \iota^*$ the
vanishing ideal of $C$ [which is an ideal in the associative algebra
$\Cinfty(M)$ and a Poisson subalgebra of $\Cinfty(M)$], defining
$\mathcal{B}_C = \left\{f \in \Cinfty(M) \; | \; \{f, \mathcal{J}_C\}
    \subseteq \mathcal{J}_C \right\}$, and assuming that the canonical
foliation of $C$ has a nice leaf space $M_\red$ (i.e. a structure of a
smooth manifold such that the canonical projection $\pi: C
\longrightarrow M_\red $ is a submersion);  then
\begin{equation}
    \mathcal{B}_C \big/ \mathcal{J}_C
    \ni [f] \; \mapsto \; \iota^*f \in
    \pi^* \Cinfty(M_\red) = \mathcal{A}_{red}
\end{equation}
induces an isomorphism of Poisson algebras. We recall in
Section~\ref{subsec:ClassicalKoszulResolution} this isomorphism in our
setting of Marsden Weinstein reduction using the Koszul complex.

Passing to a deformation quantized version of phase space reduction,
one starts with a formal star product $\star$ on $M$.  The associative
algebra $\deform{\mathcal{A}} = (\Cinfty(M)[[\lambda]], \star)$ is
playing the role of the quantized observables of the big system.  A
good analog of the vanishing ideal $\mathcal{J}_C$ will be a left
ideal $\deform{\mathcal{J}}_C \subseteq \Cinfty(M)[[\lambda]] $ such
that the quotient $\Cinfty(M)[[\lambda]] \big/ \deform{\mathcal{J}}_C$
is in $\mathbb{C}[[\lambda]]$-linear bijection to the functions
$\Cinfty(C)[[\lambda]]$ on $C$.  Then we define
$\deform{\mathcal{B}}_C = \{a \in \deform{\mathcal{A} } \; | \; [a,
\deform{\mathcal{J}}_C] \subseteq \deform{\mathcal{J}}_C \}$, i.e. the
normalizer of $\deform{\mathcal{J}}_C$ with respect to the commutator
Lie bracket of $\deform{\mathcal{A}},$ and consider the associative
algebra $\deform{\mathcal{B}}_C \big/ \deform{\mathcal{J}}_C$ as the
reduced algebra $\deform{\mathcal{A}}_{red}$. Of course, this is only
meaningful if one can show that $\deform{\mathcal{B}}_C \big/
\deform{\mathcal{J}}_C$ is in $\mathbb{C}[[\lambda]]$-linear bijection
to $\Cinfty(M_\red)[[\lambda]]$ in such a way, that the isomorphism
induces a star product $\starred$ on $M_\red$.  Starting from a
strongly invariant star product on $M$, we describe in
Section~\ref{subsec:QuantizedKoszulComplex} a method to construct a
good left ideal inspired by the BRST approach in
\cite{bordemann.herbig.waldmann:2000a} but simpler as we only need the
deformation of the Koszul part of the BRST complex.

The algebra of quantum observables is not only an associative algebra
but is has a $^*$-involution; in the usual picture, where observables
are represented by operators, this $^*$-involution corresponds to the
passage to the adjoint operator. In the framework of deformation
quantization, a way to have a $^*$-involution on
$\deform{\mathcal{A}}=(C^\infty(M)[[\lambda]], \star)$ is to ask the
star product to be Hermitian, i.e such that $\cc{f \star g} = \cc{g}
\star \cc{f}$ and the $^*$-involution is then just given by complex
conjugation.  A first question that we discuss in this paper is how to
get in a natural way a $^*$-involution for the reduced algebra,
assuming that $\star$ is a Hermitian star product on $M$.  We want a
construction coming from the reduction process itself; we start with a
left ideal $\deform {\mathcal{J}} \subseteq \deform{\mathcal{A}}$ in
some algebra and take $ \deform{\mathcal{B}} \big/
\deform{\mathcal{J}}$ as the reduced algebra, where $
\deform{\mathcal{B}}$ is the normaliser of $\deform{\mathcal{J}}$ in
$\deform{\mathcal{A}}.$ If now $\deform{\mathcal{A}}$ is in addition a
$^*$-algebra we have to construct a $^*$-involution for
$\deform{\mathcal{B}} \big/ \deform{\mathcal{J}}$.  From all relevant
examples in deformation quantization one knows that
$\deform{\mathcal{J}}$ is only a left ideal, hence can not be a
$^*$-ideal and thus $\deform{\mathcal{B}}$ can not be a
$^*$-subalgebra. Consequently, there is no obvious way to define a
$^*$-involution on the quotient.

The main idea here is to use a representation of the reduced quantum
algebra and to translate the notion of the adjoint. Observe that
$\deform{\mathcal{B}} \big/ \deform{\mathcal{J}}$ can be identified
(with the opposite algenbra structure) to the algebra of
$\deform{\mathcal{A}}$-linear endomorphisms of $\deform{\mathcal{A}}
\big/ \deform{\mathcal{J}}.$ We shall use an additional positive
linear functional i.e. a $\mathbb{C}[[\lambda]]$-linear functional
$\omega: \deform{\mathcal{A}} \longrightarrow \mathbb{C}[[\lambda]]$
such that $ \omega(a^*a) \ge 0$ for all $a \in \deform{\mathcal{A}},$
where positivity in $\mathbb{C}[[\lambda]]$ is defined using the
canonical ring ordering of $\mathbb{R}[[\lambda]]$.  Defining the
Gel'fand ideal of $\omega$ by $\deform{\mathcal{J}}_\omega = \left\{a
    \in \deform{\mathcal{A}} \; \big| \; \omega(a^*a) = 0 \right\}$,
one can construct a $^*$-representation (the GNS representation), of
$\deform{\mathcal{A}}$ on $\mathcal{H}_\omega = \deform{ \mathcal{A}}
\big/ \deform{ \mathcal{J}}_\omega$ with the pre Hilbert space
structure defined via $\SP{\psi_a, \psi_b} = \omega(a^*b)$ where
$\psi_a$ denotes the equivalence class of $a\in\deform{ \mathcal{A}}.$
Then the algebra of $\deform{\mathcal{A}}$-linear endomorphisms of
$\mathcal{H}_\omega$ (with the opposite structure) is equal to
$\deform{\mathcal{B}} \big/ \deform{\mathcal{J}}_\omega$.  Hence, to
define a $^*$-involution on our reduced quantum algebra, the main idea
is now to look for a positive linear functional $\omega$ such that the
left ideal $\deform{\mathcal{J}}$ we use for reduction coincides with
the Gel'fand ideal $\deform{\mathcal{J}}_\omega$ and such that all
left $\deform{\mathcal{A}}$-linear endomorphisms of
$\mathcal{H}_\omega$ are adjointable.  In this case
$\deform{\mathcal{B}} \big/ \deform{\mathcal{J}}$ becomes in a natural
way a $^*$-subalgebra of the set $\Bounded(\mathcal{H}_\omega)$ of
adjointable maps.  Up to here, the construction is entirely algebraic
and works for $^*$-algebras over rings of the form $\ring{C} =
\ring{R}(\I)$ with $\I^2 = -1$ and an ordered ring $\ring{R}$, instead
of $\mathbb{C}[[\lambda]]$ and $\mathbb{R}[[\lambda]]$.

We show in Section~\ref{subsec:ReducedInvolution} that the choice of a
formal series of smooth densities $\sum_{r=0}^\infty \lambda^r \mu_r
\in \Gamma^\infty(|\Anti^{\mathrm{top}}| T^*C)[[\lambda]]$ on the
coisotropic submanifold $C$ such that $\cc{\mu} = \mu$ is real, $\mu_0
> 0$ and so that $\mu$ transforms under the $G$-action as
$\mathsf{L}^*_{g^{-1}} \mu = \frac{1}{\Delta(g)} \mu$ where $\Delta$
is the modular function yields a positive linear functional which
defines a $^*$-involution on the reduced space. Along the way we
identify the corresponding GNS representation.  We show that in the
classical Marsden Weinstein reduction, complex conjugation is a
$^*$-involution of the reduced quantum algebra. Looking in general to
the question whether the $^*$-involution corresponding to a series of
densities $\mu$ is the complex conjugation yields  a new notion of
quantized unimodular class.

The next problem that we tackle in this paper is the study of the
representations of the reduced algebra with the $^*$-involution given
by complex conjugation.  We want to relate the categories of modules
of the big algebra and the reduced algebra.  The usual idea is to use
a bimodule and the tensor product to pass from modules of one algebra
to modules of the other. In the context of quantization and reduction
this point of view has been pushed forward by Landsman
\cite{landsman:1998a}, mainly in the context of geometric
quantization. Contrary to his approach, we have, by construction of
the reduced star product, a bimodule structure on
$\Cinfty(C)[[\lambda]]$.  We want more properties to have a relation
between the $^*$-representations of our algebras on inner product
modules. The notions are transferred, following
\cite{bursztyn.waldmann:2005b, bursztyn.waldmann:2001a}, from the
theory of Hilbert modules over $C^*$-algebras to our more algebraic
framework and are recalled in
Sections~\ref{subsec:AlgebraValuedInnerProducts} and
\ref{subsec:StrongMoritaEquivalence}.

We look at $\Cinftycf(C) = \left\{ \phi \in \Cinfty(C) \; \big| \;
    \supp(\phi) \cap \pi^{-1}(K) \; \textrm{is compact for all
      compact} \; K \subseteq M_\red\right\};$ then
$\Cinftycf(C)[[\lambda]]$ is a left $(\Cinfty(M)[[\lambda]],
\star)$-module and a right $(\Cinfty(M_\red)[[\lambda]],
\starred)$-module; we define on it a
$\Cinfty(M_\red)[[\lambda]]$-valued inner product and we establish in
Section~\ref{sec:StrongMoritaEquivalenceBimodule} that this bimodule
structure and inner product on $\Cinftycf(C)[[\lambda]]$ gives a
strong Morita equivalence bimodule between
$\Cinfty(M_\red)[[\lambda]]$ and the finite rank operators on
$\Cinftycf(C)[[\lambda]]$. The crucial point is here to show the
complete positivity of the inner product. In some sense, the resulting
equivalence bimodule can be viewed as a deformation of the
corresponding classical limit which is studied independently in the
context of the strong Morita equivalence of the crossed product
algebra with the reduced algebra.  If $G$ is not finite, the finite
rank operators do not have a unit, thus we have a first non-trivial
example of a $^*$-equivalence bimodule for star product algebras going
beyond the unital case studied in \cite{bursztyn.waldmann:2002a}.

We show that the $^*$-algebra $(\Cinfty(M)[[\lambda]], \star)$ acts on
$\Cinftycf(C)[[\lambda]]$ in an adjointable way with respect to the
$\Cinfty(M_\red)[[\lambda]]$-valued inner product and we obtain a
Rieffel induction functor from the strongly non-degenerate
$^*$-representations of $ \left(\Cinfty(M_\red)[[\lambda]],
    \starred\right)$ on pre-Hilbert right $\mathcal{D}$-modules to
those of $\left(\Cinfty(M)[[\lambda]], \star \right),$ for any
auxiliary coefficient $^*$-algebra $\mathcal{D}$ over
$\mathbb{C}[[\lambda]]$.

In Section~\ref{sec:Example}, we consider the geometrically trivial
situation $M = M_\red \times T^*G$ where on $M_\red$ a Poisson bracket
and a corresponding star product $\starred$ is given while on $T^*G$
we use the canonical symplectic Poisson structure and the canonical
star product $\starG$ from \cite{gutt:1983a}.  Up to the completion
issues, the Rieffel induction with $\Cinftycf(M_\red \times
G)[[\lambda]]$ simply consists in tensoring the given
$^*$-representation of $\Cinfty(M_\red)[[\lambda]]$ with the
Schrödinger representation (see \eqref{eq:SchroedingerRep}) on
$\Cinfty_0(G)[[\lambda]]$.

\medskip

\noindent
\textbf{Acknowledgements:} It is a pleasure to thank Martin Bordemann,
Henrique Bursztyn, and Dominic Maier for valuable discussions and
comments. We thank the FNRS for a grant which allowed SW to be in
Brussels during part of the preparation of this paper.

%
%

\section{The classical construction}
\label{sec:ClassicalConstruction}

In this section we recall some basic features of phase space reduction
in order to establish our notation. The material is entirely standard,
we essentially follow \cite{bordemann.herbig.waldmann:2000a}.

%
%

\subsection{The geometric framework}
\label{subsec:GeometricFramework}

Throughout this paper, $M$ will denote a Poisson manifold with Poisson
bracket $\{\cdot, \cdot\}$ coming from a real Poisson tensor. Thus the
complex-valued functions $\Cinfty(M)$ on $M$ become a Poisson
$^*$-algebra with respect to $\{\cdot, \cdot\}$ and the pointwise
complex conjugation $f \mapsto \cc{f}$ as $^*$-involution.

Let $\iota: C \hookrightarrow M$ be an embedded submanifold and denote
by $\mathcal{J}_C = \{f \in \Cinfty(M) \; | \; \iota^*f = 0\} = \ker
\iota^*$ the vanishing ideal of $C$ which is an ideal in the
associative algebra $\Cinfty(M)$.  Then $C$ is called
\emph{coisotropic} (or \emph{first class constraint}) if
$\mathcal{J}_C$ is a Poisson subalgebra of $\Cinfty(M)$.  In this case
we define
\begin{equation}
    \label{eq:BCDef}
    \mathcal{B}_C = 
    \left\{
        f \in \Cinfty(M) 
        \; \big| \;
        \{f, \mathcal{J}_C\} \subseteq \mathcal{J}_C 
    \right\},
\end{equation}
which turns out to be the largest Poisson subalgebra of $\Cinfty(M)$
which contains $\mathcal{J}_C$ as a Poisson ideal. The geometric
meaning of $\mathcal{B}_C$ is now the following: since $C$ is
coisotropic we have a canonical foliation of $C$ which we assume to
have a nice leaf space $M_\red$. More technically, we assume that
$M_\red$ can be equipped with the structure of a smooth manifold such
that the canonical projection
\begin{equation}
    \label{eq:piCtoMred}
    \pi: C \longrightarrow M_\red
\end{equation}
is a submersion. In this case one knows that $M_\red$ is a Poisson
manifold in a canonical way such that
\begin{equation}
    \label{eq:BcmodJCtoCinftyMred}
    \mathcal{B}_C \big/ \mathcal{J}_C 
    \ni [f] \mapsto \iota^*f \in
    \pi^* \Cinfty(M_\red)
\end{equation}
induces an isomorphism of Poisson algebras, see e.g.
\cite{bordemann.herbig.waldmann:2000a, bordemann:2005a}. In fact, we
will give a detailed proof of this in some more particular situation
later.

While in principle, phase space reduction and its deformation
quantization analogs are interesting for general coisotropic
submanifolds, we shall consider only a very particular case, the
Marsden-Weinstein reduction: let $\mathsf{L}: G \times M
\longrightarrow M$ be a smooth left action of a connected Lie group
$G$ on $M$ by Poisson diffeomorphisms. Moreover, assume we have an
$\ad^*$-equivariant momentum map
\begin{equation}
    \label{eq:MomentumMap}
    J: M \longrightarrow \lie{g}^*
\end{equation}
for this action, i.e. an $\ad^*$-equivariant smooth map with values in
the dual $\lie{g}^*$ of the Lie algebra $\lie{g}$ of $G$ such that the
Hamiltonian vector field $X_{J_\xi} = \{\cdot, J_\xi\}$ for $J_\xi \in
\Cinfty(M)$ with $J_\xi(p) = \SP{J(p), \xi}$ coincides with the
fundamental vector field $\xi_M\in \Gamma^\infty(TM)$ for all $\xi \in
\lie{g}$. We use the convention that $\xi \mapsto \Lie_{\xi_M}$
defines an anti-homomorphism of Lie algebras, i.e. 
\begin{equation}
    \label{eq:FundamentalVectorField}
    \xi_M(p) =
    \frac{\D}{\D t}\At{t=0} \mathsf{L}_{\exp(t\xi)}(p)    
\end{equation}
for all $p \in M$.  The $\ad^*$-equivariance can be expressed by
\begin{equation}
    \label{eq:JadstarEquivariant}
    \{J_\xi, J_\eta\} = J_{[\xi, \eta]}
\end{equation}
for all $\xi, \eta \in \lie{g}$ and it is equivalent to
$\Ad^*$-equivariance with respect to $G$ as $G$ is connected.

The constraint manifold $C$ is now chosen to be the level surface of
$J$ for momentum $0 \in \lie{g}^*$. Thus we assume that $0$ is a value
and, for simplicity, that $0$ is even a \emph{regular value}. Then
\begin{equation}
    \label{eq:CLevelNullofJ}
    C = J^{-1}(\{0\})
\end{equation}
is an embedded submanifold which turns out to be coisotropic. The
group $G$ acts on $C$ as well since $0$ is $\Ad^*$-invariant. We use
the same symbol $\mathsf{L}$ for this action. The
quotient~\eqref{eq:piCtoMred} turns out to be just the orbit space of
this group action of $G$ on $C$, i.e.
\begin{equation}
    \label{eq:piCtoMredIsCmodG}
    \pi: C \longrightarrow M_\red = C \big/ G.
\end{equation}
In order to guarantee a good quotient we assume that $G$ acts
\emph{freely} and \emph{properly}: in this case $C$ is a principal
$G$-bundle over $M_\red$ and \eqref{eq:piCtoMredIsCmodG} is a
surjective submersion as wanted. To be conform with the usual
principal bundle literature, sometimes we pass to the corresponding
right action of $G$ on $C$ given by $\mathsf{R}: C \times G
\longrightarrow C$ with $\mathsf{R}_g(p) = \mathsf{L}_{g^{-1}}(p)$ as
usual. Note however, that $\xi_M$ as well as $\xi_C$ are the
fundamental vector fields with respect to the \emph{left} actions on
$M$ and $C$, respectively, as in \eqref{eq:FundamentalVectorField}.

%
%

\subsection{The classical Koszul resolution}
\label{subsec:ClassicalKoszulResolution}

For $C$ we can now define the classical Koszul resolution. As a
complex we consider $\Cinfty(M, \AntiC^\bullet \lie{g}) = \Cinfty(M)
\otimes \AntiC^\bullet \lie{g}$ with the canonical free
$\Cinfty(M)$-module structure. The group $G$ acts on $\Cinfty(M,
\AntiC^\bullet \lie{g})$ by the combined action of $G$ on the manifold
and the adjoint action on $\lie{g}$ extended to $\AntiC^\bullet
\lie{g}$ by automorphisms of the $\wedge$-product. We shall denote
this $G$-action and the corresponding $\lie{g}$-action by
$\varrho$. The \emph{Koszul differential} is now defined by
\begin{equation}
    \label{eq:ClassicalKoszul}
    \koszul x = \ins(J) x,
\end{equation}
where $x \in \Cinfty(M, \AntiC^\bullet \lie{g})$ and $\ins(J)$ denotes
the insertion of $J$ at the first position in the $\AntiC^\bullet
\lie{g}$-part of $x$. If $e_1, \ldots, e_N \in \lie{g}$ denotes a
basis with dual basis $e^1, \ldots, e^N \in \lie{g}^*$ then we can
write $J = J_a e^a$ with scalar functions $J_a \in \Cinfty(M)$. Here
and in the following we shall use Einstein's summation convention. The
Koszul differential is then
\begin{equation}
    \label{eq:KoszulInBasis}
    \koszul x = J_a \ins(e^a) x.
\end{equation}
Clearly, $\koszul$ is a super derivation of the canonical
$\wedge$-product on $\Cinfty(M, \AntiC^\bullet \lie{g})$ of degree
$-1$ and $\koszul^2 = 0$. Moreover, $\koszul$ is $\Cinfty(M)$-linear
hence we have a complex of free $\Cinfty(M)$-modules. Sometimes we
write $\koszul_k$ for the restriction of $\koszul$ to the
antisymmetric degree $k \ge 1$.

Before proving that this indeed gives an acyclic complex we make some
further simplifying assumptions needed later in the quantum version.
We assume that $G$ acts properly not only on $C$ but on all of $M$. In
this case we can find an open neighbourhood $M_\nice \subseteq M$ of
$C$ with the following properties: there exists a $G$-equivariant
diffeomorphism
\begin{equation}
    \label{eq:PhiMniceCtimesgStar}
    \Phi: M_\nice \longrightarrow U_\nice \subseteq C \times \lie{g}^*
\end{equation}
onto an open neighbourhood $U_\nice$ of $C \times \{0\}$, where the
$G$-action on $C \times \lie{g}^*$ is the product action of the one on
$C$ and $\Ad^*$, such that for each $p \in C$ the subset $U_\nice \cap
(\{p\} \times \lie{g}^*)$ is star-shaped around the origin $\{p\}
\times \{0\}$ and the momentum map $J$ is given by the projection onto
the second factor, i.e. $J \at{M_\nice} = \pr_2 \circ \Phi$. For a
proof of this well-known fact see e.g.
\cite[Lem.~3]{bordemann.herbig.waldmann:2000a}.

We can use this particular tubular neighbourhood $M_\nice$ of $C$ to
define the following \emph{prolongation map}
\begin{equation}
    \label{eq:Prol}
    \prol: \Cinfty(C)
    \ni \phi \mapsto \prol(\phi) = (\pr_1 \circ \Phi)^* \phi \in
    \Cinfty(M_\nice). 
\end{equation}
By the equivariance of the diffeomorphism $\Phi$ the prolongation is
$G$-equivariant as well, i.e. for $g \in G$ we have
\begin{equation}
    \label{eq:prolEquivariant}
    \mathsf{L}^*_g \prol(\phi) = \prol(\mathsf{L}^*_g \phi).
\end{equation}
The prolongation deserves its name as clearly we have for all $\phi
\in \Cinfty(C)$
\begin{equation}
    \label{eq:ProlisProl}
    \iota^* \prol (\phi) = \phi.
\end{equation}
The last ingredient from the classical side is the following homotopy
which we also define only on $M_\nice$ for convenience. Let $x \in
\Cinfty(M_\nice, \AntiC^k \lie{g})$. Since $U_\nice$ is star-shaped,
we set
\begin{equation}
    \label{eq:Homotopyk}
    (h_k x)(p) = e_a \wedge \int_0^1 t^k 
    \frac{\partial (x \circ \Phi^{-1})}{\partial \mu_a}
    (c, t\mu)
    \D t,
\end{equation}
where $\Phi(p) = (c, \mu)$ for $p \in M_\nice$ and $\mu_a$ denote the
linear coordinates on $\lie{g}^*$ with respect to the basis $e^1,
\ldots, e^N$. The collection of all these maps $h_k$ gives a map
\begin{equation}
    \label{eq:Homotopy}
    h: \Cinfty(M_\nice, \AntiC^\bullet \lie{g})
    \longrightarrow
    \Cinfty(M_\nice, \AntiC^{\bullet+1} \lie{g}),
\end{equation}
whose properties are summarized in the following proposition, see
e.g. \cite[Lem.~5 \& 6]{bordemann.herbig.waldmann:2000a}:
\begin{proposition}
    \label{proposition:ClassicalHomotopy}
    The Koszul complex $(\Cinfty(M_\nice, \AntiC^\bullet \lie{g}),
    \koszul)$ is acyclic with explicit homotopy $h$ and homology
    $\Cinfty(C)$ in degree $0$. In detail, we have
    \begin{equation}
        \label{eq:KoszulAcyclic}
        h_{k-1} \koszul_k + \koszul_{k+1} h_k
        = \id_{\Cinfty(M_\nice, \AntiC^k \lie{g})}
    \end{equation}
    for $k \ge 1$ and
    \begin{equation}
        \label{eq:proliota}
        \prol \iota^* + \koszul_1 h_0 = \id_{\Cinfty(M_\nice)}
    \end{equation}
    as well as $\iota^* \koszul_1 = 0$. Thus the Koszul complex is a
    free resolution of $\Cinfty(C)$ as $\Cinfty(M_\nice)$-module. We
    have
    \begin{equation}
        \label{eq:hNullprol}
        h_0 \prol = 0,
    \end{equation}
    and all the homotopies $h_k$ are $G$-equivariant.
\end{proposition}

Here resolution means that the homology at $k = 0$ is isomorphic to
$\Cinfty(C)$ as a $\Cinfty(M_\nice)$-module: indeed, the image of
$\koszul_1$ is just $\mathcal{J}_C \cap \Cinfty(M_\nice)$ as
\eqref{eq:proliota} shows. This gives immediately
\begin{equation}
    \label{eq:CinftyCiskerkoszulNullmodJC}
    \Cinfty(M_\nice) \big/ (\mathcal{J}_C \cap \Cinfty(M_\nice))
    =
    \ker \koszul_0 \big/ (\mathcal{J}_C \cap \Cinfty(M_\nice))
    \cong
    \Cinfty(C),
\end{equation}
induced via $\iota^*$ and $\prol$.

It will be useful to consider the augmented Koszul complex where in
degree $k = -1$ one puts $\Cinfty(C)$ and re-defines $\koszul_0 =
\iota^*$. With $h_{-1} = \prol$ the proposition yields
\begin{equation}
    \label{eq:TrivialHomologyForAllk}
    h_{k-1} \koszul_k + \koszul_{k+1} h_k = \id_k
\end{equation}
for all $k \ge -1$. This augmented complex has now trivial homology in
\emph{all} degrees.

We can use the Koszul complex to prove \eqref{eq:BcmodJCtoCinftyMred}:
indeed, for $u \in \Cinfty(M_\red)$ we have $\prol(\pi^*u) \in
\mathcal{B}_C$ whence \eqref{eq:BcmodJCtoCinftyMred} is surjective.
The injectivity of \eqref{eq:BcmodJCtoCinftyMred} is clear by
definition. The Poisson bracket on $M_\red$ can then be defined
through \eqref{eq:BcmodJCtoCinftyMred} and gives explicitly
\begin{equation}
    \label{eq:ReducedPoissonBracket}
    \pi^* \{u, v\}_\red = \iota^*\{\prol(\pi^* u), \prol(\pi^*v)\}
\end{equation}
for $u, v \in \Cinfty(M_\red)$, since the left hand side of
\eqref{eq:BcmodJCtoCinftyMred} is canonically a Poisson algebra.
\begin{remark}[$M$ versus $M_\nice$]
    \label{remark:ComplexOverMnice}
    For simplicity, we have defined $\prol$ as well as the homotopy
    $h$ only on the neighbourhood $M_\nice$. In
    \cite{bordemann.herbig.waldmann:2000a} it was shown that one can
    extend the definitions to all of $M$ preserving the
    $G$-equivariance and the properties \eqref{eq:KoszulAcyclic},
    \eqref{eq:proliota}, and \eqref{eq:hNullprol}. Since for the phase
    space reduction in deformation quantization we will only need a
    very small neighbourhood (in fact: an infinitesimal one) of $C$,
    the neighbourhood $M_\nice$ is completely sufficient. The geometry
    of $M$ far away from $C$ will play no role in the following. Thus
    we may even assume $M_\nice = M$ without restriction in the
    following to simplify our notation.
\end{remark}

%
%

\section{The quantized bimodule structure}
\label{sec:QuantizedBimodule}

When passing to a deformation quantized version of phase space
reduction we have to reformulate everything in terms of now
non-commutative algebras where Poisson brackets are to be replaced by
commutators. We recall here a general approach to reduction as
proposed by Bordemann \cite{bordemann:2005a} as well as by Cattaneo
and Felder \cite{cattaneo.felder:2007a, cattaneo.felder:2004a} and
others, see also \cite{lu:1993a}.

Thus in the following, let $\star$ be a formal star product
\cite{bayen.et.al:1978a} on $M$, i.e. a
$\mathbb{C}[[\lambda]]$-bilinear associative multiplication for
$\Cinfty(M)[[\lambda]]$ with
\begin{equation}
    \label{eq:StarProduct}
    f \star g = \sum_{r=0}^\infty \lambda^r C_r (f, g),
\end{equation}
where $C_0(f, g) = fg$ and $C_1 (f, g) - C_1(g, f) = \I \{f, g\}$.
Moreover, we assume that $\star$ is bidifferential and satisfies $1
\star f = f = f \star 1$ for all $f \in \Cinfty(M)[[\lambda]]$.
Physically speaking, the formal parameter $\lambda$ corresponds to
Planck's constant $\hbar$ whenever we can establish convergence of the
above formal series, see e.g. \cite{dito.sternheimer:2002a} for a
review on deformation quantization and \cite{waldmann:2007a} for a
gentle introduction.

The first observation is that a good analog of the vanishing ideal
$\mathcal{J}_C$ will be a \emph{left ideal}: this is Dirac's old ideal
of ``weakly vanishing operators'' annihilating the ``true physical
states'' inside some ``unphysical, too big Hilbert space'', see
\cite{dirac:1964a} as well as \cite{lu:1993a}. Thus the general
situation is to have an associative algebra $\mathcal{A}$ playing the
role of the observables of the big system with a left ideal
$\mathcal{J} \subseteq \mathcal{A}$. The functions on the constraint
surface will correspond to the left $\mathcal{A}$-module $\mathcal{A}
\big/ \mathcal{J}$ in the non-commutative world. The following simple
proposition gives now a nice motivation how to define the reduced
algebra, i.e. the observables of the reduced system:
\begin{proposition}
    \label{proposition:AmodJ}
    Let $\mathcal{A}$ be a unital algebra with a left ideal
    $\mathcal{J} \subseteq \mathcal{A}$. Define
    \begin{equation}
        \label{eq:BDef}
        \mathcal{B} = 
        \{a \in \mathcal{A} 
        \; \big| \;
        [a, \mathcal{J}] \subseteq \mathcal{J} \},
    \end{equation}
    i.e. the normalizer of $\mathcal{J}$ with respect to the
    commutator Lie bracket of $\mathcal{A}$. Then $\mathcal{B}$ is the
    largest unital subalgebra of $\mathcal{A}$ such that $\mathcal{J}
    \subseteq \mathcal{B}$ is a two-sided ideal and
    \begin{equation}
        \label{eq:BmodJisEndAmodJ}
        \mathcal{B} \big/ \mathcal{J}
        \ni [b] \mapsto ([a] \mapsto [ab]) \in
        \End_{\mathcal{A}} 
        \left(\mathcal{A} \big/ \mathcal{J}\right)^\opp
    \end{equation}
    is an isomorphism of unital algebras.
\end{proposition}

This observation gives now the guideline for the reduction of star
products: for the star product $\star$ on $M$, we have to find a left
ideal $\deform{\mathcal{J}}_C \subseteq \Cinfty(M)[[\lambda]]$ such
that the quotient $\Cinfty(M)[[\lambda]] \big/ \deform{\mathcal{J}}_C$
is in $\mathbb{C}[[\lambda]]$-linear bijection to the functions
$\Cinfty(C)[[\lambda]]$ on $C$. Then we consider the associative
algebra $\deform{\mathcal{B}}_C \big/ \deform{\mathcal{J}}_C$ as the
reduced algebra. Of course, this is only meaningful if one can show
that $\deform{\mathcal{B}}_C \big/ \deform{\mathcal{J}}_C$ is in
$\mathbb{C}[[\lambda]]$-linear bijection to
$\Cinfty(M_\red)[[\lambda]]$ in such a way, that the isomorphism
induces a star product $\starred$ on $M_\red$. This is the general
reduction philosophy as proposed by
\cite{bordemann.herbig.waldmann:2000a, bordemann:2005a,
  cattaneo.felder:2007a, cattaneo.felder:2004a, lu:1993a} which makes
sense for general coisotropic submanifolds. We note that as a result
one obtains even a \emph{bimodule} structure on
$\Cinfty(C)[[\lambda]]$ where $(\Cinfty(M)[[\lambda]], \star)$ acts
from the left and $(\Cinfty(M_\red)[[\lambda]], \starred)$ acts from
the right.  Note also that the situation will be quite asymmetric in
general: while all left $\star$-linear endomorphisms are indeed given
by right multiplications with functions in
$\Cinfty(M_\red)[[\lambda]]$ according to
Proposition~\ref{proposition:AmodJ}, the converse needs not to be true
in general: In fact, we will see explicit counter-examples later.

%
%

\subsection{The quantized Koszul complex}
\label{subsec:QuantizedKoszulComplex}

We describe now a method how to construct a left ideal and a deformed
left module structure for the functions on $C$ inspired by the BRST
approach in \cite{bordemann.herbig.waldmann:2000a}. However, for us
things will be slightly simpler as we only need the Koszul part of the
BRST complex.

Before defining the deformed Koszul operator we have to make some
further assumptions on the star product $\star$ on $M$. First, we want
it to be \emph{$\lie{g}$-covariant}, i.e.
\begin{equation}
    \label{eq:StargCovariant}
    J_\xi \star J_\eta - J_\eta \star J_\xi
    = \I \lambda J_{[\xi, \eta]} 
\end{equation}
for all $\xi, \eta \in \lie{g}$. Second, we need $\star$ to be
\emph{$G$-invariant}, i.e.
\begin{equation}
    \label{eq:StarGInvariant}
    \mathsf{L}^*_g (f \star h)
    = (\mathsf{L}^*_g f) \star (\mathsf{L}^*_g h)
\end{equation}
for all $g \in G$ and $f, h \in \Cinfty(M)[[\lambda]]$. In general,
both conditions are quite independent but there is one way to
guarantee both features: we ask for a \emph{strongly invariant} star
product, see also \cite{arnal.cortet.molin.pinczon:1983a}. This means
\begin{equation}
    \label{eq:StronglyInvariant}
    J_\xi \star f - f \star J_\xi 
    = \I \lambda \{J_\xi, f\}
    = - \I \lambda \Lie_{\xi_M} f
\end{equation}
for all $f \in \Cinfty(M)[[\lambda]]$ and $\xi \in \lie{g}$. Indeed,
\eqref{eq:StronglyInvariant} clearly implies \eqref{eq:StargCovariant}
by taking $f = J_\eta$ using \eqref{eq:JadstarEquivariant}.  Since the
left hand side of \eqref{eq:StronglyInvariant} is a (quasi-inner)
derivation of $\star$ so is the right hand side. Thus the
invariance~\eqref{eq:StarGInvariant} follows by differentiation of $g
= \exp(t\xi)$ as usual. Note that $G$ is assumed to be connected in
the context of phase space reduction.
\begin{remark}
    \label{remark:StrongInvariance}
    Since the action of $G$ is assumed to be proper we find an
    invariant covariant derivative $\nabla$ on $M$. Out of this, one
    can construct strongly invariant star products by means of
    Fedosov's technique in the symplectic case \cite{fedosov:1996a}
    and, more generally, by Dolgushev's equivariant formality in the
    general Poisson case \cite{dolgushev:2005a}. Thus the, in general
    quite strong, assumption \eqref{eq:StronglyInvariant} is
    achievable in the case of a proper action of $G$.
\end{remark}
Using the $\wedge$-product for $\AntiC^\bullet \lie{g}$ we extend
$\star$ to $\Cinfty(M, \AntiC^\bullet \lie{g})$ in the canonical way.
This allows for the following definition:
\begin{definition}[Quantized Koszul operator]
    \label{definition:QuantizedKoszul} 
    Let $\kappa \in \mathbb{C}[[\lambda]]$. The quantized Koszul
    operator $\deform{\koszul}^{(\kappa)}: \Cinfty(M, \AntiC^\bullet
    \lie{g})[[\lambda]] \longrightarrow \Cinfty(M, \AntiC^{\bullet+1}
    \lie{g})[[\lambda]]$ is defined by
    \begin{equation}
        \label{eq:QuantizedKoszul}
        \deform{\koszul}^{(\kappa)} x
        =
        \ins(e^a) x \star J_a
        +
        \frac{\I\lambda}{2} C_{ab}^c 
        e_c \wedge \ins(e^a) \ins(e^b) x
        +
        \I \lambda \kappa \ins(\Delta) x,
    \end{equation}
    where $C_{ab}^c = e^c([e_a, e_b])$ are the structure constants of
    $\lie{g}$ and
    \begin{equation}
        \label{eq:ModularOneForm}
        \Delta (\xi) = \tr \ad (\xi)
        \quad
        \textrm{for}
        \quad
        \xi \in \lie{g}
    \end{equation}
    is the modular one-form $\Delta \in \lie{g}^*$ of $\lie{g}$.
\end{definition}
Note that with respect to the chosen basis we have
\begin{equation}
    \label{eq:ModularOneFormStructureConstants}
    \Delta = C_{ab}^b e^a.
\end{equation}
\begin{lemma}
    \label{lemma:QuantizedKoszul}
    Let $\star$ be strongly invariant and $\kappa \in
    \mathbb{C}[[\lambda]]$.
    \begin{compactenum}
    \item \label{item:KoszulCommutesInsDelta} One has
        $\deform{\koszul}^{(0)} \ins (\Delta) + \ins(\Delta)
        \deform{\koszul}^{(0)} = 0$.
    \item \label{item:DeformKoszulLeftLinear}
        $\deform{\koszul}^{(\kappa)}$ is left $\star$-linear.
    \item \label{item:DeformKoszulClassicalLimit} The classical limit
        of $\deform{\koszul}^{(\kappa)}$ is $\koszul$.
    \item \label{item:DeformKoszulInvariant}
        $\deform{\koszul}^{(\kappa)}$ is $G$-equivariant.
    \item \label{item:DeformKoszulSquareZero}
        $\deform{\koszul}^{(\kappa)} \circ \deform{\koszul}^{(\kappa)}
        = 0$.
    \end{compactenum}
\end{lemma}
\begin{proof}
    For the first part we note that the insertion of the
    \emph{constant} one-form $\Delta \in \lie{g}^*$ anti-commutes with
    the first part of $\deform{\koszul}^{(0)}$. It also anti-commutes
    with the second part as $\Delta$ vanishes on Lie brackets. The
    second and third part is clear. The fourth part is a simple
    computation. For the last part it is sufficient to consider the
    case $\kappa = 0$ which is a straightforward computation using the
    covariance of $\star$. Then $\ins(\Delta) \ins(\Delta) = 0$ and
    the first part give also the general case $\kappa \in
    \mathbb{C}[[\kappa]]$.
\end{proof}

The importance of the correction term $\I\lambda\kappa \ins(\Delta)$
will become clear in Section~\ref{subsec:PositiveFunctional}. For the
time being, $\kappa$ can be arbitrary. In particular, $\kappa = 0$
gives a very simple choice for the quantized Koszul operator. However,
we set
\begin{equation}
    \label{eq:DeformKoszulForSpecialKappa}
    \deform{\koszul} = \deform{\koszul}^{(\kappa = \frac{1}{2})}
\end{equation}
for abbreviation as this value of $\kappa$ will turn out to be the
most useful choice.  The following constructions will always depend
on $\kappa$. If we omit the reference to $\kappa$ in our notation, we
always mean the particular value of $\kappa$ as in
\eqref{eq:DeformKoszulForSpecialKappa}.

Following \cite{bordemann.herbig.waldmann:2000a} we obtain a
deformation of the restriction map $\iota^*$ as follows. We define
\begin{equation}
    \label{eq:deformiota}
    \deformiotak
    =
    \iota^*\left(
        \id + \left(
            \deform{\koszul}^{(\kappa)}_1 - \koszul_1
        \right)
        h_0
    \right)^{-1}:
    \Cinfty(M)[[\lambda]]
    \longrightarrow
    \Cinfty(C)[[\lambda]]
\end{equation}
and
\begin{equation}
    \label{eq:deformhNull}
    \deform{h}^{(\kappa)}_0
    =
    h_0 \left(
        \id + \left(
            \deform{\koszul}^{(\kappa)}_1 - \koszul_1
        \right)
        h_0
    \right)^{-1}:
    \Cinfty(M)[[\lambda]]
    \longrightarrow
    \Cinfty(M, \lie{g})[[\lambda]],
\end{equation}
which are both well-defined since $\deform{\koszul}^{(\kappa)}$ is a
deformation of $\koszul$. From
\cite[Prop.~25]{bordemann.herbig.waldmann:2000a} we know that
\begin{equation}
    \label{eq:DeformKoszulAndHomotopyNull}
    \deform{h}^{(\kappa)}_0 \prol = 0,
    \quad
    \deformiotak \deform{\koszul}^{(\kappa)}_1 = 0,
    \quad
    \textrm{and}
    \quad
    \deformiotak \prol = \id_{\Cinfty(C)[[\lambda]]}.
\end{equation}
Analogously to the definition of $\deform{h}^{(\kappa)}_0$ one can
also deform the higher homotopies $h_k$ by setting
\begin{equation}
    \label{eq:DeformedHomotopy}
    \deform{h}^{(\kappa)}_k 
    = h_k 
    \left(
        h_{k-1} \deform{\koszul}^{(\kappa)}_k
        + \deform{\koszul}^{(\kappa)}_{k+1}  h_k
    \right)^{-1},
\end{equation}
for which one obtains the following properties
\cite{bordemann.herbig.waldmann:2000a}:
\begin{lemma}
    \label{lemma:DeformedHomotopy}
    The deformed augmented Koszul complex, where
    $\deform{\koszul}^{(\kappa)}_0 = \deformiotak$, has
    trivial homology: with $\deform{h}^{(\kappa)}_{-1} = \prol$ one
    has
    \begin{equation}
        \label{eq:DeformedTrivialHomology}
        \deform{h}^{(\kappa)}_{k-1} \deform{\koszul}^{(\kappa)}_k
        +
        \deform{\koszul}^{(\kappa)}_{k+1} \deform{h}^{(\kappa)}_k
        =
        \id_{\Cinfty(M, \AntiC^k \lie{g})[[\lambda]]}
    \end{equation}
    for $k \ge 0$ and $\deformiotak \prol =
    \id_{\Cinfty(C)[[\lambda]]}$ for $k = -1$. Moreover, the maps
    $\deformiotak$ and $\deform{h}^{(\kappa)}_k$ are
    $G$-equivariant.
\end{lemma}

For $k = 0$ the homotopy equation~\eqref{eq:DeformedHomotopy} becomes
explicitly
\begin{equation}
    \label{eq:HomotopyForKNull}
    \prol \deformiotak 
    + \deform{\koszul}^{(\kappa)}_1 \deform{h}^{(\kappa)}_0 =
    \id_{\Cinfty(M)[[\lambda]]}. 
\end{equation}
In fact, we will only need this part of the Koszul resolution.
Finally, we mention the following locality feature of
$\deformiotak$ which is remarkable since the homotopy $h_0$
used in \eqref{eq:deformiota} is \emph{not} local, see
\cite[Lem.~27]{bordemann.herbig.waldmann:2000a}:
\begin{lemma}
    \label{lemma:RestrictionLocal}
    There is a formal series $S_\kappa = \id + \sum_{r=1}^\infty
    \lambda^r S^{(\kappa)}_r$ of $G$-invariant differential operators
    $S^{(\kappa)}_r$ on $M$ such that
    \begin{equation}
        \label{eq:deformedIotaDifferential}
        \deformiotak = \iota^* \circ S_\kappa.
    \end{equation}
    Moreover, $S_\kappa$ can be arranged such that $S_\kappa 1 = 1$.
\end{lemma}

%
%

\subsection{The reduced star product and the bimodule}
\label{subsec:ReducedStarProductBimodule}

Let us now use the deformed homotopy equation
\eqref{eq:DeformedTrivialHomology} to construct the bimodule structure
on $\Cinfty(C)[[\lambda]]$. This construction is implicitly available
in \cite{bordemann.herbig.waldmann:2000a}, see also
\cite{bordemann:2005a} for a more profound discussion.
\begin{definition}
    \label{definition:DeformedLeftModule}
    The deformed left multiplication of $\phi \in
    \Cinfty(C)[[\lambda]]$ by some $f \in \Cinfty(M)[[\lambda]]$ is
    defined by
    \begin{equation}
        \label{eq:fbulletphi}
        f \bulletk \phi
        = \deformiotak (f \star \prol(\phi)).
    \end{equation}
\end{definition}
This defines a left module structure indeed. Moreover, it has nice
locality and invariance properties which we summarize in the following
proposition:
\begin{proposition}
    \label{proposition:LeftModuleStructure}
    Let $\deform{\mathcal{J}}_C = \image
    \deform{\koszul}^{(\kappa)}_1$ be the image of the Koszul
    differential.
    \begin{compactenum}
    \item \label{item:JCisLeftIdeal} $\deform{\mathcal{J}}_C$ is a
        left $\star$-ideal.
    \item \label{item:CinftyMModJC} The left module
        $\Cinfty(M)[[\lambda]] \big/ \deform{\mathcal{J}}_C$ is
        isomorphic to $\Cinfty(C)[[\lambda]]$ equipped with
        $\bulletk$ via the mutually inverse isomorphisms
        \begin{equation}
            \label{eq:deformIotaIso}
            \Cinfty(M)[[\lambda]] \big/ \deform{\mathcal{J}}_C
            \ni [f] \mapsto \deformiotak f \in 
            \Cinfty(C)[[\lambda]]
        \end{equation}
        and
        \begin{equation}
            \label{eq:ClassOfProlIso}
            \Cinfty(C)[[\lambda]]
            \ni \phi \mapsto [\prol(\phi)] \in
            \Cinfty(M)[[\lambda]] \big/ \deform{\mathcal{J}}_C.
        \end{equation}
    \item \label{item:ModuleStructureBidifferential} The left module
        structure $\bulletk$ is bidifferential along $\iota^*$,
        i.e. we have $\mathbb{C}$-bilinear operators $L^{(\kappa)}_r:
        \Cinfty(M) \times \Cinfty(C) \longrightarrow \Cinfty(C)$ with
        \begin{equation}
            \label{eq:fbulletphiBidifferential}
            f \bulletk \phi 
            = \iota^*(f) \phi
            + \sum_{r=1}^\infty \lambda^r L^{(\kappa)}_r(f, \phi),
        \end{equation}
        where $L^{(\kappa)}_r$ is differential along $\iota^*$ in the
        first and differential in the second argument.
    \item \label{item:LeftModuleGInvariant} The left module structure
        is $G$-invariant in the sense that
        \begin{equation}
            \label{eq:LeftModuleGInvariant}
            \mathsf{L}^*_g  (f \bulletk \phi)
            = (\mathsf{L}^*_g f) \bulletk (\mathsf{L}^*_g \phi) 
        \end{equation}
        for all $g \in G$, $f \in \Cinfty(M)[[\lambda]]$, and $\phi
        \in \Cinfty(C)[[\lambda]]$. Moreover, we have for all $\xi \in
        \lie{g}$
        \begin{equation}
            \label{eq:Jxibulletphi}
            J_\xi \bulletk \phi 
            = 
            - \I \lambda \Lie_{\xi_C} \phi 
            - \I \lambda \kappa \Delta(\xi) \phi.
        \end{equation}
    \end{compactenum}
\end{proposition}
\begin{proof}
    For the reader's convenience we sketch the proof, see also
    \cite{bordemann.herbig.waldmann:2000a, bordemann:2005a}. Recall
    that we assume $M = M_\nice$.  Since
    $\deform{\koszul}^{(\kappa)}_1$ is left $\star$-linear its image
    is a left ideal.  Then the well-definedness of
    \eqref{eq:deformIotaIso} and \eqref{eq:ClassOfProlIso} follows
    from \eqref{eq:HomotopyForKNull} and
    \eqref{eq:DeformKoszulAndHomotopyNull}. It is clear that they are
    mutually inverse to each other. The canonical module structure of
    the left hand side of \eqref{eq:deformIotaIso} transported to
    $\Cinfty(C)[[\lambda]]$ via \eqref{eq:deformIotaIso} and
    \eqref{eq:ClassOfProlIso} gives \eqref{eq:fbulletphi}. This shows
    the second part. The third follows from
    Lemma~\ref{lemma:RestrictionLocal} since $\star$ is
    bidifferential, too.  The $G$-invariance is clear as $\star$,
    $\prol$, and $\deformiotak$ are $G$-invariant. The last part is a
    straightforward computation using the strong invariance
    \eqref{eq:StronglyInvariant}. We have
    \begin{align*}
        J_\xi \bulletk \phi
        &=
        \deformiotak
        \left(
            J_\xi \star \prol(\phi) - \prol(\phi) \star J_\xi +
            \prol(\phi) \star J_\xi
        \right) \\
        &=
        \deformiotak\left(
            \I \lambda \{J_\xi, \prol(\phi)\}
            +
            \deform{\koszul}^{(\kappa)}_1
            \left(\prol(\phi) \otimes \xi\right)
            -
            \I\lambda\kappa \ins(\Delta)
            \left(\prol(\phi) \otimes \xi\right)
        \right) \\
        &=
        -
        \I \lambda \deformiotak \Lie_{\xi_M} \prol(\phi)
        -
        \I\lambda\kappa \deformiotak 
        \left(\Delta(\xi) \prol(\phi)\right)
        =
        -
        \I \lambda \Lie_{\xi_C} \phi
        -
        \I \lambda \kappa \Delta(\xi) \phi,
    \end{align*}
    using the invariance of $\deformiotak$ and
    $\deformiotak \prol(\phi) = \phi$.
\end{proof}
\begin{remark}
    \label{remark:MniceM}
    Thanks to the locality features of $\deformiotak$ and $\bulletk$
    we see that only $M_\nice \subseteq M$ enters the game. Thus this
    justifies our previous simplification in
    Remark~\ref{remark:ComplexOverMnice} to consider $M_\nice$ only
    and assume $M_\nice = M$ from the beginning.
\end{remark}
\begin{remark}
    \label{remark:DeformIotaLeftLinear}
    Since $\deform{\koszul}_\kappa$ is left $\star$-linear, it follows
    from \eqref{eq:HomotopyForKNull} that $\deformiotak$ is left
    $\star$-linear, i.e a module homomorphism. This way, the deformed
    Koszul complex becomes indeed a (free) resolution of the deformed
    module $(\Cinfty(C)[[\lambda]], \bulletk)$, see also the proof of
    Theorem~\ref{theorem:GNS}.
\end{remark}
\begin{remark}
    \label{remark:KappaNull}
    From \eqref{eq:Jxibulletphi} we see that $\kappa = 0$ would also
    be a preferred choice. Note that all choices of $\kappa$ are
    compatible with the representation property
    \begin{equation}
        \label{eq:JbulletIsRepresentation}
        J_\xi \bulletk J_\eta \bulletk \phi
        -
        J_\eta \bulletk J_\xi \bulletk \phi
        =
        \I\lambda J_{[\xi, \eta]} \bulletk \phi
    \end{equation}
    since $\Delta(\xi)$ is a constant. Of course,
    \eqref{eq:JbulletIsRepresentation} is also clear from
    \eqref{eq:StargCovariant} and $\bulletk$ being a left module
    structure.
\end{remark}

From our general considerations in Proposition~\ref{proposition:AmodJ}
we know already how to compute the module endomorphisms of the
deformed module $(\Cinfty(C)[[\lambda]], \bulletk)$. The next
proposition gives now an explicit description of the quotient
$\deform{\mathcal{B}}_C \big/ \deform{\mathcal{J}}_C$ where
\begin{equation}
    \label{eq:DeformedBCDef}
    \deform{\mathcal{B}}_C
    =
    \left\{
        f \in \Cinfty(M)[[\lambda]]
        \; \big| \;
        [f, \deform{\mathcal{J}}_C]_\star 
        \subseteq \deform{\mathcal{J}}_C
    \right\}
\end{equation}
according to \eqref{eq:BDef}. This way, we also obtain the explicit
form of the bimodule structure, see \cite[Thm.~29 \&
32]{bordemann.herbig.waldmann:2000a}:
\begin{proposition}
    \label{proposition:ReducedAlgebraBimodule}
    Let $f \in \Cinfty(M)[[\lambda]]$, $\phi \in
    \Cinfty(C)[[\lambda]]$, and $u, v \in \Cinfty(M_\red)[[\lambda]]$.
    \begin{compactenum}
    \item \label{item:fInBC} We have $f \in \deform{\mathcal{B}}_C$
        iff $\Lie_{\xi_C} \deformiotak f = 0$ for all $\xi
        \in \lie{g}$ iff $\deformiotak f \in \pi^*
        \Cinfty(M_\red)[[\lambda]]$.
    \item \label{item:BcmodJCisMred} The quotient algebra
        $\deform{\mathcal{B}}_C \big/ \deform{\mathcal{J}}_C$ is
        isomorphic to $\Cinfty(M_\red)[[\lambda]]$ via the mutually
        inverse maps
        \begin{equation}
            \label{eq:BCJCtoMred}
            \deform{\mathcal{B}}_C \big/ \deform{\mathcal{J}}_C
            \ni [f] \mapsto \deformiotak f \in
            \pi^*\Cinfty(M_\red)[[\lambda]]
        \end{equation}
        and
        \begin{equation}
            \label{eq:MredToBCJC}
            \Cinfty(M_\red)[[\lambda]]
            \ni u \mapsto [\prol(\pi^*u)] \in
            \deform{\mathcal{B}}_C \big/ \deform{\mathcal{J}}_C.
        \end{equation}
    \item \label{item:ReducedStarProduct} The induced associative
        product $\starredk$ on $\Cinfty(M_\red)[[\lambda]]$ from
        $\deform{\mathcal{B}}_C \big/ \deform{\mathcal{J}}_C$ is
        explicitly given by
        \begin{equation}
            \label{eq:ReducedStarProduct}
            \pi^*(u \starredk v)
            = 
            \deformiotak \left(
                \prol(\pi^* u) \star \prol(\pi^* v)
            \right).
        \end{equation}
        This is a bidifferential star product quantizing the Poisson
        bracket \eqref{eq:ReducedPoissonBracket}.
    \item \label{item:RightModuleStructure} The induced right
        $(\Cinfty(M_\red)[[\lambda]], \starredk)$-module
        structure $\bulletredk$ on $\Cinfty(C)[[\lambda]]$
        from \eqref{eq:BmodJisEndAmodJ} is bi\-differential and
        explicitly given by
        \begin{equation}
            \label{eq:ReducedRightModule}
            \phi \bulletredk u 
            =
            \deformiotak \left(
                \prol(\phi) \star \prol(\pi^* u)
            \right).
        \end{equation}
    \item \label{item:RightModuleGInvariant} The right module
        structure is $G$-invariant, i.e. for $g \in G$ we have
        \begin{equation}
            \label{eq:RightModuleGInvariant}
            \mathsf{L}^*_g (\phi \bulletredk u)
            =
            (\mathsf{L}^*_g \phi) \bulletredk u.
        \end{equation}
    \item \label{item:EinsBulletu} We have $1 \bulletredk u
        = \pi^*u$.
    \end{compactenum}
\end{proposition}
\begin{proof}
    Again, we sketch the proof. For the first part note that
    $\deform{\mathcal{J}}_C = \ker \deformiotak$ according
    to \eqref{eq:HomotopyForKNull}. Now let $g = g^a \star J_a +
    \I\lambda\kappa C_{ba}^a g^b$ with $g^a \in \Cinfty(M)[[\lambda]]$
    be in the image of $\deform{\koszul}^{(\kappa)}_1$. For $f \in
    \Cinfty(M)[[\lambda]]$ we have by a straightforward computation
    \[
    [f, g]_\star 
    = \deform{\koszul}^{(\kappa)}_1 h 
    + \I \lambda g^a \star \Lie_{(e_a)_M} f
    \]
    with some $h \in \Cinfty(M, \lie{g})$ using the strong
    invariance~\eqref{eq:StronglyInvariant} of $\star$. Thus $[f,
    g]_\star$ is in $\deform{\mathcal{J}}_C$ iff $g^a \star
    \Lie_{(e_a)_M} f$ is in the image of
    $\deform{\koszul}^{(\kappa)}_1$ for all $g^a$. This shows that $f
    \in \deform{\mathcal{B}}_C$ iff $\Lie_{\xi_M} f \in \image
    \deform{\koszul}^{(\kappa)}_1 = \ker \deformiotak$.
    Since $\deformiotak$ is $G$-invariant the first part
    follows.  The second part is then clear from the first part and
    \eqref{eq:ReducedStarProduct} is a straightforward translation
    using the isomorphisms \eqref{eq:BCJCtoMred} and
    \eqref{eq:MredToBCJC}. From \eqref{eq:ReducedStarProduct} and
    Lemma~\ref{lemma:RestrictionLocal} it follows that
    $\starredk$ is bidifferential. The first orders of
    $\starredk$  are easily computed showing that it is
    indeed a star product on $M_\red$. The fourth part is clear from
    the second, the next part follows from the $G$-invariance of all
    involved maps.  The last part is clear since $\prol(1) = 1$.
\end{proof}
\begin{remark}
    \label{remark:Endomorphism}
    From the general Proposition~\ref{proposition:AmodJ} we know that
    the right multiplications by functions $u \in
    \Cinfty(M_\red)[[\lambda]]$ via $\bulletredk$ constitute
    precisely the module endomorphisms with respect to the left
    $\bulletk$-multiplications. The converse is not true: though
    the map
    \begin{equation}
        \label{eq:MtoEndoMred}
        (\Cinfty(M)[[\lambda]], \star)
        \ni f \mapsto (\phi \mapsto f \bulletk \phi) \in
        \End_{(\Cinfty(M_\red)[[\lambda]], \starredk)} 
        \left(\Cinfty(C)[[\lambda]], \bulletredk\right)
    \end{equation}
    is a homomorphism of algebras, it is neither injective nor
    surjective: By the locality \eqref{eq:fbulletphiBidifferential} it
    is clear that $f \bulletk \phi = 0$ for all $\phi$ if all
    derivatives of $f$ vanish on $C$. In particular, we have $f
    \bulletk \phi = 0$ for $\supp f \cap C = \emptyset$. Also, the map
    $\phi \mapsto \mathsf{L}^*_g \phi$ for $g \in G$ is in the module
    endomorphisms with respect to the right $\starredk$-module
    structure by \eqref{eq:RightModuleGInvariant}. Being a
    \emph{non-local} operation (unless $g$ acts trivially on $C$) we
    conclude that it can not be of the form $\phi \mapsto f \bulletk
    \phi$.
\end{remark}

%
%

\section{$^*$-Involutions by reduction}
\label{sec:AlternativeInvolutionsReduction}

In this section we discuss how a $^*$-involution for $\starred$ can be
constructed. To this end we assume that $\star$ is a \emph{Hermitian}
star product on $M$, i.e. we have
\begin{equation}
    \label{eq:starHermitian}
    \cc{f \star g} = \cc{g} \star \cc{f}
\end{equation}
for all $f, g \in \Cinfty(M)[[\lambda]]$. The existence of such
Hermitian star products is well-understood, see e.g.
\cite{neumaier:2002a} for the symplectic case.

The question we would like to address is whether and how one can
obtain a star product $\starred$ for which the complex conjugation or
a suitable deformation is a $^*$-involution. In principle, there is a
rather cheap answer: one has to compute a certain characteristic class
of $\starred$, apply the results of \cite{neumaier:2002a}, and
conclude that there is an equivalent star product to $\starred$ which
is Hermitian. However, we want a construction coming from the
reduction process itself and hence from $M$ instead of the above more
intrinsic argument. From a more conceptual point of view this is very
much desirable as ultimately one wants to apply reduction procedures
also to situations where nice differential geometry for $M_\red$, and
hence the results of e.g.  \cite{neumaier:2002a}, may fail due to
singularities.

Now this approach makes things more tricky: according to our reduction
philosophy we start with a left ideal $\mathcal{J} \subseteq
\mathcal{A}$ in some algebra and take $\mathcal{B} \big/ \mathcal{J}$
as the reduced algebra. If now $\mathcal{A}$ is in addition a
$^*$-algebra we have to construct a $^*$-involution for $\mathcal{B}
\big/ \mathcal{J}$. From all relevant examples in deformation
quantization one knows that $\mathcal{J}$ is \emph{only} a left ideal.
Thus $\mathcal{J}$ can not be a $^*$-ideal and thus $\mathcal{B}$ can
not be a $^*$-subalgebra. Consequently, there is \emph{no obvious} way
to define a $^*$-involution on the quotient $\mathcal{B} \big/
\mathcal{J}$. In fact, some additional ingredients will be needed.

%
%

\subsection{Algebraic preliminaries}
\label{subsec:AlgebraicPreliminaries}

The main idea to construct the $^*$-involution is to use a
representation of the reduced algebra as adjointable operators acting
on a pre-Hilbert space over $\mathbb{C}[[\lambda]]$. Since the reduced
algebra $\mathcal{B} \big/\mathcal{J}$ can be identified  to the algebra $
\End_{\mathcal{A}}(\mathcal{A} \big/ \mathcal{J})^\opp$ (i.e. with the
opposite algebra structure), a first idea
is to build a structure of pre-Hilbert space on $\mathcal{A} \big/
\mathcal{J}$. To this aim, one considers an additional positive linear
functional on $\mathcal{A}$. To put things into the correct algebraic
framework we consider the following situation, see e.g.
\cite[Chap.~7]{waldmann:2007a} for more details and further
references. Let $\ring{R}$ be an ordered ring and $\ring{C} =
\ring{R}(\I)$ its extension by a square root $\I$ of $-1$.  The
relevant examples for us are $\ring{R} = \mathbb{R}$ and $\ring{R} =
\mathbb{R}[[\lambda]]$ with $\ring{C} = \mathbb{C}$ and $\ring{C} =
\mathbb{C}[[\lambda]]$, respectively. Recall that a formal series in
$\mathbb{R}[[\lambda]]$ is called positive if the lowest non-vanishing
order is positive. Then let $\mathcal{A}$ be a $^*$-algebra over
$\ring{C}$, i.e. an associative algebra equipped with a
$\ring{C}$-antilinear involutive anti-automorphism, the
$^*$-involution $^*: \mathcal{A} \longrightarrow \mathcal{A}$. The
ring ordering allows now the following definition: a $\ring{C}$-linear
functional $\omega: \mathcal{A} \longrightarrow \ring{C}$ is called
positive if for all $a \in \mathcal{A}$
\begin{equation}
    \label{eq:omegaPositive}
    \omega(a^*a) \ge 0.
\end{equation}
In this case we have a Cauchy-Schwarz inequality $\omega(a^*b)
\cc{\omega(a^*b)} \le \omega(a^*a) \omega(b^*b)$ and the reality
$\omega(a^*b) = \cc{\omega(b^*a)}$ as usual. It follows that
\begin{equation}
    \label{eq:GelfandIdeal}
    \mathcal{J}_\omega
    = 
    \left\{
        a \in \mathcal{A}
        \; \big| \; 
        \omega(a^*a) = 0
    \right\}
\end{equation}
is a left ideal in $\mathcal{A}$, the \emph{Gel'fand ideal} of
$\omega$.

Recall that a \emph{pre Hilbert space} $\mathcal{H}$ over $\ring{C}$
is a $\ring{C}$-module equipped with a scalar product $\SP{\cdot,
  \cdot}: \mathcal{H} \times \mathcal{H} \longrightarrow \ring{C}$
which is $\ring{C}$-linear in the second argument and satisfies
$\cc{\SP{\phi, \psi}} = \SP{\psi, \phi}$ as well as $\SP{\phi, \phi} >
0$ for $\phi \ne 0$. By $\Bounded(\mathcal{H})$ we denote the
\emph{adjointable operators} on $\mathcal{H}$, i.e. those maps $A:
\mathcal{H} \longrightarrow \mathcal{H}$ for which there is a map
$A^*: \mathcal{H} \longrightarrow \mathcal{H}$ with $\SP{\phi, A\psi}
= \SP{A^*\phi, \psi}$ for all $\phi, \psi \in \mathcal{H}$. It follows
that such a map is $\ring{C}$-linear and $\Bounded(\mathcal{H})$
becomes a unital $^*$-algebra over $\ring{C}$.  This allows to define
a \emph{$^*$-representation} of $\mathcal{A}$ on $\mathcal{H}$ to be a
$^*$-homomorphism $\pi: \mathcal{A} \longrightarrow
\Bounded(\mathcal{H})$. For these notions and further references on
$^*$-representation theory we refer to \cite[Chap.~7]{waldmann:2007a}
as well as to \cite{schmuedgen:1990a} for the more particular case of
$^*$-algebras and $O^*$-algebras over $\mathbb{C}$.

Having a positive linear functional $\omega: \mathcal{A}
\longrightarrow \ring{C}$ one constructs a $^*$-representation
$(\mathcal{H}_\omega, \pi_\omega)$, the \emph{GNS representation}, of
$\mathcal{A}$ as follows: setting $\mathcal{H}_\omega = \mathcal{A}
\big/ \mathcal{J}_\omega$ yields on one hand a left
$\mathcal{A}$-module and on the other hand a pre Hilbert space via
$\SP{\psi_a, \psi_b} = \omega(a^*b)$ where $\psi_a, \psi_b \in
\mathcal{H}_\omega$ denote the equivalence classes of $a, b \in
\mathcal{A}$. Then one checks immediately that the canonical left
module structure, denoted by $\pi_\omega(a) \psi_b = \psi_{ab}$ in
this context, is a $^*$-representation of $\mathcal{A}$ on
$\mathcal{H}_\omega$.

Back to our reduction problem, the main idea is now to look for a
positive linear functional $\omega$ such that the left ideal
$\mathcal{J}$ we use for reduction coincides with the Gel'fand ideal
$\mathcal{J}_\omega$. In this case we have the following simple
statement:
\begin{proposition}
    \label{proposition:ReducedInvolutionGeneral}
    Assume $\omega: \mathcal{A} \longrightarrow \ring{C}$ is a
    positive linear functional with $\mathcal{J}_\omega = \mathcal{J}$
    and hence $\mathcal{A} \big/ \mathcal{J} = \mathcal{H}_\omega$.
    Then $\End_{\mathcal{A}} (\mathcal{A} \big/ \mathcal{J}) \cap
    \Bounded(\mathcal{H}_\omega)$ is a $^*$-subalgebra of
    $\Bounded(\mathcal{H}_\omega)$ and a subalgebra of
    $\End_{\mathcal{A}}(\mathcal{A} \big/ \mathcal{J})$.
\end{proposition}
\begin{proof}
    In general, there may be left module endomorphisms which are not
    adjointable and adjointable endomorphisms which are not left
    $\mathcal{A}$-linear. The non-trivial part is to show that for $A
    \in \End_{\mathcal{A}}(\mathcal{A} \big/ \mathcal{J}) \cap
    \Bounded(\mathcal{H}_\omega)$ also $A^*$ is left
    $\mathcal{A}$-linear. Thus let $a \in \mathcal{A}$ and $\phi, \psi
    \in \mathcal{H}_\omega = \mathcal{A} \big/ \mathcal{J}$. Then
    \[
    \SP{A^* \pi_\omega(a)\phi, \psi}
    = \SP{\pi_\omega(a)\phi, A\psi}
    = \SP{\phi, \pi_\omega(a^*) A\psi}
    = \SP{\phi, A \pi_\omega(a^*)\psi}
    = \SP{\pi_\omega(a) A^*\phi, \psi},
    \]
    since $\pi_\omega(a)^* = \pi_\omega(a^*)$. Since $\SP{\cdot,
      \cdot}$ is non-degenerate we conclude that $A^*\pi_\omega(a) =
    \pi_\omega(a) A^*$.
\end{proof}

Thus from $\mathcal{B} \big/ \mathcal{J} \cong
\End_{\mathcal{A}}(\mathcal{A} \big/ \mathcal{J})^\opp$ we obtain at
least a subalgebra of $\mathcal{B} \big/ \mathcal{J}$ which is a
$^*$-algebra with $^*$-involution inherited from
$\Bounded(\mathcal{H}_\omega)$.

\begin{remark}
    \label{remark:InvolutionProgram}
    In conclusion, we want a positive functional $\omega: \mathcal{A}
    \longrightarrow \ring{C}$ such that first $\mathcal{J} =
    \mathcal{J}_\omega$ and second \emph{all} left
    $\mathcal{A}$-linear endomorphisms of $\mathcal{H}_\omega =
    \mathcal{A} \big/ \mathcal{J}$ are adjointable. In this case
    $\mathcal{B} \big/ \mathcal{J}$ becomes a $^*$-subalgebra of
    $\Bounded(\mathcal{H}_\omega)$ in a natural way.  Of course, up to
    now this is only an algebraic game as the existence of such a
    functional $\omega$ is by far not obvious.
\end{remark}

%
%

\subsection{The positive functional}
\label{subsec:PositiveFunctional}

While in the general algebraic situation not much can be said about
the existence of a suitable positive functional, in our geometric
context we can actually construct a fairly simple $\omega$.

To this end we investigate the behaviour of the operators introduced
in Section~\ref{sec:ClassicalConstruction} and
\ref{sec:QuantizedBimodule} under complex conjugation. On $\Cinfty(M,
\AntiC^\bullet \lie{g})$ we define the complex conjugation pointwise
in $M$ and require the elements of $\Anti_{\mathbb{R}}^\bullet
\lie{g}$ to be real. Recall that our construction of $\starred$ uses a
strongly invariant Hermitian star product $\star$ on $M$.
\begin{lemma}
    Let $x \in \Cinfty(M, \AntiC^\bullet \lie{g})[[\lambda]]$. Then
    \begin{equation}
        \label{eq:cchxkoszuliota}
        \cc{hx} = h \cc{x},
        \quad
        \cc{\koszul x} = \koszul \cc{x},
        \quad
        \textrm{and}
        \quad
        \cc{\iota^* x} = \iota^* \cc{x},
    \end{equation}
    and
    \begin{equation}
        \label{eq:ccDeformedKoszul}
        \cc{\deform{\koszul}^{(\kappa)} x}
        = \deform{\koszul}^{(\kappa)} \cc{x}
        - \I \lambda \Lie_{(e_a)_M} \ins(e^a) \cc{x}
        - \I \lambda C_{ab}^c e_c \wedge \ins(e^a) \ins(e^b) \cc{x}
        - \I\lambda (\cc{\kappa} + \kappa) \ins(\Delta) \cc{x}.
    \end{equation}
    Moreover, for $\phi \in \Cinfty(C)[[\lambda]]$ we have
    \begin{equation}
        \label{eq:ccProl}
        \cc{\prol(\phi)} = \prol(\cc{\phi}).
    \end{equation}
\end{lemma}
\begin{proof}
    The claims in \eqref{eq:cchxkoszuliota} and \eqref{eq:ccProl} are
    trivial. For $\kappa = 0$ the claim \eqref{eq:ccDeformedKoszul} is
    a simple computation using again the strong invariance
    \eqref{eq:StronglyInvariant} as well as \eqref{eq:starHermitian}.
    But then the case $\kappa \in \mathbb{C}[[\lambda]]$ follows since
    the one-form $\Delta$ is real.
\end{proof}

Before we proceed we have to rewrite the result
\eqref{eq:ccDeformedKoszul} in the following way.  Since $\star$ is
strongly invariant we have for the (left) action of $\lie{g}$ on
$\Cinfty(M, \AntiC^\bullet \lie{g})$
\begin{equation}
    \label{eq:LeftActiongOnKoszulComplex}
    \varrho(\xi) x 
    = - \Lie_{\xi_M} x + \ad_\xi x
    = \frac{1}{\I\lambda} [J_\xi, x]_\star + \ad_\xi x
\end{equation}
for all $x \in \Cinfty(M, \AntiC^\bullet \lie{g})$.  Using this, we
get from \eqref{eq:ccDeformedKoszul}
\begin{equation}
    \label{eq:ccDeformedKoszulII}
    \cc{\deform{\koszul}^{(\kappa)} x} 
    =
    \deform{\koszul}^{(\kappa)} \cc{x} 
    +
    \I\lambda \left(\varrho(e_a) - \ad_{e_a}\right) \ins(e^a) \cc{x}
    -
    \I\lambda C_{ab}^c e_c \wedge \ins(e^a) \ins(e^b) \cc{x}
    -
    \I\lambda(\cc{\kappa} + \kappa) \ins(\Delta) \cc{x}.
\end{equation}
While the Lie derivative $\Lie_{\xi_M}$ commutes with all the
insertions $\ins(\alpha)$ of \emph{constant} one-forms $\alpha \in
\lie{g}^*$ this is no longer true for $\varrho(\xi)$ and $\ad_\xi$. In
fact, by a simple computation we get
\begin{equation}
    \label{eq:adeainsea}
    \ad_{e_a} \ins(e^a) x - \ins(e^a) \ad_{e_a} x
    =
    \ins(\Delta) x
    =
    \varrho(e_a) \ins(e^a) x - \ins(e^a) \varrho(e_a) x
\end{equation}
for all $x \in \Cinfty(M, \AntiC^\bullet)[[\lambda]]$.

From \eqref{eq:ccDeformedKoszul} we already see that the behaviour of
$\deformiotak$ and $\deform{h}^{(\kappa)}$ is more complicated under
complex conjugation. For the relevant geometric series in
\eqref{eq:deformiota} we have the following result:
\begin{lemma}
    \label{lemma:ccFunnyGeometricSeries}
    Let $f \in \Cinfty(M)[[\lambda]]$. Then the operators
    \begin{equation}
        \label{eq:OperatorAa}
        A_\kappa^a(f)
        =
        \sum_{k=1}^\infty \sum_{\ell = 1}^k
        \left(
            \left(\koszul_1 - \deform{\koszul}^{(\kappa)}_1\right)
            h_0
        \right)^{k-\ell}
        \ins(e^a)
        h_0
        \cc{
          \left(
              \left(\koszul_1 - \deform{\koszul}^{(\kappa)}_1\right)
              h_0
          \right)^{\ell - 1}
          \cc{f}
        }
    \end{equation}
    and
    \begin{equation}
        \label{eq:TheCoolOperatorB}
        B_\kappa(f)
        =
        \sum_{k=1}^\infty \sum_{\ell = 1}^k
        \left(
            \left(\koszul_1 - \deform{\koszul}^{(\kappa)}_1\right)
            h_0
        \right)^{k-\ell}
        \ins(\Delta)
        h_0
        \cc{
          \left(
              \left(\koszul_1 - \deform{\koszul}^{(\kappa)}_1\right)
              h_0
          \right)^{\ell - 1}
          \cc{f}
        }
    \end{equation}
    yield well-defined $\mathbb{C}[[\lambda]]$-linear maps
    $A^a_\kappa, B_\kappa: \Cinfty(M)[[\lambda]] \longrightarrow
    \Cinfty(M)[[\lambda]]$ such that
    \begin{align}
        \label{eq:ccFunnyGeometricSeriesI}
        \cc{
          \left(
              \id + \left(
                  \deform{\koszul}^{(\kappa)}_1 - \koszul_1
              \right) h_0
          \right)^{-1} f
        }
        &=
        \left(
            \id + \left(
                \deform{\koszul}^{(\kappa)}_1 - \koszul_1
            \right) h_0
        \right)^{-1} \cc{f}
        + \I \lambda \Lie_{(e_a)_M} A^a (\cc{f})
        + \I\lambda (\cc{\kappa} + \kappa) B(\cc{f}) \\
        &=
        \left(
            \id + \left(
                \deform{\koszul}^{(\kappa)}_1 - \koszul_1
            \right) h_0
        \right)^{-1} \cc{f}
        + \I \lambda A^a \left(\Lie_{(e_a)_M} (\cc{f})\right)
        + \I\lambda (\cc{\kappa} + \kappa - 1) B(\cc{f})
        \label{eq:ccFunnyGeometricSeriesII}
    \end{align}
\end{lemma}
\begin{proof}
    Since the term with the two insertions does not contribute and
    since $\ad_{e_a}$ vanishes on functions we get from
    \eqref{eq:ccDeformedKoszulII}
    \[
    \cc{\left(\koszul_1 - \deform{\koszul}^{(\kappa)}_1\right) h_0 f}
    =
    \left(\koszul_1 - \deform{\koszul}^{(\kappa)}_1\right)
    h_0 \cc{f}
    -
    \I\lambda \varrho(e_a) \ins(e^a) h_0 \cc{f}
    +
    \I\lambda (\cc{\kappa} + \kappa) \ins(\Delta) h_0 \cc{f}.
    \tag{$*$}
    \]
    Now using \eqref{eq:adeainsea} we get the alternative version
    \[
    \cc{\left(\koszul_1 - \deform{\koszul}^{(\kappa)}_1\right) h_0 f}
    =
    \left(\koszul_1 - \deform{\koszul}^{(\kappa)}_1\right)
    h_0 \cc{f}
    -
    \I\lambda \ins(e^a) h_0 (\varrho(e_a) \cc{f})
    +
    \I\lambda (\cc{\kappa} + \kappa - 1) \ins(\Delta) h_0 \cc{f}
    \tag{$**$}
    \]
    for the commutation relation. It is this equation which motivates
    $\kappa = \frac{1}{2}$ instead of $\kappa = 0$. Applying the
    general commutation relation $C B^k = [C, B] B^{k-1} + B [C, B]
    B^{k-2} + \cdots + B^{k-1} [C, B] + B^k C$ to the complex
    conjugation and the map $\left(\koszul_1 -
        \deform{\koszul}^{(\kappa)}_1\right) h_0$ gives
    \begin{align*}
        &\cc{
          \left(
              \left(
                  \koszul_1 - \deform{\koszul}^{(\kappa)}_1
              \right) h_0
          \right)^k f
        }
        =
        \left(
            \left(
                \koszul_1 - \deform{\koszul}^{(\kappa)}_1
            \right) h_0
        \right)^k \cc{f} \\
        & \quad+ 
        \sum_{\ell = 1}^k
        \left(  
            \left(
                \koszul_1 - \deform{\koszul}^{(\kappa)}_1
            \right) h_0
        \right)^{k-\ell}
        \left(
            - \I\lambda \varrho(e_a) \ins(e^a) h_0 
            + \I\lambda(\cc{\kappa} + \kappa) \ins(\Delta) h_0
        \right)
        \cc{
          \left(  
              \left(
                  \koszul_1 - \deform{\koszul}^{(\kappa)}_1
              \right) h_0
          \right)^{\ell-1} f
        }.
    \end{align*}
    Using the operators $A^a_\kappa$ and $B_\kappa$, which are clearly
    well-defined as formal series since the difference $\koszul_1 -
    \deform{\koszul}^{(\kappa)}_1$ is at least of order $\lambda$, we
    get for the geometric series
    \[
    \cc{
      \left(
          \id + \left(
              \deform{\koszul}^{(\kappa)}_1 - \koszul_1
          \right) h_0
      \right)^{-1} f
    }
    =
    \left(
        \id + \left(
            \deform{\koszul}^{(\kappa)}_1 - \koszul_1
        \right) h_0
    \right)^{-1} \cc{f}
    -
    \I\lambda \varrho(e_a) A^a_\kappa(\cc{f})
    +
    \I\lambda (\cc{\kappa} + \kappa) B_\kappa (\cc{f}),
    \]
    since the action of $e_a$ can be commuted to the \emph{left} as
    all operators are $G$-invariant. This proves the first equation
    \eqref{eq:ccFunnyGeometricSeriesI} since $A^a_\kappa(\cc{f})$ is a
    function whence the left action of $e_a$ is just $- \Lie_{(e_a)_M}$.
    Now conversely, using the second version ($**$) we can commute
    the action of $e_a$ to the right, as now only invariant operators
    remain. This gives \eqref{eq:ccFunnyGeometricSeriesII}.
\end{proof}
\begin{corollary}
    \label{corollary:ccDeformedIota}
    Let $f \in \Cinfty(M)[[\lambda]]$. Then we have
    \begin{align}
        \label{eq:ccDeformedIotaI}
        \cc{\deformiotak f}
        &=
        \deformiotak \cc{f}
        +
        \I\lambda \Lie_{(e_a)_C} \iota^* A^a_\kappa(\cc{f})
        +
        \I\lambda (\cc{\kappa} + \kappa) \iota^* B_\kappa (\cc{f}) \\
        &=
        \deformiotak \cc{f}
        +
        \I\lambda \iota^* A^a_\kappa
        \left(\Lie_{(e_a)_M} \cc{f}\right)
        +
        \I\lambda (\cc{\kappa} + \kappa - 1) \iota^* B_\kappa (\cc{f}).
        \label{eq:ccDeformedIotaII}
    \end{align}
\end{corollary}
Analogously, we obtain the behaviour of the homotopy
$\deform{h}^{(\kappa)}$ under complex conjugation. It is clear from
\eqref{eq:ccDeformedIotaII} that the value
\begin{equation}
    \label{eq:kappaIsFracEinsZwei}
    \kappa = \frac{1}{2}
\end{equation}
will simplify things drastically as in this case the presence of the
operator $B$ is absent in \eqref{eq:ccDeformedIotaII}. Thus from now
on we will exclusively consider $\kappa = \frac{1}{2}$ and omit the
subscript $\kappa$ at all relevant places. A first consequence of this
choice is the following:
\begin{corollary}
    \label{corollary:ccIotaOnGInvariantFunctions}
    If $f \in \Cinfty(M)[[\lambda]]$ is $G$-invariant then
    \begin{equation}
        \label{eq:ccIotaInvariantf}
        \cc{\deform{\iota^*} f} = \deform{\iota^*} \cc{f}.
    \end{equation}
\end{corollary}

Even though it is not in the main line of our argument according to
the previous section, we can now directly prove that $\starred$ is
Hermitian:
\begin{proposition}
    \label{proposition:StarredIsHermitian}
    The star product $\starred$ is Hermitian.
\end{proposition}
\begin{proof}
    Let $u, v \in \Cinfty(M_\red)[[\lambda]]$. Then we have
    \begin{align*}
        \cc{\pi^*(u \starred v)}
        &=
        \cc{\deform{\iota^*} \left(
              \prol(\pi^*u) \star \prol(\pi^*v)
          \right)}
        =
        \deform{\iota^*} \left(
            \cc{\prol(\pi^*u) \star \prol(\pi^*v)}
        \right)
        =
        \deform{\iota^*} \left(
            \prol(\pi^* \cc{v}) \star \prol(\pi^* \cc{u})
        \right) \\
        &=
        \pi^*(\cc{v} \starred \cc{u}),
    \end{align*}
    since by the $G$-invariance of $\prol$ and $\star$ as in
    \eqref{eq:StarGInvariant} we know that $\prol(\pi^*u) \star
    \prol(\pi^* v)$ is $G$-invariant and thus
    Corollary~\ref{corollary:ccIotaOnGInvariantFunctions} applies.
\end{proof}
\begin{remark}
    \label{remark:StarredIsHermitian}
    From the proof we see that one needs the complex conjugation of
    the functions $\Cinfty(C)[[\lambda]]$. In a purely algebraic
    setting as in Section~\ref{subsec:AlgebraicPreliminaries} this
    would mean to have a $^*$-involution on the \emph{module}
    $\mathcal{A} \big/ \mathcal{J}$, which is clearly a non-canonical
    extra structure.  Thus also the above seemingly canonical proof
    that $\starred$ is Hermitian is not that conceptual from the point
    of view of our considerations in
    Section~\ref{subsec:AlgebraicPreliminaries}.
\end{remark}

We will now come back to the construction of the positive functional.
First we choose a formal series of densities $\mu =
\sum_{r=0}^\infty \lambda^r \mu_r \in
\Gamma^\infty(|\Anti^{\mathrm{top}}| T^*C)[[\lambda]]$ on $C$ such
that $\cc{\mu} = \mu$ is real and $\mu_0 > 0$ is everywhere positive.
Moreover, we require that $\mu$ transforms under the $G$ action as
follows
\begin{equation}
    \label{eq:muModularWeight}
    \mathsf{L}^*_{g^{-1}} \mu = \frac{1}{\Delta(g)} \mu,
\end{equation}
where $\Delta: G \longrightarrow \mathbb{R}^+$ is the modular function
of $G$ and we take the \emph{left} action of $G$ on densities. Recall
that $\Delta$ is the Lie group homomorphism obtained from
exponentiating the Lie algebra homomorphism $\Delta (\xi) = \tr
ad_\xi$, thereby motivating our notation. For the (well-known)
existence of densities with \eqref{eq:muModularWeight} see
Appendix~\ref{sec:DensitiesPrincipalBundles}.
\begin{definition}
    \label{definition:ScalarProductmu}
    For $\phi, \psi \in \Cinfty_0(C)[[\lambda]]$ we define
    \begin{equation}
        \label{eq:ScalarProductMu}
        \SP{\phi, \psi}_\mu
        = \int_C \deform{\iota^*} \left(
            \cc{\prol(\phi)} \star \prol(\psi)
        \right) \mu.
    \end{equation}
\end{definition}
\begin{lemma}
    \label{lemma:ScalarProductTechnicalities}
    Let $\phi, \psi \in \Cinfty_0(C)[[\lambda]]$ and $f \in
    \Cinfty_0(M)[[\lambda]]$. Then $\SP{\phi, \psi}_\mu$ is
    well-defined and we have
    \begin{equation}
        \label{eq:SPmuAlternative}
        \SP{\phi, \psi}_\mu 
        = \int_C \left(\cc{\prol(\phi)} \bullet \psi\right) \mu
    \end{equation}
    and
    \begin{equation}
        \label{eq:ccintCf}
        \cc{\int_C \deform{\iota^*}(f) \mu}
        =
        \int_C \deform{\iota^*}\left(\cc{f}\right) \mu.
    \end{equation}
\end{lemma}
\begin{proof}
    Even though the prolongation is not a local operator the supports
    are only changed in transverse directions to $C$. It follows that
    the support of the integrand in \eqref{eq:SPmuAlternative} as well
    as in \eqref{eq:ccintCf} is compact in every order of $\lambda$.
    Then the first part is clear from \eqref{eq:fbulletphi}. 
    For the second we compute
    \begin{align*}
        \int_C \cc{\deform{\iota^*} f} \mu
        &=
        \int_C \deform{\iota^*} \cc{f}
        +
        \I\lambda \int_C \Lie_{(e_a)_C} \iota^* A^a(\cc{f}) \mu
        +
        \I\lambda \int_C \iota^*B(\cc{f}) \mu \\
        &=
        \int_C \deform{\iota^*} \cc{f}
        -
        \I\lambda \int_C \iota^* A^a(\cc{f}) \Lie_{(e_a)_C} \mu
        +
        \I\lambda \int_C \iota^* B(\cc{f}) \mu \\
        &=
        \int_C \deform{\iota^*} \cc{f}
        -
        \Delta(e_a)
        \I\lambda \int_C \iota^* A^a(\cc{f}) \mu
        +
        \I\lambda \int_C \iota^* B(\cc{f}) \mu.
    \end{align*}
    Here we used $\Lie_{\xi_C} \mu = \Delta(\xi) \mu$ which follows
    from differentiating \eqref{eq:muModularWeight}. Finally, a close
    look at the definitions of $A^a$ and $B$ shows that $\Delta(e_a)
    A^a(f) = B(f)$. This shows the second part.
\end{proof}
\begin{proposition}
    \label{proposition:SPmuIsScalarProduct}
    The scalar product $\SP{\cdot, \cdot}_\mu$ makes
    $\Cinfty_0(C)[[\lambda]]$ a pre Hilbert space over
    $\mathbb{C}[[\lambda]]$ and the left module structure $\bullet$
    becomes a $^*$-representation of $(\Cinfty(M)[[\lambda]], \star)$
    on it.
\end{proposition}
\begin{proof}
    The $\mathbb{C}[[\lambda]]$-linearity of $\SP{\cdot, \cdot}_\mu$
    in the second argument is clear. For the symmetry we compute
    \begin{align*}
        \cc{\SP{\phi, \psi}_\mu}
        &=
        \cc{
          \int_C \deform{\iota^*} \left(
              \cc{\prol(\phi)} \star \prol(\psi)
          \right) \mu
        }
        \stackrel{\eqref{eq:ccintCf}}{=}
        \int_C \deform{\iota^*} \left(
            \cc{
              \cc{\prol(\phi)} \star \prol(\psi)
            }
        \right) \mu
        =
        \int_C \deform{\iota^*} \left(
            \cc{\prol(\psi)} \star \prol(\phi)
        \right) \mu \\
        &=
        \SP{\psi, \phi}_\mu,
    \end{align*}
    using that $\star$ is Hermitian. Finally, for $\phi = \sum_{r=r_0}
    \lambda^r \phi_r$ with $\phi_{r_0} \ne 0$ the lowest non-vanishing
    term in $\SP{\phi, \phi}_\mu$ is simply given by the
    $\mu_0$-integral of $\cc{\phi_{r_0}}\phi_{r_0}$ over $C$ which
    gives a positive result since $\mu_0$ is a \emph{positive}
    density. Thus $\Cinfty_0(C)[[\lambda]]$ becomes a pre Hilbert
    space indeed. Now for $f \in \Cinfty(M)[[\lambda]]$ we compute
    \begin{align*}
        \SP{\phi, f \bullet \psi}_\mu
        &\stackrel{\mathclap{\eqref{eq:SPmuAlternative}}}{=} \;
        \int_C \left(
            \cc{\prol(\phi)} \bullet (f \bullet \psi)
        \right) \mu \\
        &=
        \int_C \left(
            (\cc{\prol(\phi)} \star f) \bullet \psi
        \right) \mu \\
        &=
        \int_C \left(
            \cc{(\cc{f} \star \prol(\phi))} \bullet \psi
        \right) \mu \\
        &=
        \int_C \deform{\iota^*}
        \left(
            \cc{(\cc{f} \star \prol(\phi))} \star \prol(\psi)
        \right) \mu \\
        &\stackrel{\mathclap{\eqref{eq:ccintCf}}}{=} \;
        \cc{
          \int_C \deform{\iota^*}
          \left(
              \cc{\prol(\psi)} \star (\cc{f} \star \prol(\phi))
          \right) \mu
        } \\
        &=
        \cc{
          \int_C \deform{\iota^*}
          \left(
              (\cc{\prol(\psi)} \star \cc{f}) \star \prol(\phi)
          \right) \mu
        } \\
        &=
        \cc{
          \int_C
          \left(
              (\cc{\prol(\psi)} \star \cc{f}) \bullet \phi
          \right) \mu
        } \\
        &=
        \cc{
          \int_C
          \left(
              \cc{\prol(\psi)} \bullet (\cc{f} \bullet \phi)
          \right) \mu
        } \\
        &\stackrel{\mathclap{\eqref{eq:SPmuAlternative}}}{=} \;
        \cc{\SP{\psi, \cc{f} \bullet \phi}_\mu},
    \end{align*}
    using the fact that $\bullet$ is a left $\star$-module structure.
    It follows that we have a $^*$-representation.
\end{proof}

We will now show that this $^*$-representation is actually the GNS
representation of the following positive linear functional:
\begin{definition}
    \label{definition:ThePositiveFunctional}
    For $f \in \Cinfty_0(M)[[\lambda]]$ we define the
    $\mathbb{C}[[\lambda]]$-linear functional $\omega_\mu$ by
    \begin{equation}
        \label{eq:omegamuDef}
        \omega_\mu (f) = \int_C \deform{\iota^*}(f) \mu.
    \end{equation}
\end{definition}
First note that $\deform{\iota^*} f \in \Cinfty_0(C)[[\lambda]]$ by
Lemma~\ref{lemma:RestrictionLocal} whence $\omega_\mu$ is
well-defined. Since neither $C$ nor $M$ need to be compact we are
dealing with a $^*$-algebra $\Cinfty_0(M)[[\lambda]]$ \emph{without}
unit in general. The following lemma would be much easier if $1 \in
\Cinfty_0(C)$:
\begin{lemma}
    \label{lemma:PositiveFunctionalAndGelfandIdeal}
    The $\mathbb{C}[[\lambda]]$-linear functional $\omega_\mu$ is
    positive and its Gel'fand ideal is
    \begin{equation}
        \label{eq:TheGelfandIdeal}
        \mathcal{J}_{\omega_\mu} =
        \left\{
            f \in \Cinfty_0(M)[[\lambda]]
            \; \big| \;
            \deform{\iota^*} f = 0
        \right\}.
    \end{equation}
\end{lemma}
\begin{proof}
    In order to show the positivity $\omega_\mu(\cc{f} \star f) \ge 0$
    it is sufficient to consider $f \in \Cinfty_0(M)$ without higher
    $\lambda$-orders thanks to \cite[Prop.~7.1.51]{waldmann:2007a}.
    For $f \in \Cinfty_0(M)$ we choose a $\chi = \cc{\chi} \in
    \Cinfty_0(C)$ such that $\prol(\chi)\at{\supp f} = 1$ which is
    clearly possible. By the locality of $\star$ we get $f \star
    \prol(\chi) = f = \prol(\chi) \star f$ and hence also $f \bullet
    \chi = \deform{\iota^*} f$. Thus
    \begin{align*}
        \SP{\chi, f \bullet \chi}_\mu
        &=
        \int_C 
        \left(\cc{\prol(\chi)} \bullet (f \bullet \chi)\right)
        \mu
        =
        \int_C 
        \left((\prol(\chi) \star f) \bullet \chi\right) \mu
        =
        \int_C (f \bullet \chi) \mu
        =
        \int_C \deform{\iota^*}(f) \mu \\
        &=
        \omega_\mu(f).
    \end{align*}
    Since $\supp(\cc{f} \star f) \subseteq \supp f$ the same applies
    for $\cc{f} \star f$ and we have
    \[
    \omega_\mu(\cc{f} \star f) 
    = \SP{\chi, (\cc{f} \star f) \bullet \chi}_\mu
    = \SP{f \bullet \chi, f \bullet \chi}_\mu
    = \SP{\deform{\iota^*} f, \deform{\iota^*} f}_\mu
    \ge 0
    \]
    by the positivity of $\SP{\cdot, \cdot}_\mu$. Finally,
    $\omega_\mu(\cc{f} \star f) = 0$ iff $\deform{\iota^*} f = 0$ is
    clear from this computation.
\end{proof}

This lemma allows to identify the GNS representation induced by
$\omega_\mu$ easily. Recall that the GNS representation automatically
extends to the whole algebra since $\Cinfty_0(M)[[\lambda]]$ is a
$^*$-ideal, see e.g. \cite[Lem.~7.2.18]{waldmann:2007a}:
\begin{theorem}
    \label{theorem:GNS}
    The GNS representation of $\Cinfty(M)[[\lambda]]$ on
    $\mathcal{H}_{\omega_\mu} = \Cinfty_0(M)[[\lambda]] \big/
    \mathcal{J}_{\omega_\mu}$ is unitarily equivalent to the
    $^*$-representation $\bullet$ on $\Cinfty_0(C)[[\lambda]]$ where
    the inner product is $\SP{\cdot, \cdot}_\mu$. The unitary
    intertwiner is explicitly given by
    \begin{equation}
        \label{eq:UnitaryIntertwiner}
        \Cinfty_0(M)[[\lambda]] \big/ \mathcal{J}_{\omega_\mu}
        \ni \psi_f \mapsto \deform{\iota^*} f \in
        \Cinfty_0(C)[[\lambda]].
    \end{equation}
\end{theorem}
\begin{proof}
    By \eqref{eq:TheGelfandIdeal} it follows that
    \eqref{eq:UnitaryIntertwiner} is well-defined and injective. Now
    let $\chi \in \Cinfty(M)$ be a function such that $\chi$ is equal
    to one in an open neighbourhood of $C$ but has compact support in
    directions of the fibers of the tubular neighbourhood of $C$.
    Clearly, such a function exists (and can even be chosen to be
    $G$-invariant thanks to the properness of the action). It follows
    that $\chi \prol(\phi) \in \Cinfty_0(M)[[\lambda]]$ for $\phi \in
    \Cinfty_0(C)[[\lambda]]$. Moreover, the locality of
    $\deform{\iota^*}$ according to Lemma~\ref{lemma:RestrictionLocal}
    shows that $\deform{\iota^*}(\chi \prol(\phi)) = \phi$ proving the
    surjectivity of \eqref{eq:UnitaryIntertwiner}. Finally, let $f, g
    \in \Cinfty_0(M)$ and chose $\chi \in \Cinfty_0(C)$ such that
    $\prol(\chi)\at{\supp f \cup \supp g} = 1$. Again, such a $\chi$
    exists. Then we can proceed as in the proof of
    Lemma~\ref{lemma:PositiveFunctionalAndGelfandIdeal} and have
    \[
    \SP{\deform{\iota^*} f, \deform{\iota^*} g}_\mu
    =
    \SP{f \bullet \chi, g \bullet \chi}_\mu
    =
    \SP{\chi, (\cc{f} \star g) \bullet \chi}_\mu
    =
    \omega_\mu(\cc{f} \star g)
    =
    \SP{\psi_f, \psi_g}_{\omega_\mu}.
    \]
    This shows that \eqref{eq:UnitaryIntertwiner} is isometric on
    $\psi_f, \psi_g$ without higher orders of $\lambda$. By
    $\mathbb{C}[[\lambda]]$-sesquilinearity of both inner products
    this holds in general whence \eqref{eq:UnitaryIntertwiner} is
    unitary. But then by \eqref{eq:HomotopyForKNull}
    \begin{align*}
        \pi_{\omega_\mu} (f) \psi_g 
        \mapsto
        \deform{\iota^*}(f \star g)
        = \deform{\iota^*}\left(
            f \star \prol(\deform{\iota^*}g) + f \star
            \deform{\koszul}_1 (\deform{h}_0 g)
        \right)
        =
        f \bullet \deform{\iota^*}g
        +
        \deform{\iota^*}\left(\deform{\koszul}_1
            (f \star \deform{h}_0 g)
        \right)
        =
        f \bullet \deform{\iota^*}g,
    \end{align*}
    since $\deform{\iota^*} \deform{\koszul}_1 = 0$. Thus
    \eqref{eq:UnitaryIntertwiner} intertwines the GNS representation
    $\pi_{\omega_\mu}$ into $\bullet$ as claimed.
\end{proof}
\begin{remark}[Quantization of coisotropic submanifolds]
    \label{remark:QuiantizationCoisotropic}
    In \cite{bordemann:2005a} as well as in
    \cite{cattaneo.felder:2007a, cattaneo.felder:2004a} the question
    was raised whether the classical $\Cinfty(M)$-module structure of
    $\Cinfty(C)$ of a submanifold $\iota: C \longrightarrow M$ can be
    quantized with respect to a given star product. In general, $C$
    has to be coisotropic but there are still obstructions beyond this
    zeroth order condition, see e.g.~\cite{willwacher:2007a} for a
    simple counter-example. In view of Theorem~\ref{theorem:GNS} one
    can rephrase and sharpen this task as follows: one should try to
    find a positive density $\mu_0 > 0$ on $C$ such that the
    functional $\omega_0(f) = \int_C f \mu_0$ allows for a deformation
    into a positive functional with respect to $\star$ \emph{and} such
    that it yields a GNS pre Hilbert space isomorphic to
    $\Cinfty_0(C)[[\lambda]]$.  This way one could obtain a
    deformation of the classical module structure which is even a
    $^*$-representation. Note that in the zeroth order this is
    consistent: the Gel'fand ideal of the classical integration
    functional is precisely the vanishing ideal of $C$.  Note also,
    that \emph{every} classically positive functional can be deformed
    into a positive functional with respect to $\star$.  However, the
    behaviour of the Gel'fand ideal under this deformation is rather
    mysterious, see e.g. the discussion in \cite[Sect.~7.1.5 \&
    Sect.~7.2.4]{waldmann:2007a} for further details and
    references. In any case, Theorem~\ref{theorem:GNS} gives some hope
    that this might be a reasonable approach also in some greater
    generality.
\end{remark}

%
%

\subsection{The reduced $^*$-involution}
\label{subsec:ReducedInvolution}

According to Proposition~\ref{proposition:ReducedInvolutionGeneral} we
have to show that the right $\bulletred$-multiplications are
adjointable with respect to $\SP{\cdot, \cdot}_\mu$. We prove a
slightly more general statement:
\begin{lemma}
    \label{lemma:DiffopsAdjointable}
    Let $P_r$ be bidifferential operators on $C$ and let $\mu_0 > 0$
    be an everywhere positive, smooth density on $C$.
    \begin{compactenum}
    \item \label{item:WellDefInnerProductP} Then the inner product
        \begin{equation}
            \label{eq:SPP}
            \SP{\phi, \psi}_P 
            =
            \int_C \left(
                \cc{\phi} \psi 
                + \sum_{r=1}^\infty \lambda^r P_r(\cc{\phi}, \psi)
            \right) \mu_0
        \end{equation}
        is well-defined for $\phi, \psi \in \Cinfty_0(C)[[\lambda]]$
        and non-degenerate.
    \item \label{item:DiffopAdjointable} Every formal series $D =
        \sum_{r=0}^\infty \lambda^r D_r$ of differential operators on
        $C$ is adjointable with respect to $\SP{\cdot, \cdot}_P$.
    \item \label{item:ZerothOrderOfAdjoint} The adjoint $D^+$ is again
        a formal series of differential operators and $D^+_0$
        coincides with the usual adjoint of $D_0$ with respect to the
        integration density $\mu_0$.
    \end{compactenum}
\end{lemma}
\begin{proof}
    The first part is clear. The second is shown order by
    order. Assume that we have found a differential operator
    $D^+_{(k)} = D^+_0 + \lambda D^+_1 + \cdots + \lambda^k D^+_k$
    such that
    \[
    \SP{\phi, D \psi}_P - \SP{D^+_{(k)} \phi, \psi}_P
    =
    \int_C \left(
        \sum_{r=k+1}^\infty \lambda^r E^{(k)}_r (\cc{\phi}, \psi)
    \right) \mu_0
    \]
    with some bidifferential operators $E^{(k)}_r$. For $k=0$ this is
    clearly achievable by the choice of $D^+_0$ as claimed in the
    third part. For a differential operator $D^+_{k+1}$ we have
    \begin{align*}
        &\SP{\phi, D\psi}_P 
        - \SP{
          \left(
              D^+_{(k)} + \lambda^{k+1} D^+_{k+1}
          \right) \phi,
          \psi}_P \\
        &\qquad=
        \lambda^{k+1} \int_C
        \left(
            E_{k+1}^{(k)} (\cc{\phi}, \psi) 
            -
            \cc{D^+_{k+1} \phi} \psi
        \right) \mu_0
        +
        \sum_{r=k+2}^\infty \lambda^r \int_C
        \left(
            E_r^{(k)}(\cc{\phi}, \psi)
            +
            P_{r-(k+1)}\left(\cc{D^+_{k+1} \phi}, \psi\right)
        \right) \mu_0.
    \end{align*}
    Integration by parts shows that we can arrange that all
    derivatives in $E^{(k)}_{k+1}(\cc{\phi}, \psi)$ are moved to
    $\cc{\phi}$ including terms coming from derivatives of $\mu_0$.
    But then we can chose $D^+_{k+1}$ to cancel the order
    $\lambda^{k+1}$. Since with this choice, $D^+_{k+1}$ is a
    differential operator itself, the error terms in higher orders are
    encoded by bidifferential operators again. Thus we can proceed by
    induction showing the second part. The third is clear from this
    construction.
\end{proof}
\begin{theorem}[Reduced $^*$-Involution]
    \label{theorem:ReducedStarInvolution}
    Let $u \in \Cinfty(M_\red)[[\lambda]]$. Then there exists a unique
    $u^* \in \Cinfty(M_\red)[[\lambda]]$ such that for all $\phi, \psi
    \in \Cinfty_0(C)[[\lambda]]$
    \begin{equation}
        \label{eq:ReducedStarInvolution}
        \SP{\phi, \psi \bulletred u}_\mu
        =
        \SP{\phi \bulletred u^*, \psi}_\mu.
    \end{equation}
    The map $u \mapsto u^*$ is a $^*$-involution for $\starred$ of the
    form
    \begin{equation}
        \label{eq:ReducedStarInvolutionExplicit}
        u^* = \cc{u} + \sum_{r=1}^\infty \cc{I_r(u)}
    \end{equation}
    with differential operators $I_r$ on $M_\red$.
\end{theorem}
\begin{proof}
    From Proposition~\ref{proposition:ReducedAlgebraBimodule},
    \refitem{item:RightModuleStructure} we know that the map $\phi
    \mapsto \phi \bulletred u$ is a formal series of differential
    operators. Moreover, from the locality of $\star$ and
    $\deform{\iota^*}$ it is clear that
    $\deform{\iota^*}(\cc{\prol(\phi)} \star \prol{\psi}) =
    \cc{\phi}\psi + \sum_{r=1}^\infty \lambda^r \tilde{P}_r(\cc{\phi},
    \psi)$ with some bidifferential operators $\tilde{P}_r$. Since
    $\mu_0 > 0$ we can write $\mu = \frac{\mu}{\mu_0} \mu_0$ with some
    \emph{function} $\frac{\mu}{\mu_0} = 1 + \cdots \in
    \Cinfty(C)[[\lambda]]$. Resorting by powers of $\lambda$ we
    conclude that $\SP{\cdot, \cdot}_\mu$ is of the
    form~\eqref{eq:SPP}. Then Lemma~\ref{lemma:DiffopsAdjointable},
    \refitem{item:DiffopAdjointable} shows that $\phi \mapsto \phi
    \bulletred u$ is adjointable. By
    Proposition~\ref{proposition:ReducedInvolutionGeneral} we known
    that the adjoint is necessarily of the form $\phi \mapsto \phi
    \bulletred u^*$ with a unique $u^* \in
    \Cinfty(M_\red)[[\lambda]]$. Moreover, it is clear that $u \mapsto
    u^*$ is a $^*$-involution for $\starred$. Then $u^* = \cc{u} +
    \cdots$ follows from Lemma~\ref{lemma:DiffopsAdjointable},
    \refitem{item:ZerothOrderOfAdjoint}. Finally, since the
    construction of the adjoint as in
    Lemma~\ref{lemma:DiffopsAdjointable} consists in finitely many
    integrations by parts and multiplications by coefficient functions
    in each fixed order of $\lambda$, we conclude that the higher
    order corrections in \eqref{eq:ReducedStarInvolutionExplicit} are
    differential.
\end{proof}

We want to determine the $^*$-involution
\eqref{eq:ReducedStarInvolutionExplicit} more closely and relate it to
the complex conjugation, which is a $^*$-involution of $\starred$ as
well, see Proposition~\ref{proposition:StarredIsHermitian}. To this
end we consider the formal series of densities $\Omega \in
\Secinfty(\Density T^*M_\red)[[\lambda]]$ corresponding to $\mu$ under
the canonical isomorphism \eqref{eq:OmegaToMu}. To proceed locally, we
chose a small enough open subset $U \subseteq M_\red$ and a
trivialization $\Phi: U \times G \longrightarrow \pi^{-1}(U) \subseteq
C$ where we trivialize the principal bundle $C$ as a \emph{right}
principal bundle, i.e. $\Phi$ is equivariant for the right actions.
\begin{proposition}
    \label{proposition:KMSlikeStuff}
    Let $\Omega \in \Secinfty(\Density T^*M_\red)[[\lambda]]$ be the
    pre-image of $\mu$ under \eqref{eq:OmegaToMu}.
    \begin{compactenum}
    \item One has $\Omega_0 > 0$, $\cc{\Omega} = \Omega$ and locally
        $\Phi^* \left(\mu\at{\pi^{-1}(U)}\right) = \Omega\at{U}
        \boxtimes \Dleft g$.
    \item For $u, v \in \Cinfty_0 (M_\red)[[\lambda]]$ one has
        \begin{equation}
            \label{eq:KMSlikeProperty}
            \int_{M_\red} v \starred u \; \Omega
            =
            \int_{M_\red} \cc{u^*} \starred v \; \Omega.
        \end{equation}
    \end{compactenum}
\end{proposition}
\begin{proof}
    The first part is clear from \eqref{eq:muIsDleftTensorOmega}
    discussed in the appendix.  To prove \eqref{eq:KMSlikeProperty} it
    is clearly sufficient to assume $u, v \in \Cinfty_0(U)[[\lambda]]$
    by a partition of unity argument. Choose a $\chi \in \Cinfty_0(G)$
    with $\int_G \chi \Dleft g = 1$.  We use now the trivialization
    $\Phi: U \times G \longrightarrow \pi^{-1}(U)$ to identify
    functions on $U \times G$ with those on $\pi^{-1}(U) \subseteq C$
    without explicitly writing $\Phi^*$ to simplify our notation.
    Then we consider $(1 \otimes \chi) \bulletred v \in \Cinfty_0(U
    \times G)[[\lambda]]$. Moreover, we use $1 \bulletred u = \pi^*u$
    according to Proposition~\ref{proposition:ReducedAlgebraBimodule},
    \refitem{item:EinsBulletu}.  This allows to evaluate the inner
    product $\SP{\cc{(1 \otimes \chi) \bulletred v}, \pi^*u}_\mu$ in
    two ways. First we have
    \begin{align*}
        \SP{\cc{(1 \otimes \chi) \bulletred v},
          \pi^*u}_\mu
        &=
        \int_C \deform{\iota^*}\left(
            \cc{
              \prol\left(
                  \cc{(1 \otimes \chi) \bulletred v}
              \right)
            }
            \star
            \prol(\pi^*u)
        \right) \mu \\
        &=
        \int_C
        \left((1 \otimes \chi) \bulletred v\right) \bulletred u
        \; \mu \\
        &=
        \int_C (1 \otimes \chi) \bulletred (v \starred u) \; \mu \\
        &=
        \int_{U \times G}
        (1 \otimes \chi) \bulletred (v \starred u) 
        \; \Omega \boxtimes \Dleft g \\
        &=
        \int_U \left(
            p \mapsto
            \int_G \mathsf{L}^*_{g^{-1}}
            \left(
                (1 \otimes \chi) \bulletred (v \starred u)
            \right) \Dleft g
            \At{(p, e)}
        \right) \Omega \\
        &\stackrel{\mathclap{\eqref{eq:RightModuleGInvariant}}}{=} \;
        \int_U \left(
            p \mapsto
            \int_G \mathsf{L}^*_{g^{-1}} (1 \otimes \chi)
            \Dleft g
        \right)
        \bulletred (v \starred u)
        \At{(p, e)}
        \Omega \\
        &=
        \int_U (1 \bulletred (v \starred u))           
        \At{(p, e)}
        \Omega \\
        &=
        \int_U v \starred u \; \Omega,
    \end{align*}
    which is the left hand side of \eqref{eq:KMSlikeProperty}. Note
    that the integral $\int_G \mathsf{L}^*_{g^{-1}} (1 \otimes \chi)
    \Dleft g$ is still understood as a function on $U \times G$. The
    second way to compute the inner product is
    \begin{align*}
        \SP{\cc{(1 \otimes \chi) \bulletred v},
          \pi^*u}_\mu
        &=
        \SP{\cc{(1 \otimes \chi) \bulletred v}, 1 \bulletred u}_\mu \\
        &\stackrel{\mathclap{(*)}}{=}
        \SP{\cc{(1 \otimes \chi) \bulletred v} \bulletred u^*, 1}_\mu
        \\ 
        &=
        \int_C \deform{\iota^*} \left(
            \cc{
              \prol\left(
                  \cc{(1 \otimes \chi) \bulletred v} \bulletred u^*
              \right)
            }
            \star 1
        \right) \mu \\
        &=
        \int_C
        \cc{
          \cc{(1 \otimes \chi) \bulletred v} \bulletred u^*
        } \; \mu \\
        &=
        \int_C \cc{
          \deform{\iota^*} \left(
              \prol\left(
                  \cc{(1 \otimes \chi) \bulletred v}
              \right)
              \star
              \prol(\pi^*u^*)
          \right)
        } \mu \\
        &\stackrel{\mathclap{\eqref{eq:ccintCf}}}{=} \;
        \int_C \deform{\iota^*} \left(
            \cc{
              \prol\left(
                  \cc{(1 \otimes \chi) \bulletred v}
              \right)
              \star
              \prol(\pi^*u^*)
            }
        \right) \mu \\
        &=
        \int_C \deform{\iota^*} \left(
            \cc{\prol(\pi^*u^*)}
            \star
            \prol\left(
                (1 \otimes \chi) \bulletred v
            \right)
        \right) \mu \\
        &=
        \int_U \left(
            p \mapsto
            \int_G \mathsf{L}^*_{g^{-1}}
            \left(
                \deform{\iota^*} \left(
                    \prol(\pi^*\cc{u^*})
                    \star
                    \prol\left((1 \otimes \chi) \bulletred v\right)
                \right)
            \right)
            \Dleft g
            \At{(p, e)}
        \right) \Omega \\
        &\stackrel{\mathclap{\eqref{eq:RightModuleGInvariant}}}{=} \;
        \int_U \left(
            p \mapsto
            \deform{\iota^*} \left(
                \prol(\pi^*\cc{u^*})
                \star
                \prol\left(
                    \left(
                        1 \otimes
                        \int_G \mathsf{L}^*_{g^{-1}} \chi \Dleft g
                    \right)
                    \bulletred v
                \right)
            \right)
            \At{(p, e)}
        \right) \Omega \\
        &=
        \int_U \left(
            p \mapsto
            \deform{\iota^*} \left(
                \prol(\pi^*\cc{u^*})
                \star
                \prol(\pi^*v)
            \right)
            \At{(p, e)}
        \right) \Omega \\
        &=
        \int_U \cc{u^*} \starred v \; \Omega,
    \end{align*}
    where we have used in $(*)$ that the integrals are still
    well-defined even if one of the functions is not in
    $\Cinfty_0(C)[[\lambda]]$ but the other is. Moreover, we used the
    fact that all maps are $G$-equivariant. Thus the integral can be
    moved directly in front of $\chi$ where it gives the constant
    function $1$ on $U \times G$. Then $1 \bulletred v = \pi^* v$ can
    be applied once more.
\end{proof}
\begin{remark}[KMS functional]
    \label{remark:KMSFunctional}
    Since we already know that the complex conjugation is a
    $^*$-involution for $\starred$ as well, the map
    \begin{equation}
        \label{eq:ImuAutomorphism}
        I_\mu: u \mapsto \cc{u^*}
    \end{equation}
    is a $\mathbb{C}[[\lambda]]$-linear \emph{automorphism} of
    $\starred$. Then the result \eqref{eq:KMSlikeProperty} means that
    the functional
    \begin{equation}
        \label{eq:KMSFunctional}
        \tau_\Omega:
        \Cinfty_0(M_\red)[[\lambda]]
        \ni u
        \; \mapsto \; 
        \tau_\Omega(u) = \int_{M_\red} u \; \Omega
        \in \mathbb{C}[[\lambda]]
    \end{equation}
    is actually a \emph{KMS functional} with respect to the $I_\mu$,
    i.e. we have the KMS property
    \begin{equation}
        \label{eq:tauOmegaIsKMS}
        \tau_\Omega(v \starred u)
        =
        \tau_\Omega(I_\mu(u) \starred v),
    \end{equation}
    see \cite{basart.lichnerowicz:1985a} as well as
    \cite[Sect.~7.1.4]{waldmann:2007a} for a discussion of the KMS
    condition in the context of deformation quantization including a
    proof of the classification of KMS functionals.
\end{remark}

The construction of the $^*$-involution $^*$ depends on the choice of
$\mu$. Two such choices $\mu$ and $\mu'$ are related by a unique
function $\varrho = \cc{\varrho} \in \Cinfty(M_\red)[[\lambda]]$ with
$\varrho_0 > 0$ via $\mu' = \pi^*\varrho \mu$. The corresponding
densities on $M_\red$ are then related by $\Omega' = \varrho \Omega$.
To relate the $^*$-involutions $^*$ and $^{*'}$ corresponding to $\mu$
and $\mu'$, respectively, we consider the KMS functionals
$\tau_\Omega$ and $\tau_{\Omega'}$:
\begin{lemma}
    \label{lemma:KMSFunctionalsNotDisjoint}
    Let $\Omega' = \varrho \Omega$ be as above. Then there exists a
    unique $\deform{\varrho} \in \Cinfty(M_\red)[[\lambda]]$ with
    $\deform{\varrho}_0 = \varrho_0 > 0$ such that for all $u \in
    \Cinfty_0(M_\red)[[\lambda]]$ we have
    \begin{equation}
        \label{eq:tauOmegaPrimeTauOmega}
        \tau_{\Omega'} (u)
        = \tau_{\Omega} (\deform{\varrho} \starred u). 
    \end{equation}
\end{lemma}
\begin{proof}
    This is shown inductively order by order in $\lambda$. Clearly,
    $\deform{\varrho}_0 = \varrho_0$ is the unique choice to satisfy
    \eqref{eq:tauOmegaPrimeTauOmega} in zeroth order. Then
    $\deform{\varrho}_1, \deform{\varrho}_2, \ldots$ are obtained by
    integration by parts, relying on the fact that $\starred$ is
    bidifferential. Uniqueness is clear from the non-degeneracy of the
    integration.
\end{proof}
\begin{theorem}
    \label{theorem:InvolutionInnerAuto}
    The $^*$-involutions $^*$ and $^{*'}$ obtained from different
    choices of $\mu$ and $\mu'$, respectively, are related by an inner
    automorphism
    \begin{equation}
        \label{eq:InvolutionInnerAuto}
        u^{*'} = 
        \cc{\deform{\varrho}} \starred
        u^*
        \starred \cc{\deform{\varrho}}^{-1},
    \end{equation}
    with $\deform{\varrho}$ as in
    Lemma~\ref{lemma:KMSFunctionalsNotDisjoint}.
\end{theorem}
\begin{proof}
    This is an easy computation. For $u, v \in
    \Cinfty_0(M_\red)[[\lambda]]$ we have
    \begin{align*}
        \tau_\Omega\left(
            \deform{\varrho} \starred \cc{u^{*'}} \starred v
        \right)
        &=
        \tau_{\Omega'} \left(
            \cc{u^{*'}} \starred v
        \right) \\
        &\stackrel{\mathclap{\eqref{eq:tauOmegaIsKMS}}}{=} \;
        \tau_{\Omega'} \left(
            v \starred u
        \right) \\
        &=
        \tau_\Omega\left(
            \deform{\varrho} \starred v \starred u
        \right) \\
        &\stackrel{\mathclap{\eqref{eq:tauOmegaIsKMS}}}{=} \;
        \tau_\Omega\left(
            \cc{u^*} \starred \deform{\varrho} \starred v
        \right),
    \end{align*}
    from which we deduce $\deform{\varrho} \starred \cc{u^{*'}} =
    \cc{u^*} \starred \deform{\varrho}$ as $v$ is arbitrary. Since
    $\deform{\varrho}$ starts with $\deform{\varrho}_0 = \varrho_0 >
    0$ it is $\starred$-invertible. This completes the proof.
\end{proof}
\begin{corollary}
    \label{corollary:TraceDensity}
    The $^*$-involution $^*$ coincides with the complex conjugation
    iff $\mu$ yields a trace density $\Omega$, i.e. $\tau_\Omega$ is a
    trace functional.
\end{corollary}
\begin{proof}
    If $u^* = \cc{u}$ then Proposition~\ref{proposition:KMSlikeStuff}
    gives the trace property immediately. Conversely, assume
    $\tau_\Omega$ is a trace. Then \eqref{eq:tauOmegaIsKMS} implies
    $\tau_\Omega(\cc{u^*} \starred v) = \tau_\Omega(u \starred v)$ for
    all $u, v \in \Cinfty_0(M_\red)[[\lambda]]$. But this gives
    $\cc{u^*} = u$ by the non-degeneracy of the integration.
\end{proof}
\begin{remark}[Unimodular Poisson structures]
    \label{remark:UnimodularPoisson}
    The \emph{existence} of a trace density for $\starred$ is
    non-trivial: the lowest order condition implies that $\Omega_0$ is
    a Poisson trace, i.e. the functional $\tau_{\Omega_0}$ vanishes on
    Poisson brackets. Thus the existence of such a $\Omega_0$ is
    equivalent to say that the Poisson structure of $M_\red$ is
    \emph{unimodular}, see e.g. \cite{weinstein:1997a}.
\end{remark}
\begin{remark}[Symplectic trace density]
    \label{remark:SymplecticTraceDensity}
    In the case where $M_\red$ is symplectic the Liouville volume
    density $\Omega_0 = |\omega_\red \wedge \cdots \wedge \omega_\red|
    \in \Secinfty(\Density T^*M_\red)$ is known to be (up to a
    normalization constant) the unique Poisson trace density.
    Moreover, in this case every star product $\starred$ allows a
    trace density $\Omega = \Omega_0 + \cdots$ which is again unique
    up to a normalization in $\mathbb{R}[[\lambda]]$. In fact, there
    is even a canonical way to fix the normalization, see e.g.
    \cite{karabegov:1998b, gutt.rawnsley:2002a,
      nest.tsygan:1995a}. Thus in the symplectic case there is a
    preferred choice for $\mu$ yielding the complex conjugation as
    $^*$-involution via Theorem~\ref{theorem:ReducedStarInvolution}.
\end{remark}

We can now give another interpretation of
Theorem~\ref{theorem:InvolutionInnerAuto}. Two choices of the density
$\mu$ (or equivalently, of $\Omega$) yield $^*$-involutions which are
related by an \emph{inner} automorphism. The question whether we can
modify $\Omega$ to get the complex conjugation directly boils down to
the question whether $I_\Omega = I_\mu$ from
\eqref{eq:ImuAutomorphism} is an inner automorphism or not. From
Theorem~\ref{theorem:ReducedStarInvolution} we known that
\begin{equation}
    \label{eq:ImuIOmega}
    I_\Omega = I_\mu = \id + \sum_{r=1}^\infty \lambda^r I_r
\end{equation}
with differential operators $I_r$ depending on the choice of
$\Omega$. Any automorphism starting with the identity in zeroth order
is necessarily of the form $I_\Omega = \exp(D_\Omega)$ with a
\emph{derivation}
\begin{equation}
    \label{eq:DOmega}
    D_\Omega = \sum_{r=1}^\infty \lambda^r D_\Omega^{(r)}
\end{equation}
of the star product $\starred$, see e.g.
\cite[Prop.~6.2.7]{waldmann:2007a} or
\cite[Lem.~5]{bursztyn.waldmann:2002a}. The automorphism $I_\Omega$
changes by the inner automorphism $\Ad(\cc{\deform{\varrho}})$ when
passing to $\Omega'$ according to
Theorem~\ref{theorem:InvolutionInnerAuto}. We arrive at the following
result:
\begin{proposition}[Modular class]
    \label{proposition:ModularClass}
    Let $D_\Omega$ be the derivation determined by $\Omega$ as above.
    \begin{compactenum}
    \item The first order term of $D_\Omega$ satisfies the classical
        infinitesimal KMS condition
        \begin{equation}
            \label{eq:InfinitesimalKMSClassical}
            \I \int_{M_\red} \{u, v\}_\red \Omega_0
            +
            \int_{M_\red} D_\Omega^{(1)} (u) v \Omega_0
            = 0
        \end{equation}
        for $u, v \in \Cinfty_0(M_\red)$.
    \item Denote the modular vector field of $M_\red$ with respect to
        $\Omega_0$ by $\Delta_{\Omega_0}$. Then we have
        \begin{equation}
            \label{eq:DOmegaIsModularVectorField}
            D_\Omega^{(1)} = \I \Delta_{\Omega_0}.
        \end{equation}
    \item For a different choice $\Omega'$ the difference $D_\Omega -
        D_{\Omega'}$ is an inner derivation of $\starred$. Hence the
        Hochschild cohomology class of $D_\Omega$ is independent of
        $\Omega$.
    \end{compactenum}
\end{proposition}
\begin{proof}
    The first part is the lowest non-vanishing order of
    \eqref{eq:tauOmegaIsKMS}. For the second, recall that the modular
    vector field $\Delta_{\Omega_0}$ with respect to a positive
    density is defined by $\Lie_{X_u} \Omega_0 = \Delta_{\Omega_0}(u)
    \Omega_0$ where $u \in \Cinfty(M_\red)$. Then
    \eqref{eq:DOmegaIsModularVectorField} is clear from
    \eqref{eq:InfinitesimalKMSClassical}. Thus let $\Omega'$ be
    another choice and let $\deform{\varrho}$ be given as in
    Lemma~\ref{lemma:KMSFunctionalsNotDisjoint}. Then
    \eqref{eq:InvolutionInnerAuto} gives
    \[
    \exp(D_{\Omega'})(u)
    = I_{\Omega'}(u)
    = \deform{\varrho}^{-1} 
    \starred I_\Omega(u) \starred
    \deform{\varrho}
    = \exp\left(- \ad_{\starred} (\Log(\varrho))\right) 
    \left(\exp(D_\Omega)(u)\right),
    \]
    where $\Log(\deform{\varrho}) = \log(\deform{\varrho}_0) + \cdots
    \in \Cinfty(M_\red)[[\lambda]]$ is the $\starred$-logarithm of
    $\deform{\varrho}$. Indeed, this logarithm exists globally thanks
    to $\deform{\varrho}_0 > 0$ and it is unique up to constants in $2
    \pi \I \mathbb{Z}$, see \cite[Sect.~6.3.1]{waldmann:2007a} for a
    detailed discussion of the logarithm with respect to star
    products. Since both derivations $D_\Omega$ and $-
    \ad_{\starred}(\Log(\deform{\varrho}))$ start in first order,
    their BCH series is well-defined. Thus
    \[
    D_{\Omega'} 
    = \mathrm{BCH}(-
    \ad_{\starred}(\Log(\deform{\varrho})), D_\Omega)
    =
    D_\Omega + \ad_{\starred} (w)
    \]
    with some $w \in \Cinfty(M_\red)[[\lambda]]$ since the commutators
    in the BCH series are all \emph{inner} derivations.
\end{proof}
\begin{remark}
    \label{remark:ModularClass}
    On one hand, the proposition gives us a quantum analog of the
    modular class $[\Delta_{\Omega_0}]$ in the first Poisson
    cohomology as discussed in \cite{weinstein:1997a}. Indeed, the
    Hochschild cohomology class $[D_\Omega]$ of $D_\Omega$ is a
    deformation of $[\Delta_{\Omega_0}]$ in a very good sense and
    measures the analogous quantity, namely whether one can find a
    trace density. On the other hand, the proposition tells us that
    this \emph{modular class} $[D_\Omega]$ of $\starred$ is precisely
    the obstruction for our construction of $^*$ to yield the complex
    conjugation by a clever choice of $\mu$.
\end{remark}

%
%

\section{Construction of the inner product bimodule}
\label{sec:ConstructionInnerProduct}

Having constructed the reduced algebra $(\Cinfty(M_\red)[[\lambda]],
\starred)$ out of the algebra $(\Cinfty(M)[[\lambda]], \star)$ we want
to relate their representation theories, i.e. their categories of
modules, as well. From a physical point of view this is even crucial:
ultimately, we need representations on some pre Hilbert space in order
to establish the superposition principle, see e.g.
\cite[Chap.7]{waldmann:2007a} for a detailed discussion in the context
of deformation quantization.

The usual idea is to use a bimodule and the tensor product to pass
from modules of one algebra to modules of the other in a functorial
way. Since we have constructed a bimodule structure on
$\Cinfty(C)[[\lambda]]$ it is tempting to use this particular
bimodule. While from a ring-theoretic point of view this is already
interesting, we want to compare $^*$-representations of the
$^*$-algebras on pre-Hilbert spaces and more generally on
algebra-valued inner product modules. To this end, we want to add some
more specific structure to the bimodule and make it an inner product
bimodule with (ultimately) a completely positive inner product. The
latter positivity will be discussed in
Section~\ref{sec:StrongMoritaEquivalenceBimodule}, here we focus on
the remaining properties of the inner product.

%
%

\subsection{Algebra-valued inner products and $^*$-representations of
  $^*$-algebras} 
\label{subsec:AlgebraValuedInnerProducts}

In this short subsection we collect some basic facts and definitions
from \cite{bursztyn.waldmann:2005b}. One may recognize that all the
notions are transferred from the theory of Hilbert modules over
$C^*$-algebras to our more algebraic framework.

Let again $\ring{R}$ be an ordered ring and $\ring{C} = \ring{R}(\I)$
as in Section~\ref{subsec:AlgebraicPreliminaries} and consider a
$^*$-algebra $\mathcal{A}$ over $\ring{C}$. Then an
\emph{$\mathcal{A}$-valued inner product} $\SPA{\cdot, \cdot}$ on a
right $\mathcal{A}$-module $\EA$ is a map $\SPA{\cdot, \cdot}: \EA
\times \EA \longrightarrow \mathcal{A}$ which is $\ring{C}$-linear in
the second argument and satisfies $\SPA{x, y} = (\SPA{y, x})^*$ as well
as $\SPA{x, y \cdot a} = \SPA{x, y} a$. Moreover, we require
non-degeneracy, i.e. $\SPA{x, y} = 0$ for all $y$ implies $x =0$. Here
and in the following we always assume that every module over
$\mathcal{A}$ carries a compatible $\ring{C}$-module structure. If
$\EA$ is equipped with such an inner product then $(\EA, \SPA{\cdot,
  \cdot})$ is called an \emph{inner product right
  $\mathcal{A}$-module} . Inner product left $\mathcal{A}$-modules are
defined analogously, with the only difference that we require
$\ring{C}$-linearity and $\mathcal{A}$-linearity to the left in the
\emph{first} argument. For $\mathcal{A} = \ring{C}$ we get back the
usual notions of an inner product module over the scalars $\ring{C}$
as in Section~\ref{subsec:AlgebraicPreliminaries}.

A map $A: \EA \longrightarrow \EpA$ between inner product right
$\mathcal{A}$-modules is called \emph{adjointable} if there exists a
map $A^*: \EpA \longrightarrow \EA$ with $\SPEpA{Ax, y} = \SPEA{x,
  A^*y}$ for all $x \in \EA$ and $y \in \EpA$. If such an $A^*$ exists
it is unique. It follows that $A$ and $A^*$ are right
$\mathcal{A}$-linear and the adjointable maps form a
$\ring{C}$-submodule of $\Hom_{\ring{C}}(\EA, \EpA)$. Moreover, $A
\mapsto A^*$ is $\ring{C}$-antilinear and involutive. Finally, for
another adjointable map $B: \EpA \longrightarrow \EppA$ also $BA$ is
adjointable with adjoint $(BA)^* = A^* B^*$. The adjointable maps are
denoted by $\Bounded[\mathcal{A}](\EA, \EpA)$.

A particular example of an adjointable map is obtained as follows: for
$y \in \EpA$ and $x \in \EA$ we set
\begin{equation}
    \label{eq:RankOneOperator}
    \Theta_{y, x} (z) = y \cdot \SPEA{x, z}
\end{equation}
for all $z \in \EA$. This yields an adjointable operator $\Theta_{y,
  x}: \EA \longrightarrow \EpA$ with adjoint $\Theta_{y, x}^* =
\Theta_{x, y}$. The $\ring{C}$-linear span of all these \emph{rank one
  operators} are called the \emph{finite rank operators}. They will be
denoted by $\Finite[\mathcal{A}](\EA, \EpA)$. As usual, we set
$\Bounded[\mathcal{A}](\EA) = \Bounded[\mathcal{A}](\EA, \EA)$ and
$\Finite[\mathcal{A}](\EA) = \Finite[\mathcal{A}](\EA, \EA)$. It
follows that $\Finite[\mathcal{A}](\EA)$ is a $^*$-ideal in the unital
$^*$-algebra $\Bounded[\mathcal{A}](\EA)$.

If $\mathcal{B}$ is another $^*$-algebra then a
\emph{$^*$-representation} of $\mathcal{B}$ on an inner product right
$\mathcal{A}$-module $\EA$ is a $^*$-homomorphism $\pi: \mathcal{B}
\longrightarrow \Bounded[\mathcal{A}](\EA)$. This way, $\EA$ becomes a
$(\mathcal{B}, \mathcal{A})$-bimodule and sometimes we simply write $b
\cdot x = \pi(b)x$ if the map $\pi$ is clear from the context. An
\emph{intertwiner} $T$ between two such $^*$-representations $(\EA,
\pi)$ and $(\EpA, \pi')$ is a left $\mathcal{B}$-linear adjointable
map $T: \EA \longrightarrow \EpA$, and hence in particular a
$(\mathcal{B}, \mathcal{A})$-bimodule morphism. The category of
$^*$-representations of $\mathcal{B}$ on inner product right
$\mathcal{A}$-modules is denoted by $\smod[\mathcal{A}](\mathcal{B})$.

A $^*$-representation $\BEA \in \smod[\mathcal{A}](\mathcal{B})$ is
called \emph{strongly non-degenerate} if $\mathcal{B} \cdot \BEA =
\BEA$. In the unital case this is equivalent to $\Unit_{\mathcal{B}}
\cdot x = x$ for all $x \in \BEA$. The subcategory of strongly
non-degenerate $^*$-representations is then denoted by
$\sMod[\mathcal{A}](\mathcal{B})$. Such $^*$-representations will also
be referred to as \emph{inner product $(\mathcal{B},
  \mathcal{A})$-bimodules}.

%
%

\subsection{The definition of the inner product}
\label{subsec:DefinitionInnerProduct}

The $^*$-algebras in question will be the functions
$\Cinfty(M)[[\lambda]]$ with $\star$ and the complex conjugation on
one hand and $\Cinfty(M_\red)[[\lambda]]$ with $\starred$ and the
complex conjugation on the other hand. Even though for $\starred$ we
might also take the other $^*$-involutions we restrict ourselves to
the simplest case of the complex conjugation.

The first question is in which of the two $^*$-algebras the inner
product should take values. One option is ruled out by the following
proposition:
\begin{proposition}
    \label{proposition:NoMValuedInnerProduct}
    Assume $\codim C \ge 1$.
    \begin{compactenum}
    \item For the classical $\Cinfty(M)$-module structure of
        $\Cinfty_0(C)$ a $\Cinfty(M)$-valued inner product does not
        exist.
    \item For $\Cinfty_0(C)[[\lambda]]$ there is no
        $\Cinfty(M)[[\lambda]]$-valued inner product with respect to
        $\star$ and the left module structure $\bullet$.
    \end{compactenum}
\end{proposition}
\begin{proof}
    Assume there is such an inner product. Let $\phi, \psi \in
    \Cinfty_0(C)$ with $\SP{\phi, \psi} \ne 0$ be given. Then there is
    a point $p \in M \setminus C$ with $\SP{\phi, \psi}(p) \ne 0$.
    Choose $f \in \Cinfty(M)$ with $f(p) \ne 0$ but $f$ equal to zero
    in an open neighbourhood of $C$. Then we get a contradiction from
    $0 \ne f(p) \SP{\phi, \psi}(p) = \SP{f \cdot \phi, \psi} (p) =
    \SP{\iota^*f \phi, \psi}(p) = 0$ since $\iota^*f = 0$. This shows
    the first part. The second follows analogously, since
    $\deform{\iota^*}(f \star \prol(\phi)) = 0$ by the locality of
    $\star$ and Lemma~\ref{lemma:RestrictionLocal}.
\end{proof}

The other option of a $\Cinfty(M_\red)[[\lambda]]$-valued inner
product will be more promising. Before giving the definition we have
to specify the precise function space on $C$ for the module: as we
will need integrations, $\Cinfty(C)$ will be too large in general. On
the other hand, $\Cinfty_0(C)$ will work but is too small for purposes
of Morita theory in
Section~\ref{sec:StrongMoritaEquivalenceBimodule}. Thus we shall use
the following option: We define
\begin{equation}
    \label{eq:Cinftycf}
    \Cinftycf(C) =
    \left\{
        \phi \in \Cinfty(C)
        \; \big| \;
        \supp(\phi) \cap \pi^{-1}(K)
        \;
        \textrm{is compact for all compact}
        \;
        K \subseteq M_\red
    \right\},
\end{equation}
and call this subspace of $\Cinfty(C)$ the functions with
\emph{locally uniformly compact support in fiber directions}. Clearly
$\Cinfty_0(C) \subseteq \Cinftycf(C)$. If $G$ is compact then
$\Cinftycf(C) = \Cinfty(C)$ while $\Cinftycf(C) = \Cinfty_0(C)$ if $M$
is compact. The importance of the space $\Cinftycf(C)$ comes from the
following simple observation, which becomes trivial for a compact
group $G$.
\begin{lemma}
    \label{lemma:CinftycfProperties}
    \begin{compactenum}
    \item \label{item:CcfStableDiffOps} The subspace $\Cinftycf(C)
        \subseteq \Cinfty(C)$ is an ideal, stable under all
        differential operators, the $G$-action, and complex
        conjugation.
    \item \label{item:IntegrateCcf} For $\phi \in \Cinftycf(C)$ the
        function
        \begin{equation}
            \label{eq:IntegrateCcf}
            \int_G \mathsf{L}^*_{g^{-1}} \phi \Dleft g:
            c
            \; \mapsto \; 
            \int_G \phi(\mathsf{L}_{g^{-1}}(c)) \Dleft g
        \end{equation}
        is a smooth and invariant function on $C$.
    \item \label{item:FullnessFunction} There exists a function $0 \le
        \epsilon \in \Cinftycf(C)$ with
        \begin{equation}
            \label{eq:eIntegralEins}
            \int_G \mathsf{L}^*_{g^{-1}}  \epsilon \Dleft g = 1.
        \end{equation}
    \end{compactenum}
\end{lemma}
\begin{proof}
    The first part is trivial. For the second, let $U \subseteq C$ be
    an open pre-compact subset. Then $\phi(\mathsf{L}_{g^{-1}}(c)) =
    0$ for $c \in U$ unless $\mathsf{L}_{g^{-1}}(c) \in \supp
    \phi$. On the other hand we know $\mathsf{L}^*_{g^{-1}}(c) \in
    \pi^{-1}(\pi(U^\cl))$ and $\supp \phi \cap \pi^{-1}(\pi(U^\cl))$
    is compact thanks to $\phi \in \Cinftycf(C)$. Thus
    $\phi(\mathsf{L}^*_{g^{-1}}(c)) = 0$ for $c \in U$ unless $g \in
    G_{U, \phi}$ where
    \[
    G_{U, \phi} =
    \{
    g \in G \; | \;
    \textrm{there exists a}
    \;
    c \in U^\cl
    \;
    \textrm{with}
    \;
    \mathsf{L}_{g^{-1}}(c) \in \supp \phi \cap \pi^{-1}(\pi(U^\cl))
    \}.
    \]
    Hence we conclude that for $c \in U$ we have
    \[
    \int_G \phi(\mathsf{L}_{g^{-1}}(c)) \Dleft g
    =
    \int_{G_{U, \phi}} \phi(\mathsf{L}_{g^{-1}}(c)) \Dleft g.
    \]
    Since $G_{U, \phi}$ is \emph{compact} by the properness of the
    action we can deduce that \eqref{eq:IntegrateCcf} is well-defined
    and yields a smooth function on the open subset $U$ by applying
    the usual ``differentiation commutes with integration''
    techniques. But this implies smoothness everywhere. Clearly, the
    averaging integral yields an invariant function.  For the third
    part, we use an atlas of local trivializations $\{U_\alpha,
    \Phi_\alpha\}$ of the (right) principal bundle. Moreover, let $0
    \le \chi_\alpha \in \Cinfty_0(M_\red)$ be a locally finite
    partition of unity subordinate to this atlas. Finally, we choose
    $0 \le \chi \in \Cinfty_0(G)$ with $\int_G \chi(g) \Dleft g =
    1$. For $c \in C$ we define
    \[
    \epsilon(c) =
    \sum\nolimits_\alpha (\chi_\alpha \otimes \chi) \circ
    \Phi_\alpha^{-1}(c).
    \]
    It easily follows that $\epsilon \in \Cinftycf(C)$. Moreover, a
    simple computation shows that
    \[
    \int_G 
    (\chi_\alpha \otimes \chi) \circ \Phi_\alpha^{-1}
    (\mathsf{L}_{g^{-1}}(c))
    \Dleft g
    =
    \chi_\alpha (\pi(c))
    \]
    for all $\alpha$. Thus $\epsilon$ satisfies \eqref{eq:eIntegralEins}.
\end{proof}
\begin{corollary}
    \label{corollary:CcfSubModule}
    The $\mathbb{C}[[\lambda]]$-submodule $\Cinftycf(C)[[\lambda]]
    \subseteq \Cinfty(C)[[\lambda]]$ is a $\starred$-submodule with
    respect to $\bulletred$.
\end{corollary}
\begin{proof}
    Since $\bulletred$ acts via differential operators in each order
    of $\lambda$, this is clear from
    Lemma~\ref{lemma:CinftycfProperties},
    \refitem{item:CcfStableDiffOps}.
\end{proof}

It will be this submodule on which the algebra-valued inner product
can be defined.
\begin{definition}[Algebra-valued inner product]
    \label{definition:AlgebraValuedInnerProduct}
    Let $\phi, \psi \in \Cinftycf(C)[[\lambda]]$. Then one defines
    their $\Cinfty(M_\red)[[\lambda]]$-valued inner product $\SP{\phi,
      \psi}_\red$ pointwise by
    \begin{equation}
        \label{eq:SPred}
        \SP{\phi, \psi}_\red (\pi(c))
        =
        \int_G \left(
            \deform{\iota^*} \left(
                \cc{\prol(\phi)} \star \prol(\psi)
            \right)
        \right) (\mathsf{L}_{g^{-1}}(c))
        \Dleft g.
    \end{equation}
\end{definition}
\begin{lemma}
    \label{lemma:AlgebraValuedIPWelldef}
    The inner product $\SP{\cdot, \cdot}_\red$ is well-defined and
    $\mathbb{C}[[\lambda]]$-sesquilinear.
\end{lemma}
\begin{proof}
    The sesquilinearity is clear. Even though $\prol$ is a non-local
    operation it preserves the ``support in $C$-directions''. Using
    the locality of $\star$ and $\deform{\iota^*}$ we see that
    \[
    \supp \left(
        \deform{\iota^*}\left(
            \cc{\prol(\phi)} \star \prol(\psi)
        \right)
    \right)
    \subseteq
    \supp \phi \cap \supp \psi.
    \]
    Thus the integrand of \eqref{eq:SPred} is indeed in
    $\Cinftycf(C)[[\lambda]]$. By
    Lemma~\ref{lemma:CinftycfProperties} it follows that the right
    hand side of \eqref{eq:SPred} is well-defined and yields an
    invariant smooth function on $C$. Hence it is of the form $\pi^*
    \SP{\phi, \psi}_\red$ with $\SP{\phi, \psi}_\red \in
    \Cinfty(M_\red)[[\lambda]]$ as claimed.
\end{proof}
The next technical lemma shows alternative ways to compute $\SP{\phi,
  \psi}_\red$. Here again we rewrite the integral over $G$ as an
integral over a suitable compact subset.
\begin{lemma}
    \label{lemma:LocalizeSPred}
    Let $\phi, \psi \in \Cinftycf(C)$ and let $U \subseteq C$ be
    open and pre-compact.
    \begin{compactenum}
    \item \label{item:GUphipsiDef} $G_{U, \phi, \psi} = \{g \in G \; |
        \; \textrm{there exists a} \; c \in U^\cl \; \textrm{with} \;
        \mathsf{L}_{g^{-1}}(c) \in \supp \phi \cap \supp \psi \cap
        \pi^{-1}(\pi(U^\cl))\}$ is a compact subset of $G$.
    \item \label{item:SPredLocally} One has
        \begin{align}
            \label{eq:SPredLocalI}
            \pi^*\SP{\phi, \psi}_\red \At{U}
            &=
            \deform{\iota^*} \int_{G_{U, \phi, \psi}}
            \mathsf{L}^*_{g^{-1}} \left(
                \cc{\prol(\phi)} \star \prol(\psi)
            \right)
            \Dleft g \At{U} \\
            &=
            \deform{\iota^*} \int_{G_{U, \phi, \psi}}
            \cc{\prol(\mathsf{L}^*_{g^{-1}}\phi)}
            \star
            \prol(\mathsf{L}^*_{g^{-1}} \psi)
            \Dleft g \At{U}
            \label{eq:SPredLocalII}.
        \end{align}
    \end{compactenum}
\end{lemma}
\begin{proof}
    By assumptions $\supp \phi \cap \supp \psi \cap
    \pi^{-1}(\pi(U^\cl))$ is compact. Then the properness of the
    action assures that $G_{U, \phi, \psi}$ is compact as well. A
    similar argument as in the proof of
    Lemma~\ref{lemma:CinftycfProperties} shows that
    \[
    \pi^* \SP{\phi, \psi}_\red \At{U}
    =
    \int_{G_{U, \phi, \psi}}
    \mathsf{L}^*_{g^{-1}} \left(
        \deform{\iota^*}\left(
            \cc{\prol(\phi)} \star \prol(\psi)
        \right)
    \right) \Dleft g \At{U}.
    \]
    Now by $G$-invariance of $\deform{\iota^*}$ we can exchange the
    action of $g \in G_{U, \phi, \psi}$ with $\deform{\iota^*}$.
    Moreover, since $\deform{\iota^*} = \iota^* \circ S$ is in each
    order of $\lambda$ a differential operator followed by $\iota^*$,
    see Lemma~\ref{lemma:RestrictionLocal}, the integration over the
    compact subset $G_{U, \phi, \psi}$ can be exchanged with
    $\deform{\iota^*}$. Thus \eqref{eq:SPredLocalI} follows. Since
    $\star$ and $\prol$ are $G$-invariant, \eqref{eq:SPredLocalII}
    follows as well.
\end{proof}
\begin{lemma}
    \label{lemma:InnerProductRightLinear}
    The inner product $\SP{\cdot, \cdot}_\red$ is right
    $\starred$-linear, i.e. we have
    \begin{equation}
        \label{eq:InnerProductRightLinear}
        \SP{\phi, \psi \bulletred u}_\red
        = \SP{\phi, \psi}_\red \starred u
    \end{equation}
    for all $\phi, \psi \in \Cinftycf(C)[[\lambda]]$ and $u \in
    \Cinfty(M_\red)[[\lambda]]$.
\end{lemma}
\begin{proof}
    First note that it suffices to consider $\phi, \psi \in
    \Cinftycf(C)$ and $u \in \Cinfty(M_\red)$. We evaluate
    \eqref{eq:InnerProductRightLinear} on an open subset $U \subseteq
    C$ after pulling it back to $C$. In addition, we can assume $U$ to
    be pre-compact. Then let $G_{U, \phi, \psi}$ be as in
    Lemma~\ref{lemma:LocalizeSPred}, \refitem{item:GUphipsiDef}. For
    the integrand we have
    \[
    \deform{\iota^*} \left(
        \cc{\prol(\phi)} \star \prol(\psi \bulletred u)
    \right)
    =
    \left(\cc{\prol(\phi)} \bullet \psi\right) \bulletred u
    =
    \deform{\iota^*} \left(
        \prol\left(
            \deform{\iota^*} \left(
                \cc{\prol(\phi)} \star \prol(\psi)
            \right)
        \right)
        \star
        \prol(\pi^*u)
    \right)
    \]
    by the bimodule properties as in
    Proposition~\ref{proposition:ReducedAlgebraBimodule}. Since
    $\supp(\psi \bulletred u) \subseteq \supp \psi$, by
    Lemma~\ref{lemma:LocalizeSPred} we get on the open subset $U$
    \begin{align*}
        \pi^* \SP{\phi, \psi \bulletred u}_\red \At{U}
        &=
        \deform{\iota^*}
        \int_{G_{U, \phi, \psi}}
        \mathsf{L}^*_{g^{-1}} \left(
            \prol\left(
                \deform{\iota^*}\left(
                    \cc{\prol(\phi)} \star \prol(\psi)
                \right)
            \right)
            \star \prol (\pi^*u)
        \right)
        \Dleft g \At{U} \\
        &=
        \deform{\iota^*}
        \left(
            \prol \left(
                \int_{G_{U, \phi, \psi}}
                \mathsf{L}^*_{g^{-1}} 
                \left(
                    \deform{\iota^*}
                    \left(
                        \cc{\prol(\phi)} \star \prol(\psi)
                    \right)
                \right)
                \Dleft g
            \right)
            \At{U}
            \star
            \prol(\pi^*u) \At{U}
        \right) \\
        &=
        \deform{\iota^*}
        \left(
            \prol \left(\pi^*\SP{\phi, \psi}_\red \At{U}\right)
            \star 
            \prol\left(\pi^*u\At{U}\right)
        \right) \\
        &=
        \pi^* \left(\SP{\phi, \psi}_\red \starred u\right) \At{U},
    \end{align*}
    where we have used that $\prol$ commutes with the integration
    thanks to the invariance. Moreover, we used the fact that we can
    restrict to open subsets on $C$: even though $\prol$ is non-local,
    the nice tubular neighbourhood shows that this is possible.
\end{proof}
\begin{lemma}
    \label{lemma:SPredSymmetrie}
    Let $\phi, \psi \in \Cinftycf(C)[[\lambda]]$. Then we have
    \begin{equation}
        \label{eq:SPredSymmetric}
        \cc{\SP{\phi, \psi}_\red} = \SP{\psi, \phi}_\red.
    \end{equation}
\end{lemma}
\begin{proof}
    Again, it will be sufficient to consider $\phi, \psi \in
    \Cinftycf(C)$. Let $U \subseteq C$ be open and pre-compact and let
    $G_{U, \phi, \psi}$ be as in
    Lemma~\ref{lemma:LocalizeSPred},~\refitem{item:GUphipsiDef}. Then
    we compute
    \begin{align*}
        \cc{\SP{\phi, \psi}_\red}\At{U}
        &=
        \cc{
          \deform{\iota^*}
          \int_{G_{U, \phi, \psi}}
          \mathsf{L}^*_{g^{-1}}
          \left(
              \cc{\prol(\phi)} \star \prol(\psi)
          \right) \Dleft g
        } \At{U} \\
        &=
        \deform{\iota^*}
        \int_{G_{U, \phi, \psi}}
        \mathsf{L}^*_{g^{-1}}
        \left(
            \cc{
              \cc{\prol(\phi)} \star \prol(\psi)
            }
        \right) \Dleft g
        \At{U} \\
        &=
        \pi^*\SP{\psi, \phi}_\red \At{U},
    \end{align*}
    since first the integration yields some invariant functions
    allowing to use \eqref{eq:ccIotaInvariantf} and, second, the star
    product $\star$ is Hermitian.
\end{proof}
\begin{lemma}
    \label{lemma:NondegenerateAndClassicalLimitOfSPred}
    Let $\phi, \psi \in \Cinftycf(C)[[\lambda]]$ then $\SP{\phi,
      \phi}_\red = 0$ iff $\phi = 0$. Moreover, the classical limit of
    the inner product is
    \begin{equation}
        \label{eq:ClassicalLimit}
        \pi^*\SP{\phi, \psi}_\red
        =
        \int_G \mathsf{L}^*_{g^{-1}} (\cc{\phi}\psi) \Dleft g
        + \cdots. 
    \end{equation}
\end{lemma}
\begin{proof}
    The classical limit~\eqref{eq:ClassicalLimit} is clear. From this
    we also conclude the first statement by induction on the lowest
    non-vanishing order of $\phi$.
\end{proof}
We can now collect these results in the following theorem:
\begin{theorem}
    \label{theorem:AlgebraValuedInnerProduct}
    The inner product $\SP{\cdot, \cdot}_\red$ turns
    $\Cinftycf(C)[[\lambda]]$ into an inner product right module over
    $(\Cinfty(M_\red)[[\lambda]], \starred)$.
\end{theorem}

%
%

\subsection{Further properties of $\SP{\cdot, \cdot}_\red$}
\label{subsec:FurtherPropertiesOfSPred}

Since also the action of $\Cinfty(M)[[\lambda]]$ via $\bullet$ on
$\Cinfty(C)[[\lambda]]$ is by differential operators in each order,
$\Cinftycf(C)[[\lambda]]$ is also a left $\star$-submodule of
$\Cinfty(C)[[\lambda]]$ and hence a $(\star, \starred)$-bimodule. The
following proposition shows that the action is by adjointable
operators and yields thus a $^*$-representation:
\begin{proposition}
    \label{proposition:LeftModuleIsRepresentation}
    The left module structure $\bullet$ is a $^*$-representation of
    $(\Cinfty(M)[[\lambda]], \star)$ on the inner product right
    $(\Cinfty(M_\red)[[\lambda]], \starred)$-module
    $\Cinftycf(C)[[\lambda]]$, i.e. we have for all $\phi, \psi \in
    \Cinftycf(C)[[\lambda]]$ and $f \in \Cinfty(M)[[\lambda]]$
    \begin{equation}
        \label{eq:fbulletIsAdjointable}
        \SP{\phi, f \bullet \psi}_\red
        =
        \SP{\cc{f} \bullet \phi, \psi}_\red.
    \end{equation}
\end{proposition}
\begin{proof}
    Again, we consider $\phi, \psi \in \Cinftycf(C)$ and an open and
    pre-compact $U \subseteq C$ with $G_{U, \phi, \psi}$ as in
    Lemma~\ref{lemma:LocalizeSPred}, \refitem{item:GUphipsiDef}. Since
    $f$ acts by differential operators we have $\supp(f \bullet \psi)
    \subseteq \supp \psi$ as well as $\supp(\cc{f} \bullet \phi)
    \subseteq \supp \phi$. Analogously to the proof of
    Lemma~\ref{lemma:InnerProductRightLinear} we first compute
    \begin{align*}
        \deform{\iota^*} \left(
            \cc{\prol(\phi)} \star \prol (f \bullet \psi)
        \right)
        &=
        \cc{\prol(\phi)} \bullet (f \bullet \psi) \\
        &=
        \left(
            \cc{\prol(\phi)} \star f
        \right)
        \bullet \psi \\
        &=
        \left(
            \cc{\cc{f} \star \prol(\phi)}
        \right)
        \bullet \psi \\
        &=
        \deform{\iota^*}
        \left(
            \left(
                \cc{\cc{f} \star \prol(\phi)}
            \right)
            \star \prol(\psi)
        \right) \\
        &=
        \deform{\iota^*}
        \left(
            \cc{
              \cc{\prol(\psi)}
              \star
              \left(
                  \cc{f} \star \prol(\phi)
              \right)
            }
        \right).
    \end{align*}
    Using this we can take out the complex conjugation under the
    averaging integral since we have invariant functions thanks to
    \eqref{eq:ccIotaInvariantf}. This gives
    \[
    \pi^*\SP{\phi, f \bullet \psi}_\red \At{U}
    =
    \cc{
      \int_{G_{U, \phi, \psi}}
      \mathsf{L}^*_{g^{-1}} \deform{\iota^*}
      \left(
          \cc{\prol(\psi)} \star (\cc{f} \star \prol(\phi))
      \right)
      \Dleft g
    } \At{U},
    \]
    where we again used Lemma~\ref{lemma:LocalizeSPred}. Now
    analogously one shows that
    \[
    \deform{\iota^*} \left(
        \cc{\prol(\psi)} \star \prol(\cc{f} \bullet \phi)
    \right)
    =
    \deform{\iota^*} \left(
        \cc{\prol(\psi)} \star
        (\cc{f} \star \prol(\phi))
    \right),
    \]
    using the definition of the left module structure. Putting these
    together shows \eqref{eq:fbulletIsAdjointable}.
\end{proof}

The last feature of $\SP{\cdot, \cdot}_\red$ we want to discuss is the
$G$-invariance. In fact, the $G$-action of $\Cinftycf(C)[[\lambda]]$
turns out to be unitary with respect to $\SP{\cdot, \cdot}_\red$ up to
the modular function:
\begin{proposition}
    \label{proposition:GactsUnitarily}
    Let $\phi, \psi \in \Cinftycf(C)[[\lambda]]$ and $g \in G$ then
    \begin{equation}
        \label{eq:GactsAlmostUnitarily}
        \SP{\mathsf{L}^*_{g^{-1}} \phi,
          \mathsf{L}^*_{g^{-1}} \psi}_\red
        =
        \Delta(g) \SP{\phi, \psi}_\red.
    \end{equation}
\end{proposition}
\begin{proof}
    Using the $G$-equivariance of all operators involved in the
    definition of $\SP{\cdot, \cdot}_\red$ the computation is
    analogous to the one for the classical
    limit~\eqref{eq:ClassicalLimit} which is clear.
\end{proof}

Thus we obtain a unitary (left) representation of $G$ with respect to
$\SP{\cdot, \cdot}_\red$ if we set
\begin{equation}
    \label{eq:UnitaryRepresentationOfG}
    U(g) \phi = \frac{1}{\sqrt{\Delta(g)}} \mathsf{L}^*_{g^{-1}} \phi.
\end{equation}
The infinitesimal version of this action is given by the left
multiplication with the components of the momentum map as in
\eqref{eq:Jxibulletphi}: here we again see that $\kappa = \frac{1}{2}$
is the good choice to get this unitarity. From
Proposition~\ref{proposition:LeftModuleStructure},
\refitem{item:LeftModuleGInvariant} we see that the left module
structure is covariant:
\begin{corollary}
    \label{corollary:LeftModuleCovariant}
    The $^*$-representation $\bullet$ if $(\Cinfty(M)[[\lambda]],
    \star)$ on $\Cinftycf(C)[[\lambda]]$ is $G$-covariant with respect
    to the unitary representation $U$ of $G$ as in
    \eqref{eq:UnitaryRepresentationOfG}, i.e.
    \begin{equation}
        \label{eq:CovariantRepresentation}
        U(g) (f \bullet \phi)
        =
        (\mathsf{L}^*_{g^{-1}} f) \bullet (U(g)\phi)
    \end{equation}
    for $f \in \Cinfty(M)[[\lambda]]$, $g \in G$, and $\phi \in
    \Cinftycf(C)[[\lambda]]$.
\end{corollary}

%
%

\section{A strong Morita equivalence bimodule}
\label{sec:StrongMoritaEquivalenceBimodule}

In this section we establish that the bimodule structure and inner
product $\SP{\cdot, \cdot}_\red$ on $\Cinftycf(C)[[\lambda]]$ actually
gives a strong Morita equivalence bimodule between
$\Cinfty(M_\red)[[\lambda]]$ and the finite rank operators on
$\Cinftycf(C)[[\lambda]]$.

%
%

\subsection{Strong Morita equivalence of $^*$-algebras}
\label{subsec:StrongMoritaEquivalence}

There are many approaches to Morita theory of $^*$-algebras, see
e.g. \cite{ara:1999a, bursztyn.waldmann:2005b} for a detailed
discussion and further references. We recall the basic notions: If
$\BEA \in \smod[\mathcal{A}](\mathcal{B})$ and $\CFB \in
\smod[\mathcal{B}](\mathcal{C})$ are $^*$-representations on inner
product modules then on their algebraic tensor product one defines an
$\mathcal{A}$-valued inner product by
\begin{equation}
    \label{eq:RieffelInnerProduct}
    \SPFEA{\phi \otimes x, \psi \otimes y}
    = \SPEA{x, \SPFB{\phi, \psi} \cdot y}
\end{equation}
and extends this by $\ring{C}$-sesquilinearity to $\mathcal{F}
\tensor[\mathcal{B}] \mathcal{E}$. It follows that this is again an
inner product except, however, it might be degenerate. Thus one
considers the quotient
\begin{equation}
    \label{eq:InternalTensorProduct}
    \CFB \itensor[\mathcal{B}] \BEA
    =
    \CFB \tensor[\mathcal{B}] \BEA 
    \big/ \left(\CFB \tensor[\mathcal{B}] \BEA\right)^\bot,
\end{equation}
and obtains a $^*$-representation of $\mathcal{C}$ on an inner product
right $\mathcal{A}$-module, called the \emph{internal tensor product}
$\CFB \itensor[\mathcal{B}] \BEA$. It turns out that $\itensor$ gives
a functor
\begin{equation}
    \label{eq:itensorFunctor}
    \itensor:
    \smod[\mathcal{B}](\mathcal{C}) 
    \times
    \smod[\mathcal{A}](\mathcal{B}) 
    \longrightarrow
    \smod[\mathcal{A}](\mathcal{C}),
\end{equation}
which preserves strongly non-degenerate $^*$-representations. This
allows for the following definition: a $^*$-representation $\BEA \in
\sMod[\mathcal{A}](\mathcal{B})$ is called a \emph{$^*$-equivalence
  bimodule} if there exists another $^*$-representation $\AtEB \in
\smod[\mathcal{B}](\mathcal{A})$ such that 
\begin{equation}
    \label{eq:StarEquivalence}
    \AtEB \itensor[\mathcal{B}] \BEA \cong \mathcal{A}
    \quad
    \textrm{and}
    \quad
    \BEA \itensor[\mathcal{A}] \AtEB \cong \mathcal{B},
\end{equation}
where isomorphism is understood as unitary intertwiner of
$^*$-representations with the algebras being equipped with the
canonical inner product, i.e. $\SPA{a, a'} = a^*a'$. In order to have
$\mathcal{A} \in \smod[\mathcal{A}](\mathcal{A})$ we have to restrict
ourselves to \emph{non-degenerate} $^*$-algebras, i.e. $ab = 0$ for
all $b$ implies $a = 0$. Moreover, one has to require the
$^*$-algebras to be \emph{idempotent}, i.e. $\mathcal{A} \cdot
\mathcal{A} = \mathcal{A}$. Every unital $^*$-algebra is both,
non-degenerate and idempotent. For this class of $^*$-algebras it can
be checked that the existence of a $^*$-equivalence bimodule defines
indeed an equivalence relation, called \emph{$^*$-Morita equivalence}.

The following characterization of $^*$-equivalence bimodules will be
very useful for us: $\BEA \in \smod[\mathcal{A}](\mathcal{B})$ is a
$^*$-equivalence bimodule iff the following holds: First, $\mathcal{B}
\cdot \mathcal{E} = \mathcal{E} = \mathcal{E} \cdot \mathcal{A}$, i.e.
both module structures are strongly non-degenerate. Second, there is a
$\mathcal{B}$-valued inner product $\BSP{\cdot, \cdot}$ with $\BSP{x
  \cdot a, y} = \BSP{x, y \cdot a^*}$ for all $x, y \in \mathcal{E}$
and $a \in \mathcal{A}$. Third, both inner products are \emph{full},
i.e. $\SPA{\mathcal{E}, \mathcal{E}} = \mathcal{A}$ and
$\BSP{\mathcal{E}, \mathcal{E}} = \mathcal{B}$. Last, the inner
products are compatible, i.e. $\BSP{x, y} \cdot z = x \cdot \SPA{y,
  z}$ for all $x, y, z \in \mathcal{E}$. In this case, the ``inverse''
bimodule $\AtEB$ can be chosen to be the complex conjugate bimodule
$\AccEB$, defined in an obvious way. Moreover, $\mathcal{B}$ turns out
to be $^*$-isomorphic to the finite rank operators
$\Finite[\mathcal{A}](\EA)$ via the left module structure and
$\BSP{\cdot, \cdot}$ is just $\Theta_{\cdot, \cdot}$ as in
\eqref{eq:RankOneOperator} under this identification, see
\cite{bursztyn.waldmann:2005b} for a detailed discussion.

In fact, we only have to find a strongly non-degenerate right
$\mathcal{A}$-module $\EA$ with a full inner product $\SPA{\cdot,
  \cdot}$ then $\mathcal{B} = \Finite[\mathcal{A}](\EA)$ is
$^*$-Morita equivalent to $\mathcal{A}$ via $\BEA$ where the
$\Finite[\mathcal{A}](\EA)$-valued inner product is $\Theta_{\cdot,
  \cdot}$. All $^*$-Morita equivalent $^*$-algebras to $\mathcal{A}$
arise this way up to $^*$-isomorphism.

Let us finally recall the notion of strong Morita equivalence. First
recall that $a \in \mathcal{A}$ is called \emph{positive} if
$\omega(a) \ge 0$ for all positive linear functionals $\omega:
\mathcal{A} \longrightarrow \ring{C}$. This allows to define an inner
product $\SPA{\cdot, \cdot}$ on a right $\mathcal{A}$-module $\EA$ to
be \emph{completely positive} if the matrix $(\SPA{x_i, x_j}) \in
M_n(\mathcal{A})$ is positive for all $x_1, \ldots, x_n \in \EA$ and
$n \in \mathbb{N}$.  This gives a refined notion of inner product
modules: an inner product right $\mathcal{A}$-module is called
\emph{pre Hilbert right $\mathcal{A}$-module} if the inner product is
completely positive. The category of $^*$-representations of
$\mathcal{B}$ on pre Hilbert right $\mathcal{A}$-modules is denoted by
$\rep[\mathcal{A}](\mathcal{B})$ and the sub-category of strongly
non-degenerate ones by $\Rep[\mathcal{A}](\mathcal{B})$. It can be
shown that the internal tensor product $\itensor$ preserves complete
positivity of the inner products. Moreover, the canonical inner
product on $\mathcal{A}$ is easily shown to be completely
positive. This allows to define a $^*$-equivalence bimodule to be a
\emph{strong equivalence bimodule} if both inner products are
completely positive. This way one arrives at the notion of
\emph{strong Morita equivalence}.

One of the most important consequences of $^*$-Morita equivalence and
strong Morita equivalence is that the $^*$-representation theories
$\sMod[\mathcal{D}](\cdot)$ and $\Rep[\mathcal{D}](\cdot)$ of
equivalent $^*$-algebras are equivalent, even for arbitrary
``coefficient $^*$-algebra'' $\mathcal{D}$. In fact, every equivalence
bimodule $\BEA$ provides an equivalence of categories by the internal
tensor product
\begin{equation}
    \label{eq:RieffelInductions}
    \BEA \itensor[\mathcal{A}]:
    \sMod[\mathcal{D}](\mathcal{A})
    \longrightarrow
    \sMod[\mathcal{D}](\mathcal{B})
    \quad
    \textrm{and}
    \quad
    \BEA \itensor[\mathcal{A}]:
    \Rep[\mathcal{D}](\mathcal{A})
    \longrightarrow
    \Rep[\mathcal{D}](\mathcal{B}),
\end{equation}
in the case of strong Morita equivalence, respectively. We refer to
\cite{bursztyn.waldmann:2005b} for further details on strong Morita
equivalence of $^*$-algebras.

%
%

\subsection{Fullness and the finite rank operators}
\label{subsec:FullnessFiniteRankOperators}

We want to investigate the inner product bimodule
$\Cinftycf(C)[[\lambda]]$ from the point of view of Morita theory. The
first result is the fullness based on the following lemma:
\begin{lemma}
    \label{lemma:Fullness} There exists a function $\deform{e} \in
    \Cinftycf(C)[[\lambda]]$ such that
    \begin{equation}
        \label{eq:SPeeIsEins}
        \SP{\deform{e}, \deform{e}}_\red = 1.
    \end{equation}
\end{lemma}
\begin{proof}
    We consider the function $\epsilon \in \Cinftycf(C)$ from
    Lemma~\ref{lemma:CinftycfProperties},
    \refitem{item:FullnessFunction}. Since each term $\Phi_\alpha^*
    (\chi_\alpha \otimes \chi) \ge 0$ is already non-negative, we
    obtain
    \[
    \pi^*\SP{\epsilon, \epsilon}_\red
    =
    \sum_{\alpha, \beta} \int_G 
    \mathsf{L}^*_{g^{-1}} 
    \left(
        \cc{\Phi_\alpha^*(\chi_\alpha \otimes \chi)}
        \Phi_\beta^*(\chi_\beta \otimes \chi)
    \right)
    \Dleft g + \cdots
    \]
    by \eqref{eq:ClassicalLimit}. Now all integrands are non-negative
    and since $\int_G \mathsf{L}^*_{g^{-1}} \epsilon \Dleft g = 1$ we
    see that already the diagonal terms give a strictly positive
    contribution.  It follows that $\SP{\epsilon, \epsilon}_\red = u_0
    + \cdots \in \Cinfty(M_\red)[[\lambda]]$ with $u_0 > 0$. Since in
    addition $\cc{\SP{\epsilon, \epsilon}_\red} = \SP{\epsilon,
      \epsilon}_\red$ is Hermitian, we can take a (Hermitian) square
    root with respect to $\starred$ of the form
    $\sqrt[\starred]{\SP{\epsilon, \epsilon}} = \sqrt{u_0} +
    \cdots$. Clearly, this is invertible hence $\deform{e} = \epsilon
    \starred \frac{1}{\sqrt[\starred]{\SP{\epsilon, \epsilon}}}$ will
    do the job.
\end{proof}
Note that if $G$ is compact and $\Dleft g$ normalized to volume $1$
then $\deform{e} = 1 \in \Cinftycf(C)$ would be a canonical choice.
\begin{proposition}[$^*$-Equivalence bimodule]
    \label{proposition:StarEquivalenceBimodule}
    For $\Cinftycf(C)[[\lambda]]$ we have:
    \begin{compactenum}
    \item \label{item:Fullness} The inner product $\SP{\cdot,
          \cdot}_\red$ is full.
    \item \label{item:StarEquivalenceBimodule} The inner product right
        $\Cinfty(M_\red)[[\lambda]]$-module $\Cinftycf(C)[[\lambda]]$
        becomes a $^*$-equivalence bimodule for the finite rank
        operators $\Finite(\Cinftycf(C)[[\lambda]])$ acting from the
        left as usual with $\Theta_{\cdot, \cdot}$ as inner product.
    \item \label{item:Cyclic} As left
        $\Finite(\Cinftycf(C)[[\lambda]])$-module,
        $\Cinftycf(C)[[\lambda]]$ is cyclic with cyclic vector
        $\deform{e}$. Moreover, for all $\phi \in
        \Cinftycf(C)[[\lambda]]$
        \begin{equation}
            \label{eq:DualBasis}
            \phi = \Theta_{\phi, \deform{e}} (\deform{e}).
        \end{equation}
    \item \label{item:HermitianDualBasis} The pair $(\deform{e},
        \deform{e})$ constitutes a Hermitian dual basis hence
        $\Cinftycf(C)[[\lambda]]$ is finitely generated (by
        $\deform{e}$) and projective over
        $\Finite(\Cinftycf(C)[[\lambda]])$.
    \item \label{item:ThetaIsCompletelyPositive} The inner product
        $\Theta_{\cdot, \cdot}$ is completely positive.
    \end{compactenum}
\end{proposition}
\begin{proof}
    Since $\Cinfty(M_\red)[[\lambda]]$ is unital,
    \eqref{eq:SPeeIsEins} is sufficient to conclude fullness. Then the
    second part is clear by the general structure of $^*$-equivalence
    bimodules. Now \eqref{eq:DualBasis} is just a computation using
    \eqref{eq:SPeeIsEins}. This means that $\deform{e}$ is cyclic for
    the action of $\Finite(\Cinftycf(C)[[\lambda]])$. Even more, since
    the inner product $\Theta_{\phi, \deform{e}}$ is
    $\Finite(\Cinftycf(C)[[\lambda]])$-linear to the left in the first
    argument $\phi$, we have a dual basis $(\deform{e}, \Theta_{\cdot,
      \deform{e}})$ for the left
    $\Finite(\Cinftycf(C)[[\lambda]])$-module
    $\Cinftycf(C)[[\lambda]]$. Since the linear form $\Theta_{\cdot,
      \deform{e}}$ is obviously an inner product by some vector,
    namely $\deform{e}$, this is even a Hermitian dual basis, see
    \cite{bursztyn.waldmann:2005b}. The existence of such a Hermitian
    dual basis is true in general since the other algebra is unital.
    The remarkable point is that we only need one vector $\deform{e}$.
    For the last part, let $\phi_1, \ldots, \phi_n \in
    \Cinftycf(C)[[\lambda]]$ be given and consider the matrix $\Phi =
    (\Theta_{\phi_i, \phi_j}) \in
    M_n(\Finite(\Cinftycf(C)[[\lambda]]))$. From \eqref{eq:SPeeIsEins}
    it immediately follows that $\Phi = \Psi^*\Psi$ where $\Psi$ is
    the matrix with $\Theta_{\deform{e}, \phi_i}$ in the first row and
    zeros elsewhere. Thus $\Phi$ is clearly positive proving the last
    part.
\end{proof}
\begin{remark}
    \label{remark:NoUnit}
    Assume that $G$ is not finite, which we always can assume in the
    context of phase space reduction. Then the finite rank operators
    do \emph{not} have a unit. Otherwise, the module would be also
    finitely generated and projective as right
    $\Cinfty(M_\red)[[\lambda]]$-module. This is clearly not the case
    as the finitely projective modules over star product algebras are
    known to be deformations of sections of vector bundles over the
    base manifold. Thus we have a first non-trivial example of a
    $^*$-equivalence bimodule for star product algebras going beyond
    the unital case studied in \cite{bursztyn.waldmann:2002a}.
\end{remark}
\begin{remark}
    \label{remark:MNotFiniteRank}
    Note also that the $^*$-algebra $(\Cinfty(M)[[\lambda]], \star)$
    does \emph{not} act via finite rank operators on
    $\Cinftycf(C)[[\lambda]]$. The reason is that the finite rank
    operators are \emph{non-local} as they involve the integration
    along the fibers in the definition of $\SP{\cdot, \cdot}_\red$.
    However, we know that $\phi \mapsto f \bullet \phi$ is a formal
    series of differential and hence local operators. Of course, we
    can not expect $(\Cinfty(M)[[\lambda]], \star)$ and
    $(\Cinftycf(M_\red)[[\lambda]], \starred)$ to be $^*$-equivalent
    as in this case the classical limit of this bimodule would be
    still an equivalence bimodule and thus $M$ is necessarily
    diffeomorphic to $M_\red$, see
    \cite[Cor.~7.8]{bursztyn.waldmann:2001a}.
\end{remark}

In the rest of this subsection we discuss the classical limit of the
$^*$-equivalence bimodule and the finite rank operators: we consider
$\Cinftycf(C)$ as right $\Cinfty(M_\red)$-module with inner product
\begin{equation}
    \label{eq:SPredclDef}
    \SP{\phi, \psi}^\cl_\red
    =
    \int_G \mathsf{L}^*_{g^{-1}} (\cc{\phi}\psi) \Dleft g.
\end{equation}
Since we do not rely on phase space reduction, the Lie group $G$ needs
not to be connected in the following theorem:
\begin{theorem}
    \label{theorem:ClassicalLimitBimodule}
    Let $C \circlearrowright G \longrightarrow M_\red$ be a principal
    bundle.
    \begin{compactenum}
    \item \label{item:ClassicalInnerProductFullAndNice} The inner
        product $\SP{\cdot, \cdot}_\red^\cl$ is full, non-degenerate
        and completely positive.
    \item \label{item:ClassicalSMEBimodule} The pre Hilbert right
        $\Cinfty(M_\red)$-module $\Cinftycf(C)$ becomes a strong
        Morita equivalence bimodule for the finite rank operators
        $\Finite(\Cinftycf(C))$ acting from the left with the
        canonical inner product $\Theta_{\cdot, \cdot}$.
    \end{compactenum}
\end{theorem}
\begin{proof}
    Clearly, $\SP{\cdot, \cdot}_\red^\cl$ is a well-defined
    $\Cinfty(M_\red)$-valued inner product which is non-degenerate.
    Analogously to the construction of $\deform{e}$ we find a function
    $e \in \Cinftycf(C)$ with $\SP{e, e}_\red^\cl = 1$ showing
    fullness. For $\phi_1, \ldots, \phi_n \in \Cinftycf(C)$ the matrix
    $(\SP{\phi_i, \phi_j}_\red^\cl) \in M_n(\Cinfty(M_\red))$ is
    pointwise positive. But this is precisely the characterization of
    positive elements in $M_n(\Cinfty(M_\red))$ resulting from the
    general algebraic definition, see also
    \cite[App.~B]{bursztyn.waldmann:2001a}. Finally, the same argument
    as in the proof of
    Proposition~\ref{proposition:StarEquivalenceBimodule} using $e$
    instead of $\deform{e}$ shows the complete positivity for
    $\Theta_{\cdot, \cdot}$ also in this case. Thus the second
    statement follows.
\end{proof}

In order to get a more geometric description of the finite rank
operators we consider the ``extended'' principal bundle
\begin{equation}
    \label{eq:ExtendedPrincipalBundle}
    \pi_e: C_e = C \times_{M_\red} C \longrightarrow M_\red,
\end{equation}
which is indeed a $G \times G$ principal bundle over $M_\red$.  Denote
by $\pr_1, \pr_2: C \times_{M_\red} C \longrightarrow C$ the
projections onto the first and second factor of the fiber product,
respectively.  We can define $\Cinftycf(C_e)$ analogously to the case
of $\Cinftycf(C)$. On this space we define a ``matrix multiplication''
by
\begin{equation}
    \label{eq:ConvolutionProduct}
    (\Psi \conv \Xi)(c, c') = 
    \int_G
    \Psi(c, \mathsf{L}_{g^{-1}}(\tilde{c}))
    \Xi(\mathsf{L}_{g^{-1}}(\tilde{c}), c')
    \Dleft g
\end{equation}
for $(c, c') \in C_e$ and $\Phi, \Xi \in \Cinftycf(C_e)$ where
$\tilde{c}$ is an arbitrary point in the same fiber as $c$ and
$c'$. Moreover, one defines an action of $\Cinftycf(C_e)$ on
$\Cinftycf(C)$ by
\begin{equation}
    \label{eq:ClassicalLeftAction}
    (\Psi \cdot \phi)(c) 
    = \int_G
    \Psi(c, \mathsf{L}_{g^{-1}}(\tilde{c}))
    \phi(\mathsf{L}_{g^{-1}}(\tilde{c}))
    \Dleft g,
\end{equation}
where now $\Psi \in \Cinftycf(C_e)$, $\phi \in \Cinftycf(C)$, and
$\tilde{c}$ is again an arbitrary point in the same fiber as
$c$. These definitions turn out to make sense and have the following
properties:
\begin{theorem}
    \label{theorem:ClassicalKompactOperators}
    Let $\Psi, \Xi \in \Cinftycf(C_e)$ and $\phi \in \Cinftycf(C)$.
    \begin{compactenum}
    \item \label{item:ConvolutionWellDefined} The product
        \eqref{eq:ConvolutionProduct} is well-defined, independent on
        the choice of $\tilde{c}$ and yields a smooth function $\Psi
        \conv \Xi \in \Cinftycf(C_e)$ with
        \begin{equation}
            \label{eq:SupportConvolution}
            \supp(\Psi \conv \Xi) \subseteq
            \pr_1(\supp \Psi) \times_{M_\red} \pr_2 (\supp \Xi).
        \end{equation}
    \item \label{item:StarInvolutionForConvolution} Together with the
        $^*$-involution defined for $\Psi \in \Cinftycf(C_e)$ by
        \begin{equation}
            \label{eq:InvolutionForConvolutionAlgebra}
            \Psi^*(c, c') = \cc{\Psi(c', c)}
        \end{equation}
        the product $\conv$ turns
        $\Cinftycf(C_e)$ into a $^*$-algebra over $\mathbb{C}$.
    \item \label{item:WellDefinedLeftActionClassicallly} The
        definition \eqref{eq:ClassicalLeftAction} is independent on
        $\tilde{c}$ and yields a smooth function $\Psi \cdot \phi \in
        \Cinftycf(C)$. This way, $\Cinftycf(C)$ becomes a left
        $\Cinftycf(C_e)$-module. The map $\Psi \mapsto (\phi \mapsto
        \Psi \cdot \phi)$ is injective.
    \item \label{item:ClassicalModuleIsPreHilbert} One obtains a pre
        Hilbert $(\Cinftycf(C_e), \Cinfty(M_\red))$-bimodule with
        respect to the $\Cinfty(M_\red)$-valued inner product
        $\SP{\cdot, \cdot}_\red^\cl$.
    \item \label{item:FiniteRankToConvolutionAlgebra} The linear map
        determined by
        \begin{equation}
            \label{eq:FiniteRankToCinftycfCe}
            \Finite(\Cinftycf(C)) 
            \ni \Theta_{\phi, \psi} \; \mapsto \;
            \phi \otimes \cc{\psi} = 
            ((c, c') \mapsto \phi(c) \cc{\psi(c')})
            \in \Cinftycf(C_e)
        \end{equation}
        yields an injective $^*$-algebra homomorphism such that the
        left module multiplication \eqref{eq:ClassicalLeftAction} of
        the images under \eqref{eq:FiniteRankToCinftycfCe} coincides
        with the canonical action of the finite rank operators.
    \end{compactenum}
\end{theorem}
\begin{proof}
    For the pointwise existence we note that those $g \in G$ with
    $\mathsf{L}_{g^{-1}}(\tilde{c}) \in \supp \Psi(c, \cdot)$ or
    $\mathsf{L}_{g^{-1}}(\tilde{c})\in \supp \Xi(\cdot, c')$,
    respectively, are compact by the properness of the action and the
    assumption $\Psi, \Xi \in \Cinftycf(C_e)$. Thus the integral
    \eqref{eq:ConvolutionProduct} only uses the $g \in G$ in the
    intersection of these two compact subsets. This shows convergence
    of the integral. The independence on $\tilde{c}$ follows from the
    left invariance of $\Dleft g$ at once. To show smoothness of
    \eqref{eq:ConvolutionProduct} we need a locally uniform compact
    domain of integration. Thus let $U_e \subseteq C_e$ be open and
    pre-compact and assume that $\pi_e(U_e) \subseteq M_\red$ allows
    for a trivialization of $C$ over this open subset. Thus we can
    choose a smooth local section $\tilde{c}: \pi_e(U_e)
    \longrightarrow \pi^{-1}(\pi_e(U_e)) \subseteq C$. By assumption,
    the subsets $\supp \Psi \cap \pi_e^{-1}(\pi_e(U_e^\cl))$ and
    $\supp \Xi \cap \pi_e^{-1}(\pi_e(U_e^\cl))$ are compact in
    $C_e$. Thus also $K_\Psi = \pr_2 (\supp \Psi \cap
    \pi_e^{-1}(\pi_e(U_e^\cl)))$ and $K_\Xi = \pr_1 (\supp \Xi \cap
    \pi_e^{-1}(\pi_e(U_e^\cl)))$ are compact subsets of $C$ projecting
    into $\pi_e(U_e^\cl)$. Then by the properness of the action we see
    that $G_\Psi = \{g \in G \; | \; \textrm{there exists a} \; c \in
    \pr_1(U^\cl) \; \textrm{with} \;
    \mathsf{L}_{g^{-1}}(\tilde{c}(\pi(c))) \in K_\Psi\}$ as well as
    $G_\Xi = \{g \in G \; | \; \textrm{there exists a} \; c' \in
    \pr_2(U^\cl) \; \textrm{with} \;
    \mathsf{L}_{g^{-1}}(\tilde{c}(\pi(c'))) \in K_\Xi\}$ are compact
    subsets of $G$. By construction, the integration only needs the $g
    \in G_\Psi \cap G_\Xi$ for all $(c, c') \in U_e$. Thus on $U_e$,
    we have a uniform compact domain of integration. Then the
    smoothness follows easily. The
    statement~\eqref{eq:SupportConvolution} is clear from which we
    also deduce that $\Psi \conv \Xi \in\Cinftycf(C_e)$, showing the
    first part. For the second part we note that
    \eqref{eq:InvolutionForConvolutionAlgebra} is clearly an
    involutive and anti-linear endomorphism of $\Cinftycf(C_e)$. A
    simple computation shows that this indeed gives an
    anti-automorphism of $\conv$. It remains to show the associativity
    of the product $\conv$ which is an easy consequence of Fubini's
    theorem as locally we only have to integrate over compact subsets
    of $G$. For the third part, one proceeds analogously to show $\Psi
    \cdot \phi \in \Cinftycf(C)$, independently on the choice of
    $\tilde{c}$. The module property is again an application of
    Fubini's theorem. The injectivity is clear. For the fourth part we
    have to show that \eqref{eq:ClassicalLeftAction} is a
    $^*$-representation. Thus we compute
    \begin{align*}
        \SP{\phi, \Xi \cdot \psi}_\red^\cl (\pi(c))
        &=
        \int_G
        \cc{\phi(\mathsf{L}_{g^{-1}}(c))}
        (\Xi \cdot \psi)(\mathsf{L}_{g^{-1}}(c))
        \Dleft g \\
        &=
        \int_G \cc{\phi(\mathsf{L}_{g^{-1}}(c))} 
        \int_G
        \Xi(\mathsf{L}_{g^{-1}}(c), \mathsf{L}_{h^{-1}}(\tilde{c}))
        \psi(\mathsf{L}_{h^{-1}}(\tilde{c}))
        \Dleft h \Dleft g \\
        &=
        \int_G \cc{
          \int_G
          \Xi^*(\mathsf{L}_{h^{-1}}(\tilde{c}),
          \mathsf{L}_{g^{-1}}(c))
          \phi(\mathsf{L}_{g^{-1}}(c)) 
          \Dleft g
        }
        \:
        \psi(\mathsf{L}_{h^{-1}}(\tilde{c})
        \Dleft h \\
        &=
        \SP{\Xi^* \cdot \phi, \psi}_\red^\cl (\pi(\tilde{c})),
    \end{align*}
    using once again Fubini's theorem. Since $\pi(\tilde{c}) = \pi(c)$
    the fourth part follows. For the last part we first note that
    clearly $\phi \otimes \cc{\psi} \in \Cinftycf(C_e)$. We compute
    \[
    ((\phi \otimes \cc{\psi}) \cdot \chi)(c)
    =
    \int_G 
    \phi(c)
    \cc{\psi(\mathsf{L}_{g^{-1}}(\tilde{c}))}
    \chi(\mathsf{L}_{g^{-1}}(\tilde{c}))
    \Dleft g
    =
    \phi(c) \SP{\psi, \chi}_\red^\cl (\pi(c))
    =
    \left(\Theta_{\phi, \psi} \chi\right)(c).
    \]
    This shows that under \eqref{eq:FiniteRankToCinftycfCe} the usual
    action of finite rank operators is turned into
    \eqref{eq:ClassicalLeftAction}. By the injectivity statement in
    the third part, \eqref{eq:FiniteRankToCinftycfCe} is necessarily
    injective and a $^*$-homomorphism.
\end{proof}
\begin{remark}
    \label{remark:CinftycfCompactOperators}
    The $^*$-algebra $\Cinftycf(C_e)$ is typically strictly larger
    than the image of $\Finite(\Cinftycf(C))$ under the
    embedding~\eqref{eq:FiniteRankToCinftycfCe}. Nevertheless, with
    respect to a suitable locally convex topology, the finite rank
    operators are dense in $\Cinftycf(C_e)$. Morally, $\Cinftycf(C_e)$
    corresponds to ``Hilbert-Schmidt''-like operators. Note that the
    product $\conv$ can be viewed as a ``matrix-multiplication'' of
    matrices with components labeled by the continuous index $g \in
    G$. Similarly, the left module structure
    \eqref{eq:ClassicalLeftAction} is the application of a matrix to a
    vector whose components are labelled by the continuous index $g
    \in G$. Finally, the $^*$-involution is the usual
    ``matrix-adjoint''.
\end{remark}
\begin{remark}
    \label{remark:GeometryOfCe}
    Geometrically, the bundle $C_e = C \times_{M_\red} C
    \longrightarrow M_\red$ is diffeomorphic to $C \times G$ since for
    $(c, c') \in C_e$ there exists a unique $g \in G$ with $c' =
    \mathsf{L}_{g^{-1}}(c)$. However, this diffeomorphism will destroy
    the simple form of the matrix-multiplication formulas
    \eqref{eq:ConvolutionProduct} and \eqref{eq:ClassicalLeftAction}.
    Nevertheless, rewriting things this way, one recognizes the usual
    crossed product construction, here in its ``smooth'' version: the
    smooth functions on the reduced space $M_\red$ are $^*$-Morita
    equivalent to the crossed product of the functions on $C$ with the
    group $G$. Of course, since $G$ is non-compact, some care has to
    be taken and the above function space provides a good notion for
    the crossed product in the smooth situation, see
    \cite{rieffel:1974b} for the original version of this statement in
    the $C^*$-algebraic category.
\end{remark}

%
%

\subsection{Complete positivity}
\label{subsec:CompletePositivity}

Before showing the complete positivity of $\SP{\cdot, \cdot}_\red$ we
recall some facts on deformation quantization of principal bundles $C
\circlearrowright G \longrightarrow M_\red$ from
\cite{bordemann.neumaier.waldmann.weiss:2007a:pre}: it can be shown
that $\Cinfty(C)[[\lambda]]$ can always be equipped with a right
module structure $\bulletred$ with respect to a given star product
$\starred$ on $M_\red$. Moreover, $\bulletred$ is \emph{unique} up to
equivalence, i.e. up to a module isomorphism of the form $\id +
\sum_{r=1}^\infty \lambda^r T_r$ with $T_r \in \Diffop(C)$. Moreover,
it is known that the module endomorphisms of such a deformation inside
the differential operators $\Diffop(C)[[\lambda]]$ are obtained from a
deformation of the \emph{vertical} differential operators
$\Diffop_\ver(C)[[\lambda]]$, now equipped with a new, deformed
composition law $\star'$ and a deformed action $\bullet'$ on
$\Cinfty(C)[[\lambda]]$. Again, $\star'$ and $\bullet'$ are uniquely
determined by $\starred$ up to equivalence. We shall use these
results, in particular the uniqueness statements, later on.

\begin{remark}[Positive algebra elements]
    \label{remark:PositiveElementsForStarProducts}
    We have to make an additional requirement on the positive linear
    functionals $\omega = \sum_{r=0}^\infty \lambda^r \omega_r:
    \Cinfty(M_\red)[[\lambda]] \longrightarrow \mathbb{C}[[\lambda]]$
    for the following. While the algebraic definition
    \eqref{eq:omegaPositive} allows for $\omega_r$ in the full
    algebraic dual of $\Cinfty(M_\red)$ we have to restrict to
    \emph{distributions}, i.e. continuous linear functionals with
    respect to the canonical Fr\'echet topology of
    $\Cinfty(M_\red)$. It is easy to construct (algebraically)
    positive linear functionals where the higher orders are not of
    this form. However, this restriction seems to be reasonable as
    long as we work with smooth functions. Potentially, this will
    result in \emph{more} positive algebra elements.
\end{remark}

We start now with the local situation: we consider an open and small
enough subset $U \subseteq M_\red$ and assume that $\pi^{-1}(U) \cong
U \times G$ is trivial. Following the principal bundle tradition, the
group acts from the \emph{right} by right multiplications denoted by
$\mathsf{r}_g: U \times G \longrightarrow U \times G$. The
corresponding left action is therefor given by $\mathsf{L}_g =
\mathsf{r}_{g^{-1}}$, and \emph{not} by the left multiplication
$\mathsf{l}_g$.

The star product $\starred$ extends canonically to $\Cinfty(U \times
G)[[\lambda]]$ yielding a star product, still denoted by $\starred$,
for the Poisson structure on $U \times G$ which is the flat horizontal
lift of the one on $U$. This way,
\begin{equation}
    \label{eq:piAlgebraHomomorphism}
    \pi^*:
    \left(
        \Cinfty(U)[[\lambda]], \starred
    \right)
    \longrightarrow
    \left(
        \Cinfty(U \times G)[[\lambda]], \starred
    \right)
\end{equation}
is a $^*$-algebra homomorphism. Thus we also obtain a canonical right
$\starred$-module structure $\bulletcan$ on $\Cinfty(U \times
G)[[\lambda]]$ using \eqref{eq:piAlgebraHomomorphism}. For this
particular right module structure we can define a very simple inner
product. Indeed, for $\phi, \psi \in \Cinftycf(U \times G)[[\lambda]]$
we set
\begin{equation}
    \label{eq:CanonicalInnerProduct}
    \pi^*\SP{\phi, \psi}_\can
    = 
    \int_G \mathsf{L}^*_{g^{-1}} (\cc{\phi} \starred \psi) \Dleft g
    =
    \int_G \mathsf{r}^*_{g} (\cc{\phi} \starred \psi) \Dleft g.
\end{equation}
It is clear that this gives a well-defined right $\starred$-linear
$\Cinfty(M_\red)[[\lambda]]$-valued non-degenerate and full inner
product.
\begin{proposition}
    \label{proposition:SPcanIsCompletePositive}
    The canonical inner product $\SP{\cdot, \cdot}_\can$ is completely
    positive.
\end{proposition}
\begin{proof}
    Let $\phi_1, \ldots, \phi_n \in \Cinftycf(U \times G)[[\lambda]]$
    and let $\Omega = \Omega_0 + \lambda \Omega_1 + \cdots$ be a
    positive linear functional of $M_n(\Cinfty(M_\red)[[\lambda]],
    \starred)$ such that each $\Omega_r$ is a distribution according
    to our convention in
    Remark~\ref{remark:PositiveElementsForStarProducts}. Then by
    continuity
    \begin{align*}
        \Omega\left(\SP{\phi_i, \phi_j}_\can\right)
        &=
        \Omega\left(
            p \mapsto
            \int_G \left(\cc{\phi_i} \starred \phi_j\right)(p, g)
            \Dleft g
        \right) \\
        &=
        \int_G \Omega\left(
            \cc{\phi_i}(\cdot, g) \starred
            \phi_j(\cdot, g)
        \right) \Dleft g \\
        &\ge 0,
    \end{align*}
    since the lowest non-vanishing order of the integrand is positive
    for every $g \in G$.
\end{proof}

To proceed, we need a more explicit description of $\star'$ and
$\bullet'$ for our local model. Let $D \in \Diffop_\ver^k(U \times G)$
be a vertical differential operator on $U \times G$. Then for a chosen
basis $e_1, \ldots, e_N \in \lie{g}$ we have uniquely determined
functions $D^{i_1 \ldots i_N} \in \Cinfty(U \times G)$ such that
\begin{equation}
    \label{eq:VerticalDiffop}
    D
    = \sum_{|I| \le k} D^{i_1 \cdots i_N} 
    \left(\Lie_{(e_1)_{U \times G}}\right)^{i_1} \cdots
    \left(\Lie_{(e_N)_{U \times G}}\right)^{i_N}.
\end{equation}
Since the fundamental vector fields do not commute in general,
\eqref{eq:VerticalDiffop} can be viewed as a standard-ordered calculus
with respect to the chosen basis. For abbreviation we write
$e_{\vec{I}}$ for the ordered sequence of Lie derivatives in
\eqref{eq:VerticalDiffop} and set also $D^I = D^{i_1 \ldots i_N}$ for
a multiindex $I \in \mathbb{N}_0^N$. A formal series $D \in
\Diffop_\ver (U \times G)[[\lambda]]$ can then be written as formal
series $D = \sum_I D^I e_{\vec{I}}$ with $D^I \in \Cinfty(U \times
G)[[\lambda]]$ such that in each order of $\lambda$ only finitely many
differentiations occur. Up to here, this is even possible for an
arbitrary principal bundle, we do not yet need the trivialization.

Only in our local model we can define now 
\begin{equation}
    \label{eq:DeformedActionVertialDiffop}
    D \bullet' \phi
    =
    \sum\nolimits_I D^I \starred e_{\vec{I}} \, \phi
\end{equation}
for $D \in \Diffop(U \times G)[[\lambda]]$ and $\phi \in \Cinfty(U
\times G)[[\lambda]]$. For this action we have the following
properties:
\begin{lemma}
    \label{lemma:DeformedVerticalOperators}
    Let $D, \tilde{D} \in \Diffop_\ver(U \times G)[[\lambda]]$, $\phi
    \in \Cinfty(U \times G)[[\lambda]]$ and $u \in
    \Cinfty(M_\red)[[\lambda]]$.
    \begin{compactenum}
    \item \label{item:DeformedVerticalRightLinear} The definition
        \eqref{eq:DeformedActionVertialDiffop} yields a formal series
        of differential operators with
        \begin{equation}
            \label{eq:DCommutesWithu}
            (D \bullet' \phi) \bullet_\can u 
            =
            D \bullet' (\phi \bullet_\can u).
        \end{equation}
        Moreover, $\bullet'$ deforms the usual action of vertical
        differential operators.
    \item \label{item:StarPrime} There exists a unique $D \star'
        \tilde{D} \in \Diffop_\ver(U \times G)[[\lambda]]$ such that
        \begin{equation}
            \label{eq:DstarprimetildeD}
            (D \star' \tilde{D}) \bullet' \phi
            =
            D \bullet' (\tilde{D} \bullet' \phi).
        \end{equation}
    \item \label{item:StarPrimeDeformation} The product $\star'$ is
        the unique associative deformation of $\Diffop_\ver(U \times
        G)[[\lambda]]$ with the unique left module structure
        $\bullet'$ up to equivalence such that $\Cinfty(U \times
        G)[[\lambda]]$ becomes a $(\star', \starred)$-bimodule.
    \item \label{item:UniqueInvolution} There exists a uniquely
        determined $^*$-involution $D \mapsto D^*$ with respect to
        $\star'$ such that $\bullet'$ becomes a $^*$-representation of
        the pre Hilbert module $(\Cinftycf(U \times G)[[\lambda]],
        \SP{\cdot, \cdot}_\can)$.
    \end{compactenum}
\end{lemma}
\begin{proof}
    The property \eqref{eq:DCommutesWithu} is obvious by the
    associativity of $\starred$. Also $\bullet'$ deforms the usual
    action. For the second part we note that
    \[
    D \bullet' (\tilde{D} \bullet' \phi)
    =
    \sum_{I, J} D^I \starred e_{\vec{I}} 
    \left(\tilde{D}^J \starred e_{\vec{J}} \, \phi\right).
    \]
    Since the fundamental vector fields $\xi_{U \times G}$ are
    derivations of $\starred$ on $U \times G$, we have a Leibniz rule
    allowing to redistribute the $e_{\vec{I}}$ on the two factors
    $\tilde{D}^J$ and $e_{\vec{J}} \, \phi$. In the second result we
    have to reorder the Lie derivatives which gives after Lie
    algebraic combinatorics again linear combinations of $e_{\vec{K}}
    \, \phi$ with \emph{constant} coefficients. These can be viewed as
    acting by $\starred$ from the left whence in total we have by the
    associativity of $\starred$ a new $D \star' \tilde{D}$ acting via
    $\bullet'$ as wanted. Since by the first part the map $D \mapsto D
    \bullet'$ is injective, $D \star' \tilde{D}$ is uniquely
    determined. Note that for $i \le j$ we have
    \[
    \Lie_{(e_i)_{U \times G}} \star' \Lie_{(e_j)_{U \times G}} =
    \Lie_{(e_i)_{U \times G}} \Lie_{(e_j)_{U \times G}}
    \quad
    \textrm{and}
    \quad
    D^I \star' e_{\vec{I}} = D^I e_{\vec{I}}.
    \]
    Hence the fundamental vector fields and the functions $D^I \in
    \Cinfty(U \times G)[[\lambda]]$ generate (up to $\lambda$-adic
    completion) via $\star'$ all of $\Diffop_\ver(U \times
    G)[[\lambda]]$. The third part is clear from general
    considerations on principal bundles
    \cite{bordemann.neumaier.waldmann.weiss:2007a:pre}. For the last
    part we have to show that $D$ is adjointable with respect to
    $\SP{\cdot, \cdot}_\can$. We do this first for the generators. If
    $\xi \in \lie{g}$ then
    \begin{align*}
        \pi^* \SP{\phi, \Lie_{\xi_{U \times G}} \psi}_\can
        &=
        \int_G \mathsf{L}_{g^{-1}} \left(
            \cc{\phi} \starred \Lie_{\xi_{U \times G}} \psi
        \right)
        \Dleft g \\
        &=
        \int_G \mathsf{L}_{g^{-1}} \left(
            \Lie_{\xi_{U \times G}}
            \left(\cc{\phi} \starred \psi\right)
            -
            \cc{\Lie_{\xi_{U \times G}} \phi} \starred \psi
        \right)
        \Dleft g \\
        &=
        \int_G \mathsf{L}_{g^{-1}} \left(
            \cc{-\Delta(\xi) \phi - \Lie_{\xi_{U \times G}} \phi}
            \starred \psi
        \right)
        \Dleft g \\
        &=
        \pi^*\SP{
          \left(-\Delta(\xi) - \Lie_{\xi_{U \times G}}\right)
          \bullet' \phi,
          \psi
        }_\can.
    \end{align*}
    Thus we have an adjoint in $\Diffop_\ver(U \times
    G)[[\lambda]]$. Analogously, we have for a function $D^I \in
    \Cinfty(U \times G)[[\lambda]]$
    \begin{align*}
        \pi^*\SP{\phi, D^I \bullet' \psi}_\can
        &=
        \int_G \mathsf{L}_{g^{-1}} \left(
            \cc{\phi} \starred (D^I \starred \psi)
        \right)
        \Dleft g \\
        &=
        \int_G \mathsf{L}_{g^{-1}} \left(
            \cc{(\cc{D^I} \starred \phi)} \starred \psi
        \right)
        \Dleft g \\
        &=
        \pi^*\SP{\cc{D^I} \bullet' \phi, \psi}_\can,
    \end{align*}
    since $\starred$ is Hermitian. Successively using these two
    statements and the fact that these generate all vertical
    differential operators, proves that all $D \in \Diffop_\ver(U
    \times G)[[\lambda]]$ have an adjoint in the vertical differential
    operators. Thus the last part follows.
\end{proof}

On $C$ we consider now the following type of inner product: let
$\SP{\cdot, \cdot}'$ be a $\Cinfty(M_{\red})[[\lambda]]$-valued inner
product on $\Cinftycf(C)[[\lambda]]$ such that there exists a formal
series $B = B_0 + \lambda B_1 + \cdots$ of bidifferential operators on
$C$ with
\begin{equation}
    \label{eq:DeformationInnerProduct}
    \pi^* \SP{\phi, \psi}'
    = \int_G \mathsf{L}_{g^{-1}} 
    \left(B(\cc{\phi}, \psi)\right)
    \Dleft g,
\end{equation}
and $B_0(\cc{\phi}, \psi) = \cc{\phi}\psi$. In this case we call
$\SP{\cdot, \cdot}'$ a \emph{bidifferential deformation} of the
canonical classical inner product \eqref{eq:SPredclDef}. Note that $B$
is not uniquely determined by \eqref{eq:DeformationInnerProduct} since
we can still perform integrations by parts. For our local situation we
have now the following result:
\begin{lemma}
    \label{lemma:BidiffInnerProductLocally}
    Let $\SP{\cdot, \cdot}'$ be a bidifferential deformation of the
    canonical classical inner product on $\Cinftycf(U \times
    G)[[\lambda]]$.
    \begin{compactenum}
    \item \label{item:BIJ} There exist $B^{IJ} \in \Cinfty(U \times
        G)[[\lambda]]$ such that
        \begin{equation}
            \label{eq:BidiffInnerProductBIJ}
            \pi^*\SP{\phi, \psi}'
            =
            \sum_{I, J} \int_G \mathsf{L}_{g^{-1}}
            \left(
                e_{\vec{I}} \, (\cc{\phi}) 
                \starred B^{IJ}
                \starred e_{\vec{J}} \, (\psi)
            \right)
            \Dleft g,
        \end{equation}
        where the sum is infinite but in each order of $\lambda$ we
        have only finitely many differentiations.
    \item \label{item:SPprimeHSPcan} There exists a vertical
        differential operator $H \in \Diffop_\ver(U \times
        G)[[\lambda]]$ such that $H = H^* = \id + \cdots$ and
        \begin{equation}
            \label{eq:SPprimeHSPcan}
            \SP{\phi, \psi}' = \SP{\phi, H \bullet' \psi}_\can.
        \end{equation}
    \item \label{item:SPprimeIsometricSPcan} $\SP{\cdot, \cdot}'$ is
        isometric to $\SP{\cdot, \cdot}_\can$.
    \end{compactenum}
\end{lemma}
\begin{proof}
    For the first part consider $\phi = u \otimes \chi$ and $\psi = v
    \otimes \tilde{\chi}$ with $u, v \in \Cinfty(U)$ and $\chi,
    \tilde{\chi} \in \Cinfty_0(G)$. Using the right
    $\starred$-linearity of $\SP{\cdot, \cdot}'$ we get
    \[
    \int_G \mathsf{L}_{g^{-1}}
    \left(B(\cc{u \otimes \chi}, v \otimes \tilde{\chi})\right)
    \Dleft g
    =
    \pi^*(\cc{u}) \starred 
    \int_G \mathsf{L}_{g^{-1}}
    \left(B(1 \otimes \cc{\chi}, 1 \otimes \tilde{\chi})\right)
    \Dleft g
    \starred \pi^*(v).
    \]
    Now in $B(1 \otimes \cc{\chi}, 1 \otimes \tilde{\chi})$ only
    vertical differentiations can contribute. Hence we have
    \[
    B(1 \otimes \cc{\chi}, 1 \otimes \tilde{\chi})
    =
    \sum_{I, J} e_{\vec{I}} \, (1 \otimes \cc{\chi})
    B^{IJ}
    e_{\vec{J}} \, (1 \otimes \tilde{\chi})
    \]
    with formal series $B^{IJ} \in \Cinfty(U \times G)[[\lambda]]$
    such that in each order of $\lambda$ only finitely many
    differentiations occur. Since $\pi^*(\cc{u})$ and $\pi^*v$ do not
    depend on the group variables and since the fundamental vector
    fields are derivations of $\starred$ we arrive at the formula
    \[
    \pi^*\SP{u \otimes \chi, v \otimes \tilde{\chi}}'
    =
    \sum_{I, J} 
    \int_G \mathsf{L}_{g^{-1}} \left(
        \cc{e_{\vec{I}} \, (u \otimes \chi)}
        \starred B^{IJ} \starred
        e_{\vec{J}} \, (v \otimes \tilde{\chi})
    \right)
    \Dleft g.
    \]
    Now in each order of $\lambda$ we have an integration and
    bidifferential operators. By the usual continuity and density
    argument, they are already determined on their values on
    factorizing functions $u \otimes \chi$ and $v \otimes
    \tilde{\chi}$, respectively. Thus \eqref{eq:BidiffInnerProductBIJ}
    holds in general showing the first part. Since the $e_{\vec{I}}$
    are real differential operators, we can rewrite this as
    \[
    \SP{\phi, \psi}'
    = \sum_{I, J} 
    \SP{e_{\vec{I}} \bullet' \phi,
      (B^{IJ} e_{\vec{J}}) \bullet' \psi}_\can
    =
    \SP{\phi, 
      \left(
          \sum_{I, J} e_{\vec{I}}^* \star' (B^{IJ} e_{\vec{J}})
      \right)
      \bullet' \psi
    }_\can.
    \]
    This yields the vertical differential operator $H \in
    \Diffop_\ver(U \times G)[[\lambda]]$. From $\SP{\phi, H \bullet'
      \psi}_\can = \SP{\phi, \psi}' = \cc{\SP{\psi, \phi}'} =
    \cc{\SP{\psi, H \bullet' \phi}_\can} = \SP{H \bullet' \phi,
      \psi}_\can$ we see $H = H^*$. Finally, $H = \id + \cdots$ is
    clear giving the second part. The third part follows as we have a
    Hermitian $\star'$-square root $\sqrt[\star']{H} \star'
    \sqrt[\star']{H} = H$ which implements the unitary map between the
    two inner products by left $\bullet'$-multiplication.
\end{proof}
\begin{corollary}
    \label{corollary:LocallyAllCompletelyPositive}
    Every bidifferential deformation of $\SP{\cdot, \cdot}_\can^\cl$
    on $\Cinftycf(U \times G)[[\lambda]]$ is completely positive.
\end{corollary}

After these local constructions we shall now pass to the global
situation. The next proposition gives the existence of
$\Cinfty(M_\red)[[\lambda]]$-valued inner products which deform the
canonical classical one in a bidifferential way. Of course, our inner
product $\SP{\cdot, \cdot}_\red$ is of this form. However, we give an
independent proof not relying on phase space reduction thereby
including non-connected Lie groups $G$ as well.
\begin{proposition}
    \label{proposition:ExistenceInnerProducts}
    Let $C \circlearrowright G \longrightarrow M_\red$ be an arbitrary
    principal bundle. Then there exists a bidifferential deformation
    $\SP{\cdot, \cdot}$ of the canonical classical inner product on
    $\Cinftycf(C)[[\lambda]]$ with the additional feature
    \begin{equation}
        \label{eq:ExistenceOfGinvariantInnerProduct}
        \SP{\mathsf{L}_{g^{-1}} \phi, \mathsf{L}_{g^{-1}} \psi}
        =
        \Delta(g) \SP{\phi, \psi}
    \end{equation}
    for all $\phi, \psi \in \Cinftycf(C)[[\lambda]]$ and $g \in G$.
\end{proposition}
\begin{proof}
    Let $\{U_\alpha, \Phi_\alpha\}$ be again a locally finite atlas of
    trivializations and let $\{\chi_\alpha\}$ be a subordinate
    quadratic partition of unity on $M_\red$, i.e. $\supp \chi_\alpha
    \subseteq U_\alpha$ and $\sum_\alpha \cc{\chi_\alpha} \chi_\alpha
    = 1$. The global right module structure $\bulletred$ of
    $\Cinfty(C)[[\lambda]]$ restricts to $\pi^{-1}(U_\alpha)$ and, via
    $\Phi_\alpha$ we obtain a right module structure $\bullet_\alpha$
    on each $\Cinfty(U_\alpha \times G)[[\lambda]]$, i.e. we have
    \[
    \phi \bullet_\alpha u = 
    (\Phi_\alpha^*)^{-1}\left(\Phi_\alpha^* \phi \bulletred u\right)
    \]
    for $\phi \in \Cinfty(U_\alpha \times G)[[\lambda]]$ and $u \in
    \Cinfty(U_\alpha)[[\lambda]]$. By the uniqueness of the right
    module structure we find a $G$-equivariant formal series of
    differential operators $T_\alpha = \id + \sum_{r=1}^\infty
    \lambda^r T_\alpha^{(r)}$ on $U_\alpha \times G$ such that
    \[
    T_\alpha (\phi \bullet_\alpha u)
    = T_\alpha (\phi) \starred \pi^*u.
    \]
    Here we use again $\starred$ also for $\Cinfty(U_\alpha \times
    G)[[\lambda]]$ making $\pi^*$ a star product homomorphism as in
    \eqref{eq:piAlgebraHomomorphism}. This way, we define an inner
    product on $\Cinftycf(C)[[\lambda]]$ by
    \[
    \SP{\phi, \psi}
    =
    \sum_\alpha
    \SP{
      \pi^* \chi_\alpha 
      \starred T_\alpha ((\Phi_\alpha^*)^{-1}\phi),
      \pi^* \chi_\alpha 
      \starred T_\alpha ((\Phi_\alpha^*)^{-1}\psi)
    }_\can.
    \tag{$*$}
    \]
    Indeed, $\SP{\phi, \psi} \in \Cinfty(M_\red)[[\lambda]]$ is
    well-defined since $T_\alpha ((\Phi_\alpha^*)^{-1} \phi)$,
    $T_\alpha ((\Phi_\alpha^*)^{-1} \psi) \in \Cinftycf(U_\alpha
    \times G)[[\lambda]]$ become globally defined functions on $M_\red
    \times G$ after multiplying with $\pi^*\chi_\alpha$ thanks to
    $\supp \chi_\alpha \subseteq U_\alpha$ and the fact that
    $\starred$ is bidifferential. Then each term in the above sum has
    support in the appropriate $U_\alpha$. By the local finiteness of
    the cover, ($*$) is well-defined and smooth in each order of
    $\lambda$. The $\mathbb{C}[[\lambda]]$-sesquilinearity and the
    symmetry under complex conjugation is clear as $\SP{\cdot,
      \cdot}_\can$ has these features and all involved maps are
    $\mathbb{C}[[\lambda]]$-linear. Now let $u \in
    \Cinfty(M_\red)[[\lambda]]$ then we have
    \begin{align*}
        \pi^* \chi_\alpha \starred
        T_\alpha \left((\Phi_\alpha^*)^{-1} (\psi \bulletred u)\right)
        &=
        \pi^* \chi_\alpha \starred 
        T_\alpha
        \left((\Phi_\alpha^*)^{-1} \psi \bullet_\alpha u)\right) \\
        &=
        \pi^* \chi_\alpha \starred 
        \left(
            T_\alpha ((\Phi_\alpha^*)^{-1} \psi) \starred \pi^*u
        \right) \\
        &=
        \left(
            \pi^* \chi_\alpha \starred 
            T_\alpha ((\Phi_\alpha^*)^{-1} \psi)
        \right)
        \starred \pi^*u.
    \end{align*}
    Since $\SP{\cdot, \cdot}_\can$ is right $\starred$-linear in the
    second argument we deduce $\SP{\phi, \psi \bulletred u} =
    \SP{\phi, \psi} \starred u$. Thus $\SP{\cdot, \cdot}$ is indeed a
    valid inner product. We compute its classical limit. Since
    $T_\alpha$ is the identity in the zeroth order of $\lambda$ we get
    \begin{align*}
        \SP{\phi, \psi}
        &=
        \sum_\alpha
        \SP{\pi^* \chi_\alpha (\Phi_\alpha^*)^{-1} \phi,
          \pi^* \chi_\alpha (\Phi_\alpha^*)^{-1} \psi}_\can + \cdots
        \\
        &=
        \sum_\alpha
        \int_G \mathsf{L}_{g^{-1}}
        \left(
            (\Phi_\alpha^*)^{-1} \left(
                \cc{\tilde{\chi}_\alpha \phi}
                \tilde{\chi}_\alpha \psi
            \right)
        \right)
        \Dleft g
        + \cdots \\
        &=
        \sum_\alpha (\Phi_\alpha^*)^{-1}
        \left(
            \int_G \mathsf{L}_{g^{-1}}
            \left(
                \cc{\tilde{\chi}_\alpha \phi}
                \tilde{\chi}_\alpha \psi
            \right)
            \Dleft g
        \right)
        + \cdots \\
        &=
        \int_G \mathsf{L}_{g^{-1}}
        \left(
            \cc{\phi} \psi
        \right)
        \Dleft g
        + \cdots,
    \end{align*}
    where we set $\tilde{\chi}_\alpha = \Phi_\alpha^* \pi^*
    \chi_\alpha$ which yields a partition of unity on $C$ subordinate
    to the cover $\{\pi^{-1}(U_\alpha)\}$ with $\sum_\alpha
    \cc{\tilde{\chi}_\alpha} \tilde{\chi}_\alpha = 1$. The last
    equation follows since the integral is already an invariant
    function which can directly be identified with a function on
    $U_\alpha$, not needing the trivialization anymore. Thus we have a
    deformation as wanted. Finally, let $g \in G$. Since all the maps
    $T_\alpha$ and $\Phi_\alpha^*$ are equivariant and
    $\mathsf{L}_{g^{-1}} \pi^* \chi_\alpha = \pi^* \chi_\alpha$, we
    get
    \[
    \pi^*\chi_\alpha \starred
    T_\alpha ((\Phi_\alpha^*)^{-1}(\mathsf{L}_{g^{-1}} \psi))
    =
    \pi^*\chi_\alpha \starred
    \left(
        \mathsf{L}_{g^{-1}} T_\alpha ((\Phi_\alpha^*)^{-1}\phi)
    \right)
    =
    \mathsf{L}_{g^{-1}} \left(
        \pi^*\chi_\alpha \starred
        T_\alpha ((\Phi_\alpha^*)^{-1}\phi)
    \right).
    \]
    Now the locally defined $\SP{\cdot, \cdot}_\can$ has the property
    \eqref{eq:ExistenceOfGinvariantInnerProduct} and hence $\SP{\cdot,
      \cdot}$ inherits this since every term in ($*$) satisfies
    \eqref{eq:ExistenceOfGinvariantInnerProduct}.
\end{proof}

In the last step, we show three things: for a given bidifferential
deformation $\SP{\cdot, \cdot}$ of the classical canonical inner
product the vertical differential operators act in an adjointable way,
all such deformations are completely positive and isometric:
\begin{theorem}
    \label{theorem:CPOfAllInnerProducts}
    Let $C \circlearrowright G \longrightarrow M_\red$ be an arbitrary
    principal bundle. Moreover, let
    \begin{equation}
        \label{eq:ArbitrarySP}
        \SP{\cdot, \cdot}:
        \Cinftycf(C)[[\lambda]] \times \Cinftycf(C)[[\lambda]]
        \longrightarrow
        \Cinfty(M_\red)[[\lambda]]
    \end{equation}
    be a bidifferential deformation of the canonical classical inner
    product with respect to a given right module structure
    $\bulletred$. Moreover, let $\star'$ be a corresponding choice of
    a deformation of the vertical differential operators
    $\Diffop_\ver(C)[[\lambda]]$ with left module structure
    $\bullet'$.
    \begin{compactenum}
    \item \label{item:UniqueInvolutionForDiffopVer} There exists a
        unique $^*$-involution for $(\Diffop_\ver(C)[[\lambda]],
        \star')$ deforming the classical one such that $\bullet'$
        becomes a $^*$-representation with respect to $\SP{\cdot,
          \cdot}$.
    \item \label{item:AllInnerProductsAreCP} The inner product
        $\SP{\cdot, \cdot}$ is completely positive.
    \item \label{item:AllInnerProductsIsometric} Any two deformations
        are isometrically isomorphic via the left
        $\bullet'$-multiplication of some $V = \id + \sum_{r=1}^\infty
        \lambda^r V_r \in \Diffop_\ver(C)[[\lambda]]$.
    \end{compactenum}
\end{theorem}
\begin{proof}
    As before, we choose a locally finite atlas $\{(U_\alpha,
    \Phi_\alpha)\}$ of trivializations and a subordinate partition of
    unity $\{\chi_\alpha\}$. For $D \in \Diffop_\ver(C)[[\lambda]]$
    and $\phi, \psi \in \Cinftycf(C)[[\lambda]]$ we have
    \[
    \pi^* \SP{\phi, D \bullet' \psi}
    =
    \int_G \mathsf{L}^*_{g^{-1}}
    \left(B(\cc{\phi}, D \bullet' \psi)\right) \Dleft g
    =
    \sum_\alpha \int_G \mathsf{L}^*_{g^{-1}}
    \left(B
        \left(
            \cc{\phi}, (\pi^*\chi_\alpha D) \bullet' \psi
        \right)
    \right)
    \Dleft g.
    \]
    Since $B$ is bidifferential and $\bullet'$ is also local, we have
    \[
    \supp B(\cc{\phi}, (\pi^*\chi_\alpha D) \bullet' \psi)
    \subseteq
    \supp \pi^*\chi_\alpha
    \subseteq \pi^{-1} (U_\alpha).
    \]
    It follows that $\SP{\phi, (\pi^*\chi_\alpha D) \bullet' \psi}$ is
    given by the restriction of $\SP{\cdot, \cdot}$ to
    $\Cinftycf(\pi^{-1}(U_\alpha))[[\lambda]]$ evaluated on the
    restrictions of $\phi$ and $(\pi^*\chi_\alpha D) \bullet' \psi$,
    respectively. Here we can apply
    Lemma~\ref{lemma:BidiffInnerProductLocally},
    \refitem{item:SPprimeIsometricSPcan}, and find an isometry
    $V_\alpha = \id + \sum_{r=1}^\infty \lambda^r V_\alpha^{(r)} \in
    \Diffop_\ver(\pi^{-1}(U_\alpha))[[\lambda]]$ such that
    \[
    \SP{\phi, (\pi^*\chi_\alpha D) \bullet' \psi}
    =
    \SP{
      V_\alpha \bullet' \phi,
      V_\alpha \bullet' ((\pi^*\chi_\alpha D) \bullet' \psi)
    }_\can.
    \]
    With respect to the locally defined canonical inner product, the
    action of the vertical differential operators is adjointable
    according to Lemma~\ref{lemma:DeformedVerticalOperators}: there we
    have shown this for a particular choice of $\bullet'$ but all
    these choices are equivalent which allows to transport the
    $^*$-involution from the particular choice to any other $\star'$
    and $\bullet'$. This way, we get a locally defined $^*$-involution
    $^{*_\alpha}$ for $\star'$ compatible with $\bullet'$ and
    $\SP{\cdot, \cdot}_\can$. Using the invertibility of $H_\alpha =
    V_\alpha^{*_\alpha} \star' V_\alpha = \id + \cdots$ as before, we
    get
    \begin{align*}
        \SP{\phi, (\pi^*\chi_\alpha D) \bullet' \psi}
        &=
        \SP{V_\alpha \bullet' \phi,
          V_\alpha \bullet' (\pi^*\chi_\alpha D)\bullet' \psi}_\can
        \\
        &=
        \SP{H_\alpha \bullet' \phi,
          (\pi^*\chi_\alpha D)\bullet' \psi}_\can
        \\ 
        &=
        \SP{
          \left(
              (\pi^*\chi_\alpha D)^{*_\alpha} \star' H_\alpha
          \right) \bullet' \phi,
          \psi}_\can
        \\
        &=
        \SP{V_\alpha \bullet' \left(
              H_\alpha^{-1}
              \star' (\pi^*\chi_\alpha D)^{*_\alpha}
              \star' H_\alpha
          \right)
          \bullet' \phi, 
          V_\alpha \bullet' \psi}_\can
        \\
        &=
        \SP{
          \left(
              H_\alpha^{-1}
              \star' (\pi^*\chi_\alpha D)^{*_\alpha}
              \star' H_\alpha
          \right)
          \bullet' \phi, 
          \psi}.
    \end{align*}
    Since all the operations $\star'$ and $\bullet'$ preserve the
    supports we can finally take the sum over all $\alpha$ and get
    \[
    \SP{\phi, D \bullet' \psi}
    =
    \sum_\alpha \SP{\phi, (\pi^*\chi_\alpha D) \bullet' \psi}
    =
    \SP{\left(
          \sum_\alpha
          H_\alpha^{-1} \star' (\pi^*\chi_\alpha D) \star' H_\alpha
      \right) \bullet' \phi, \psi}
    =
    \SP{D^* \bullet \phi, \psi},
    \]
    with $D^* \in \Diffop_\ver(C)[[\lambda]]$ according to the term
    before. This shows that we indeed obtain an adjoint for the left
    action of $D$. Since $\SP{\cdot, \cdot}$ is in zeroth order just
    the canonical classical inner product, the classical limit of the
    $^*$-involution is the classical $^*$-involution. Since
    $\SP{\cdot, \cdot}$ is non-degenerate and $D \mapsto (\phi \mapsto
    D \bullet' \phi)$ is injective, the $^*$-involution is necessarily
    unique, proving the first part. The second part is now very easy:
    using a quadratic partition of unity $\sum_\alpha \cc{\chi_\alpha}
    \chi_\alpha = 1$ subordinate to the above atlas, we obtain
    vertical differential operators $\deform{\chi}_\alpha \in
    \Diffop_\ver(C)[[\lambda]]$ with
    \[
    \sum_\alpha \deform{\chi}_\alpha^* \star' \deform{\chi}_\alpha
    = \id,
    \]
    with $\supp \deform{\chi}_\alpha \subseteq \pi^{-1}(U_\alpha)$ and
    $\deform{\chi}_\alpha = \pi^* \chi_\alpha + \cdots$. Indeed, the
    vertical differential operator $X = \sum_\alpha
    (\pi^*\chi_\alpha)^* \star' \pi^*\chi_\alpha = \id + \cdots$ is
    Hermitian and starts with the identity, since the classical limit
    of the $^*$-involution is just the complex conjugation on
    $\pi^*\chi_\alpha$, viewed as vertical differential operator. Thus
    $\sqrt[\star']{X} = \id + \cdots$ is well-defined and invertible.
    Then $\deform{\chi}_\alpha = \pi^*\chi_\alpha \star'
    \frac{1}{\sqrt[\star']{X}}$ will do the job. Using this, we get
    for $\phi, \psi \in \Cinftycf(C)[[\lambda]]$
    \[
    \SP{\phi, \psi}
    =
    \SP{\phi, \sum_\alpha
      \left(
          \deform{\chi}_\alpha^* \star' \deform{\chi}_\alpha
      \right)
      \bullet' \psi}
    =
    \sum_\alpha
    \SP{\deform{\chi}_\alpha \bullet \phi,
      \deform{\chi}_\alpha \bullet \psi}
    \tag{$*$}
    \]
    with $\deform{\chi}_\alpha \bullet \phi, \deform{\chi}_\alpha
    \bullet \psi \in \Cinftycf(\pi^{-1}(U_\alpha))[[\lambda]]$. Here
    we can apply
    Corollary~\ref{corollary:LocallyAllCompletelyPositive} to get the
    complete positivity locally, and, since we have a (locally finite)
    convex sum in ($*$), also globally. Thus the second part follows.
    For the third, let $\SP{\cdot, \cdot}'$ be another inner product.
    Then they are isometric on $\pi^{-1}(U_\alpha)$ via some isometry
    $V_\alpha = \id + \cdots \in
    \Diffop_\ver(\pi^{-1}(U_\alpha))[[\lambda]]$, i.e.  $\SP{\phi,
      \psi} = \SP{V_\alpha \bullet' \phi, V_\alpha \bullet' \psi}'$
    for all $\phi, \psi \in \Cinftycf(\pi^{-1}(U_\alpha))[[\lambda]]$,
    according to Lemma~\ref{lemma:BidiffInnerProductLocally}. We apply
    this to ($*$) and get for arbitrary $\phi, \psi \in
    \Cinftycf(C)[[\lambda]]$
    \begin{align*}
        \SP{\phi, \psi}
        &=
        \sum_\alpha
        \SP{\deform{\chi}_\alpha \bullet' \phi,
          \deform{\chi}_\alpha \bullet' \psi} \\
        &=
        \sum_\alpha \SP{
          V_\alpha \bullet' (\deform{\chi}_\alpha \bullet' \phi),
          V_\alpha \bullet' (\deform{\chi}_\alpha \bullet' \psi)
        }' \\
        &=
        \SP{
          \phi,
          \left(
              \sum_\alpha
              \deform{\chi}_\alpha^{*'} \star' V_\alpha^{*'}
              \star' V_\alpha \star' \deform{\chi}_\alpha
          \right)
          \bullet'
          \psi}' \\
        &=
        \SP{\phi, H \bullet' \psi}',
    \end{align*}
    with some $H \in \Diffop_\ver(C)[[\lambda]]$ given explicitly by
    the locally finite sum
    \[
    H = \sum_\alpha
    \deform{\chi}_\alpha^{*'} \star' V_\alpha^{*'}
    \star' V_\alpha \star' \deform{\chi}_\alpha,
    \]
    where $^{*'}$ denotes the $^*$-involution induced by $\SP{\cdot,
      \cdot}'$ according to the first part. From the construction it
    is clear that $H = \id + \cdots$. Thus we have $H = V^{*'} \star'
    V$ with some $V = \id + \cdots \in \Diffop_\ver(C)[[\lambda]]$
    which is the isometry we are looking for.
\end{proof}
\begin{remark}
    \label{remark:Rigidity}
    This result, together with the existence according to
    Proposition~\ref{proposition:ExistenceInnerProducts}, can be seen
    as an extension of the (rigidity) results from
    \cite{bordemann.neumaier.waldmann.weiss:2007a:pre} on the
    existence and uniqueness of the right module structure
    $\bulletred$: also the canonical classical inner product allows
    for an essentially unique deformation preserving complete
    positivity and the adjointability of the vertical differential
    operators. Note also, that the above construction is independent
    of the phase space reduction approach, which also gives existence
    of an inner product but no proof for uniqueness. Moreover, in the
    phase space reduction approach we are restricted to principal
    bundles arising from \emph{connected} groups.
\end{remark}

%
%

\subsection{A strong Morita equivalence bimodule}
\label{subsec:StrongMoritaBimodule}

We can now formulate the main result of this section, the quantized
version of Theorem~\ref{theorem:ClassicalLimitBimodule}:
\begin{theorem}
    \label{theorem:QuantizedSMEBimodule} Let $C \circlearrowright G
    \longrightarrow M_\red$ be an arbitrary principal bundle and
    $\SP{\cdot, \cdot}$ a bidifferential deformation of $\SP{\cdot,
      \cdot}_\red^\cl$.
    \begin{compactenum}
    \item \label{item:FullCPInnerProduct} The inner product
        $\SP{\cdot, \cdot}$ is full, completely positive, and there is
        a $\deform{e} \in \Cinftycf(C)[[\lambda]]$ with
        \begin{equation}
            \label{eq:SPdeformedeformeEins}
            \SP{\deform{e}, \deform{e}} = 1.
        \end{equation}
    \item \label{item:ThetaCPAsWell} The canonical inner product
        $\Theta_{\cdot, \cdot}$ with values in the finite rank
        operators $\Finite(\Cinftycf(C)[[\lambda]])$ is completely
        positive as well.
    \item \label{item:CinftycfSMEBimodule} $\Cinftycf(C)[[\lambda]]$
        is a strong Morita equivalence bimodule for
        $\Finite(\Cinftycf(C)[[\lambda]])$ and
        $\Cinfty(M_\red)[[\lambda]]$ deforming the classical strong
        Morita equivalence bimodule $\Cinftycf(C)$ from
        Theorem~\ref{theorem:ClassicalLimitBimodule}.
    \end{compactenum}
\end{theorem}
\begin{proof}
    The first part is now clear from
    Theorem~\ref{theorem:CPOfAllInnerProducts},
    \refitem{item:AllInnerProductsAreCP}, and an argument analogous to
    the one in Proposition~\ref{proposition:StarEquivalenceBimodule}.
    Then the second part follows as in
    Proposition~\ref{proposition:StarEquivalenceBimodule}, too, which
    gives the last part immediately.
\end{proof}

The deformed vertical differential operators
$\Diffop_\ver(C)[[\lambda]]$ are not (strongly) Morita equivalent to
$\Cinfty(M_\red)[[\lambda]]$, neither is $(\Cinfty(M)[[\lambda]],
\star)$, see Remark~\ref{remark:MNotFiniteRank}, in the case of phase
space reduction. On the other hand, these algebras are not very far
away from being strongly Morita equivalent to
$\Cinfty(M_\red)[[\lambda]]$, since we have a strong Morita
equivalence bimodule and a $^*$-homomorphism into the adjointable
operators. The only flaw is that this $^*$-homomorphism does not map
into the finite rank operators. Note that in the case of
$\Diffop_\ver(C)[[\lambda]]$ it is even injective, while for
$\Cinfty(M)[[\lambda]]$ we clearly loose the functions with vanishing
infinite jet at $C$.

Again, we have a very rigid situation for the deformation of the inner
products and the bimodule structure as already for the strong Morita
equivalence bimodules in deformation quantization of \emph{unital}
algebras, see \cite{bursztyn.waldmann:2002a}. In our case, the crucial
new feature is that one of the algebras is non-unital.

\begin{remark}[Rieffel induction]
    \label{remark:RieffelInduction}
    Having the strong Morita equivalence bimodule we obtain by
    \emph{Rieffel induction} an equivalence of categories
    \begin{equation}
        \label{eq:EquivalenceByRieffelInduction}
        \begin{split}
            &\left(\Cinftycf(C)[[\lambda]], \bulletred\right)
            \itensor_{\left(\Cinfty(M_\red)[[\lambda]], \starred\right)}
            \; \cdot\; : \\
            &\qquad\qquad\qquad
            \Rep[\mathcal{D}]
            \left(\Cinfty(M_\red)[[\lambda]], \starred\right)
            \longrightarrow
            \Rep[\mathcal{D}]
            \left(
                \Finite\left(
                    \Cinftycf(C)[[\lambda]], \bulletred
                \right)
            \right)
        \end{split}
    \end{equation}
    for every coefficient $^*$-algebra $\mathcal{D}$, see
    \eqref{eq:RieffelInductions}. Moreover, since also the
    $^*$-algebra $(\Cinfty(M)[[\lambda]], \star)$ acts on
    $\Cinftycf(C)[[\lambda]]$ via $\bullet$ in an adjointable way
    thanks to Proposition~\ref{proposition:LeftModuleIsRepresentation}
    we obtain also a Rieffel induction functor
    \begin{equation}
        \label{eq:RieffelInduction}
        \begin{split}
            &\left(\Cinftycf(C)[[\lambda]], \bulletred\right)
            \itensor_{\left(\Cinfty(M_\red)[[\lambda]], \starred\right)}
            \; \cdot\; : \\
            &\qquad\qquad\qquad
            \Rep[\mathcal{D}]
            \left(\Cinfty(M_\red)[[\lambda]], \starred\right)
            \longrightarrow
            \Rep[\mathcal{D}]
            \left(\Cinfty(M)[[\lambda]], \star\right).
        \end{split}
    \end{equation}
    However, in general this will not be an equivalence of categories
    anymore. The reason is clear from geometric considerations:
    Indeed, the image of a $^*$-representation of
    $\Cinfty(M_\red)[[\lambda]]$ under \eqref{eq:RieffelInduction} is
    somehow located on $C$, in the sense that if $f$ vanishes on $C$
    up to infinite order, then the action of $f$ in an induced
    representation is necessarily trivial. This is clear from the
    bidifferentiality of the left module structure $\bullet$. On the
    other hand, $\Cinfty(M)[[\lambda]]$ does have non-trivial
    $^*$-representations located away from $C$: we can take any
    $\delta$-functional at $p \in M \setminus C$ and deform it into a
    positive functional $\omega_p$ with support still be given by $p$.
    Then the GNS representation $\pi_{\omega_p}$ of $\omega_p$ is not
    the trivial representation. In fact, since the deformation
    $\omega_p = \delta_p \circ S$ is obtained by means of a formal
    series $S_p = \id + \sum_{r=1}^\infty \lambda^r S_r$ with
    differential operators $S_r$ vanishing on constants, a function $f
    \in \Cinfty(M)[[\lambda]]$ which is $1$ in an open neighbourhood
    of $p$ acts as identity operator on the GNS pre Hilbert space of
    $\omega_p$. However, considering a function $f$ which vanishes up
    to infinite order on $C$, we conclude that the GNS representation
    can not be in the image of \eqref{eq:RieffelInduction} up to
    unitary equivalence. This shows that \eqref{eq:RieffelInduction}
    will not be an equivalence of categories, see also
    Remark~\ref{remark:MNotFiniteRank}.
\end{remark}

%
%

\section{An example}
\label{sec:Example}

In this concluding section we consider the geometrically trivial
situation $M = M_\red \times T^*G$ where on $M_\red$ a Poisson bracket
and a corresponding star product $\starred$ is given while on $T^*G$
we use the canonical symplectic Poisson structure and the canonical
star product $\starG$ from \cite{gutt:1983a}. Then $M$ carries the
star product $\star = \starred \tensor \starG$. Classically, the phase
space reduction for the constraint hypersurface $C = M_\red \times G$
will just omit the factor $T^*G$ and reproduces $M_\red$.

%
%

\subsection{The reduction from $M_\red \times T^*G$ to $M_\red$}
\label{subsec:ReductionFromMredtimesTstarGtoMred}

Let $\iota: G \longrightarrow T^*G$ denote the zero section of the
cotangent bundle and $\pr: T^*G \longrightarrow G$ the bundle
projection. We use the same symbols for the corresponding maps $\iota:
C = M_\red \times G \longrightarrow M$ and $\pr: M \longrightarrow
M_\red \times G$. Then $C$ is clearly coisotropic in $M$ with
corresponding orbit space $M_\red$, reproducing the given Poisson
structure.  In principle, one does not need the group structure of $G$
for this coisotropic reduction; it would work literally the same for
any cotangent bundle. However, in view of our previous framework, we
shall outline the underlying symmetry structure.

In order to be conform with the local models described in
Section~\ref{sec:StrongMoritaEquivalenceBimodule} and the appendix we
choose the \emph{right} multiplications $\mathsf{r}: G \times G
\longrightarrow G$ as group action of $G$ on itself. The canonical
lift to a \emph{left} action on $T^*G$ is then denoted by $\mathsf{L}:
G \times T^*G \longrightarrow T^*G$, i.e. $\mathsf{L}_g = T^*
\mathsf{r}_g$. This extends to $M$ in the usual way yielding a Poisson
action of $G$ on $M$. The fundamental vector fields of $\mathsf{r}$
are the left invariant vector fields $X_\xi(g) = T_e \mathsf{l}_g
(\xi)$ for $\xi \in \lie{g}$ and $g \in G$. More precisely, $\xi_G =
\frac{\D}{\D t}\at{t=0} \mathsf{r}_{\exp(t\xi)}^{-1} = - X_\xi$ since
we defined the fundamental vector field with respect to the left
action, see \eqref{eq:FundamentalVectorField}.

An effective description of the corresponding fundamental vector
fields on $T^*G$ and $M$ are obtained as follows. To every vector
field $X \in \Secinfty(TG)$ we assign a fiberwisely linear function
$\mathcal{J}(X) \in \Pol^1(T^*G)$ on $T^*G$ by
$\mathcal{J}(X)(\alpha_g) = \alpha_g(X(g))$ where $\alpha_g \in
T^*_gG$. Then the canonical Poisson bracket on $T^*G$ of such linear
functions in the ``momenta'' is $\{\mathcal{J}(X), \mathcal{J}(Y)\} =
- \mathcal{J}([X, Y])$, see also
\cite[Sect.~3.3.1]{waldmann:2007a}. Then the fundamental vector fields
of the left action $\mathsf{L} = T^*\mathsf{r}$ are given by the
Hamiltonian vector fields $\xi_{T^*G} = - X_{\mathcal{J}(X_\xi)}$.
Thus the momentum map is given by $J(\xi) = - \mathcal{J}(X_\xi)$,
which induces also the trivialization of the global tubular
neighbourhood $M_\red \times T^*G$ of $M_\red \times G$.  The
prolongation with respect to this tubular neighbourhood according to
\eqref{eq:Prol} is then just the pull-back $\prol = \pr^*$.

To describe the classical Koszul operator and the homotopy more
explicitly, we make use of a vector space basis $e_1, \ldots, e_N \in
\lie{g}$ as before. We have the corresponding left invariant vector
fields $X_a = X_{e_a}$ yielding the linear functions $J_a = - P_a = -
\mathcal{J}(X_a) \in \Pol^1(T^*G)$ in the momenta.  For a one-form
$\theta \in \Secinfty(T^*G)$ we have the vertical lift $\theta^\ver
\in \Secinfty(T(T^*G))$ to a vertical vector field on $T^*G$. In
particular, the left invariant one-forms $\theta^a \in
\Secinfty(T^*G)$ with value $e^a$ at $e \in G$ lift to vertical vector
fields denoted by $\frac{\partial}{\partial J_a} = -
\frac{\partial}{\partial P_a} = - (\theta^a)^\ver$. Indeed, we have
$\frac{\partial}{\partial J_a} J_b = \delta^a_b$, explaining our
notation. The funny minus sign is due to our previous convention on
fundamental vector fields.

The classical Koszul operator will then be given by $\koszul x =
\ins(e^a)x J_a$ as before and the homotopy $h_0$ is explicitly and
globally given by
\begin{equation}
    \label{eq:ClassicalhNullGlobal}
    (h_0 f)(p, \alpha_g)
    =
    e^a \int_0^1 \frac{\partial f}{\partial J_a}(p, t\alpha_g) \D t,
\end{equation}
where $p \in M_\red$ and $\alpha_g \in T^*_gG$ as before.

%
%

\subsection{The canonical star product $\starG$ and its Schrödinger
  representation}
\label{subsec:CanonicalStarProductSchroedingerRepresentation}

On $T^*G$ there is a canonical star product $\starG$ which can be
obtained as follows, see \cite{gutt:1983a} as well as
\cite{bordemann.neumaier.waldmann:1998a}.

For left invariant vector fields $X_{\xi_1}, \ldots X_{\xi_k} \in
\Secinfty(T^*G)$ and a function $\phi \in \Cinfty(G)$ we define the
standard ordered quantization map $\stdrep$ by
\begin{equation}
    \label{eq:stdrepphiJxis}
    \stdrep\left(
        \pr^* \phi
        \mathcal{J}(X_{\xi_1}) \cdots \mathcal{J}(X_{\xi_k})
    \right) \psi
    =
    \frac{1}{k!} \left(\frac{\lambda}{\I}\right)^r \phi
    \sum_{\sigma \in S_k}
    \Lie_{X_{\xi_{\sigma(1)}}} \cdots \Lie_{X_{\xi_{\sigma(k)}}} \psi,
\end{equation}
where $\psi \in \Cinfty(G)$, and extend this to a
$\mathbb{C}[[\lambda]]$-linear map
\begin{equation}
    \label{eq:stdrep}
    \stdrep:
    \Pol^\bullet(T^*G)[[\lambda]]
    \longrightarrow
    \Diffop(G)[[\lambda]].
\end{equation}
Clearly, at this stage for polynomial functions we have convergence in
$\lambda$ for trivial reasons. Setting $\lambda = \hbar > 0$ yields a
symbol calculus for differential operators on $G$ and symbols on
$T^*G$ which are polynomial in the fibers.

Alternatively, the above quantization can also be written as
\begin{equation}
    \label{eq:stdredExplicit}
    \stdrep(f) \psi
    =
    \sum_{r=0}^\infty \frac{1}{r!} \left(\frac{\lambda}{\I}\right)^r
    \sum_{a_1, \ldots, a_r}
    \iota^*\left(
        \frac{\partial^r f}{\partial P_{a_1} \cdots \partial P_{a_r}}
    \right)
    \Lie_{X_{a_1}} \cdots \Lie_{X_{a_r}} \psi.
\end{equation}
Now it is clear that $\stdrep$ extends to a
$\mathbb{C}[[\lambda]]$-linear map $\stdrep: \Cinfty(T^*G)[[\lambda]]
\longrightarrow \Diffop(G)[[\lambda]]$ by the very same formula as
above.
\begin{remark}
    \label{remark:AlsoByConnection}
    The above symbol calculus can also be obtained from a covariant
    derivative, namely the ``half-commutator connection'' on $G$ which
    is defined by $\nabla_{X_\xi} X_\eta = \frac{1}{2} [X_\xi,
    X_\eta]$ on left invariant vector fields. This point of view was
    taken in \cite{bordemann.neumaier.waldmann:1998a}.
\end{remark}

Clearly, for $f, g \in \Pol^\bullet(T^*G)[[\lambda]]$ there is a
unique $f \starstd g \in \Pol^\bullet(T^*G)[[\lambda]]$ with
\begin{equation}
    \label{eq:stdstarDef}
    \stdrep(f \starstd g) = \stdrep(f) \stdrep(g).
\end{equation}
Moreover, this extends to a bidifferential star product for arbitrary
$f, g \in \Cinfty(T^*G)[[\lambda]]$ preserving
\eqref{eq:stdstarDef}. This star product is \emph{standard-ordered} in
the sense that $\pr^*\phi \starstd f = \pr^*\phi f$ for arbitrary
$\phi \in \Cinfty(G)[[\lambda]]$.

The only flaw of $\starstd$ is that it is not Hermitian. This can be
understood and cured as follows. First we introduce the differential
operator
\begin{equation}
    \label{eq:DeltaNull}
    \Delta_0 =  \Lie_{X_{P_a}} \Lie_{\frac{\partial}{\partial P_a}}
\end{equation}
acting on functions on $T^*G$, where as before $X_{P_a}$ is the
Hamiltonian vector field of the global momentum function
$P_a$. Clearly, this operator is independent of the chosen basis.
Moreover, we need the vertical lift of the modular one-form $\Delta$
which yields the vector field $\Delta^\ver = C_{ab}^b
\frac{\partial}{\partial P_a} \in \Secinfty(T(T^*G))$. Following
\cite{bordemann.neumaier.waldmann:1998a} we consider the formal series
of differential operators
\begin{equation}
    \label{eq:NeumaierOperator}
    N = \exp\left(
        \frac{\lambda}{2\I}\left(\Delta_0 - \Delta^\ver\right)
    \right)
\end{equation}
acting on $\Cinfty(T^*G)[[\lambda]]$. A non-trivial integration by
parts (even possible for arbitrary cotangent bundles
\cite{bordemann.neumaier.waldmann:1998a}) yields then the result
\begin{equation}
    \label{eq:AdjointOfStdRep}
    \int_G \cc{\phi} \stdrep(f) \psi \Dleft g
    =
    \int_G \cc{\stdrep(N^2\cc{f}) \phi} \: \psi \Dleft g
\end{equation}
for $\phi, \psi \in \Cinfty_0(G)[[\lambda]]$. From this failure of
$\stdrep$ being compatible with complex conjugation we see that the
definition
\begin{equation}
    \label{eq:StarG}
    f \starG g = N^{-1}(Nf \starstd Ng)
\end{equation}
yields again a bidifferential star product for which
\begin{equation}
    \label{eq:SchroedingerRep}
    \weylrep(f)\psi
    =
    \stdrep(Nf) \psi
    =
    \iota^*\left(Nf \starstd \pr^*\phi\right)
\end{equation}
defines a $^*$-representation on $\Cinfty_0(G)[[\lambda]]$. This is
the canonical Hermitian star product on $G$, originally constructed in
\cite{gutt:1983a}: there, $\starG$ was obtained from the observation
that $\Pol^\bullet(T^*G)^G \cong \Pol^\bullet(\lie{g}^*) \cong
\Sym^\bullet(\lie{g})$, using the PBW isomorphism to the universal
enveloping algebra of $\lie{g}$, and pulling back the product. In
fact, $\starG$ turns out to be strongly invariant and
$\Pol^\bullet(T^*G)^G[[\lambda]]$ forms a sub-algebra being isomorphic
to the ``formal'' universal enveloping algebra.  The representation
$\weylrep$ is also called the Schrödinger representation in Weyl
ordering since for the Lie group $G = \mathbb{R}^n$ this indeed
reproduces the usual canonical quantization in Weyl ordering.

%
%

\subsection{The bimodule structure on $\Cinftycf(M_\red \times
  G)[[\lambda]]$} 
\label{subsec:BimoduleStructureMredTimesG}

We will now use the strongly invariant star product $\star = \starred
\tensor \starG$ on $M = M_\red \times T^*G$. For the quantized Koszul
operator we will have the following result:
\begin{lemma}
    \label{lemma:iotaNqkoszulNull}
    Let $x \in \Cinfty(M, \Anti_{\mathbb{C}}^1 \lie{g})[[\lambda]]$.
    Then we have
    \begin{equation}
        \label{eq:iotaNqkoszulNull}
        \iota^* N \deform{\koszul} x = 0.
    \end{equation}
\end{lemma}
\begin{proof}
    First it is clear that the $M_\red$-components do not enter at
    all. Thus we can compute the left hand side of
    \eqref{eq:iotaNqkoszulNull} on $T^*G$ alone.  We have $\Delta_0
    J_a = 0$ and $\Delta^\ver J_a = - C_{ab}^b$ by the explicit form
    for $\Delta^\ver$ and $J_a = - P_a$. Thus $N J_a = J_a -
    \frac{\I\lambda}{2} C_{ab}^b$. Using $\deform{\koszul} x = x^a
    \starG J_a + \frac{\I\lambda}{2} C_{ab}^b x^a$ according to
    \eqref{eq:QuantizedKoszul} we get
    \[
    N \deform{\koszul} x
    = (Nx^a) \starstd (NJ_a) + \frac{\I\lambda}{2} C_{ab}^b Nx^a
    = (Nx^a) \starstd J_a.
    \]
    Note that at this point our choice $\kappa = \frac{1}{2}$ in
    \eqref{eq:QuantizedKoszul} enters again. Now $\stdrep$ is a symbol
    calculus where $f \in \Pol^\bullet(T^*G)$ corresponds to a
    differential operator $\stdrep(f)$ with $\stdrep(f) 1 = 0$ iff $f$
    has no contributions from polynomial degree $0$. This means that
    $(Nx^a) \starstd J_a$ is at least linear in the momenta, no matter
    what $x^a \in \Cinfty(T^*G)[[\lambda]]$ is. Thus
    \eqref{eq:iotaNqkoszulNull} follows.
\end{proof}
\begin{corollary}
    \label{corollary:DeformedRestrictionMapExplicit}
    For the deformed restriction map $\deform{\iota^*}$ we have
    \begin{equation}
        \label{eq:DeformedRestrictionMapExplicit}
        \deform{\iota^*}
        = \iota^* \circ \left(\id +
            \left(\deform{\koszul}_1- \koszul_1\right) h_0
        \right)^{-1}
        = \iota^* \circ N.
    \end{equation}
\end{corollary}
\begin{proof}
    By the general argument from
    \cite[Prop.~25]{bordemann.herbig.waldmann:2000a} we know that the
    deformed restriction map $\deform{\iota^*}$ is uniquely
    characterized by the following three properties: its classical
    limit is $\iota^*$, $\deform{\iota^*} \deform{\koszul}_1 = 0$ and
    $\deform{\iota^*} \prol = \id$. Clearly, $\iota^* \circ N$
    fulfills the first requirement. Also the last requirement is clear
    as the exponent of $N$ differentiates in momenta direction and
    hence vanishes on pull-backs $\pr^*\phi$. Finally, the second
    requirements is fulfilled by Lemma~\ref{lemma:iotaNqkoszulNull}.
\end{proof}

Thus we have computed the formal series of differential operators from
Lemma~\ref{lemma:RestrictionLocal} explicitly in this situation. Of
course, handling a formal series of differential operators like $N$ is
much easier that the non-local operator $(\id + (\deform{\koszul}_1 -
\koszul1))^{-1}$. We arrive at the following statement:
\begin{theorem}
    \label{theorem:BimoduleStructureExplicitly}
    Let $f \in \Cinfty(M_\red \times T^*G)[[\lambda]]$, $\phi \in
    \Cinfty(M_\red \times G)[[\lambda]]$, and $u, v \in
    \Cinfty(M_\red)[[\lambda]]$ be given.
    \begin{compactenum}
    \item \label{theorem:LeftModuleExplicit} The left module structure
        \eqref{eq:fbulletphi} is explicitly given by
        \begin{equation}
            \label{eq:LeftModuleExplicit}
            f \bullet \phi
            =
            \iota^*\left(
                Nf (\starred \tensor \starstd) \prol(\phi)
            \right)
            =
            \sum_{r=0}^\infty \frac{1}{r!}
            \left(\frac{\lambda}{\I}\right)^r
            \iota^*\left(
                \frac{\partial^r f}
                {\partial P_{a_1} \cdots \partial P_{a_r}}
            \right)
            \starred
            \Lie_{X_{a_1}} \cdots \Lie_{X_{a_r}} \phi,
        \end{equation}
        where $\starred$ is extended to $M_\red \times G$ as usual.
    \item \label{eq:ReducedStarProductExplicit} The reduced star
        product \eqref{eq:ReducedStarProduct} reproduces $\starred$ on
        $M_\red$. The right module structure
        \eqref{eq:ReducedRightModule} is explicitly given by
        \begin{equation}
            \label{eq:RightModuleExplicit}
            \phi \bulletred u = \phi \starred \pi^*u.
        \end{equation}
    \item \label{theorem:InnerProductExplicit} The inner product
        \eqref{eq:SPred} is explicitly given by
        \begin{equation}
            \label{eq:InnerProductExplicit}
            \SP{\phi, \psi}_\red(p) =
            \int_G
            \left(\cc{\phi} \starred \psi\right)(p, g)
            \Dleft g,
        \end{equation}
        where $p \in M_\red$ and $\phi, \psi \in \Cinftycf(M_\red
        \times G)[[\lambda]]$.
    \end{compactenum}
\end{theorem}
\begin{proof}
    For the first part we compute
    \begin{align*}
        f \bullet \phi
        &=
        \deform{\iota^*}(f \star \prol(\phi)) \\
        &=
        \iota^* N (f (\starred \tensor \starG) \prol(\phi)) \\
        &=
        \iota^*
        \left(Nf (\starred \tensor \starstd) \prol(\phi)\right) \\
        &=
        \sum_{r=0}^\infty \frac{1}{r!}
        \left(\frac{\lambda}{\I}\right)^r
        \iota^*\left(
            \frac{\partial^r f}
            {\partial P_{a_1} \cdots \partial P_{a_r}}
        \right)
        \starred
        \Lie_{X_{a_1}} \cdots \Lie_{X_{a_r}} \phi.
    \end{align*}
    The second part is clear since $\prol(\pi^*u) \star \prol(\pi^*v)
    = \prol(\pi^*) (\starred \tensor \starG) \prol(\pi^*v) =
    \prol(\pi^*(u \starred v))$. Indeed, the standard-ordered product
    as well as the canonical star product reduce to the pointwise
    product if both functions are independent of the momenta. A
    further application of $N$ yields nothing new for the same reason
    showing that our general construction reproduces $\starred$ as
    expected. For the right module structure we can argue
    similarly. Finally, for the third part we have
    \begin{align*}
        \pi^* \SP{\phi, \psi}_\red
        &=
        \int_G \mathsf{L}_{g^{-1}}^* \deform{\iota^*}
        \left(
            \cc{\prol(\phi)} \star \prol(\psi)
        \right) \Dleft g \\
        &=
        \int_G \mathsf{L}_{g^{-1}}^* \iota^* N
        \left(
            \prol\left(\cc{\phi} \starred \psi\right)
        \right) \Dleft g \\
        &=
        \int_G \mathsf{L}_{g^{-1}}^*
        \left(
            \cc{\phi} \starred \psi 
        \right) \Dleft g,
    \end{align*}
    since again $N$ and $\starG$ act trivially on functions not
    depending on the momenta.
\end{proof}

%
%

\subsection{Rieffel induction}
\label{subsec:RieffelInduction}

Having an explicit description of the bimodule structure and the inner
product we can compute the result of the corresponding Rieffel
induction as well.

To simplify things slightly, we will restrict to the following unital
$^*$-subalgebra $(\Cinfty(M_\red) \tensor[\mathbb{C}]
\Cinfty(T^*G))[[\lambda]]$ of $\Cinfty(M_\red \times
T^*G)[[\lambda]]$. Thanks to the factorization of the star product,
this is indeed a subalgebra. Moreover, it acts on $(\Cinfty(M_\red)
\tensor[\mathbb{C}] \Cinfty_0(G))[[\lambda]]$ of $\Cinftycf(M_\red
\times G)[[\lambda]]$ which becomes a bimodule for $(\Cinfty(M_\red)
\tensor[\mathbb{C}] \Cinfty(T^*G))[[\lambda]]$ from the left and
$\Cinfty(M_\red)[[\lambda]]$ from the right as before. Clearly, all
our previous results restrict well to this situation. Note that the
$\Cinfty(M_\red)[[\lambda]]$-valued inner product is still \emph{full}
when restricted to $(\Cinfty(M_\red) \tensor[\mathbb{C}]
\Cinfty_0(G))[[\lambda]]$. The reason why we restrict to this
subalgebra and this submodule is that the Rieffel induction functor
will have a very nice end explicit form here.
\begin{remark}
    \label{remark:StuffIsDense}
    The natural locally convex topologies of smooth functions (with
    compact support) make $\Cinfty(M_\red) \tensor[\mathbb{C}]
    \Cinfty_0(G)$ a dense subspace of $\Cinftycf(M_\red \times G)$ and
    similarly for $\Cinfty(M_\red) \tensor[\mathbb{C}]
    \Cinfty(T^*G)$. Thus, morally, the above restriction is not
    severe: as soon as one enters a more topological framework all the
    (hopefully continuous) structure maps should be determined by
    their behaviour on these dense subspaces. Of course, the
    $\lambda$-adic topology does not fit together well with the smooth
    function topology, at least in a naive way. Nevertheless, we
    consider this to be a technicality which may only cause artificial
    difficulties but no conceptual ones.
\end{remark}

The above simplification allows to re-interpret the factorizing case
in the following, purely algebraic way. Assume that $\mathcal{A}_\red$
and $\mathcal{B}$ are unital $^*$-algebras over $\ring{C}$ and
$\mathcal{A} = \mathcal{B} \tensor[\ring{C}] \mathcal{A}_\red$ is
their algebraic tensor product, again endowed with its canonical
unital $^*$-algebra structure. Assume moreover, that $\mathcal{B}_0
\subseteq \mathcal{B}$ is a $^*$-ideal and
\begin{equation}
    \label{eq:omegaOnBNull}
    \omega: \mathcal{B}_0 \longrightarrow \ring{C}
\end{equation}
is a positive linear functional with Gel'fand ideal
$\mathcal{J}_\omega \subseteq \mathcal{B}_0$. Then it is well-known
that the GNS representation of $\mathcal{B}_0$ on
$\mathcal{B}_0 \big/ \mathcal{J}_\omega$ extends to a
$^*$-representation of $\mathcal{B}$ on $\mathcal{B}_0 \big/
\mathcal{J}_\omega$ in the canonical way.
\begin{remark}
    \label{remark:OurExample}
    In our example we have $\mathcal{A}_\red =
    \Cinfty(M_\red)[[\lambda]]$ with $\starred$ and $\mathcal{B} =
    \Cinfty(T^*G)[[\lambda]]$ as well as $\mathcal{B}_0 =
    \Cinfty_0(T^*G)[[\lambda]]$. The positive functional $\omega$ is
    then the Schrödinger functional
    \begin{equation}
        \label{eq:SchroedingerFunctional}
        \omega (f)
        = \int_G \iota^*f \Dleft g
        \stackrel{(*)}{=} \SP{1, \weylrep(f) 1},
    \end{equation}
    see \cite[Prop.~7.1.35]{waldmann:2007a} for the justification of
    ($*$). Moreover, one knows that the GNS representation
    corresponding to $\omega$ reproduces the Schrödinger
    representation $\weylrep$ on $\Cinfty_0(G)[[\lambda]]$ with the
    usual $L^2$-inner product, see e.g.
    \cite[Satz~7.2.26]{waldmann:2007a} for a discussion and further
    references.
\end{remark}

We will now make use of the external tensor product of pre Hilbert
modules \cite[Sect.~4]{bursztyn.waldmann:2005b}: for two $^*$-algebras
$\mathcal{A}_i$ with $i = 1, 2$ and corresponding pre Hilbert right
$\mathcal{A}_i$-modules $\mathcal{E}_i$ one defines on $\mathcal{E}_1
\tensor \mathcal{E}_2$ an inner product by
\begin{equation}
    \label{eq:ExternalInnerProduct}
    \IP{x \tensor x', y \tensor y'}
    {}
    {\mathcal{E}_1 \tensor \mathcal{E}_2}
    {\mathcal{A}_1 \tensor \mathcal{A}_2}
    =
    \IP{x, y}{}{\mathcal{E}_1}{\mathcal{A}_1}
    \tensor
    \IP{x', y'}{}{\mathcal{E}_2}{\mathcal{A}_2}
\end{equation}
with values in the $^*$-algebra $\mathcal{A}_1 \tensor
\mathcal{A}_2$. It turns out that \eqref{eq:ExternalInnerProduct} is
again completely positive once both inner products $\IP{\cdot,
  \cdot}{}{\mathcal{E}_i}{\mathcal{A}_i}$ were completely positive,
see \cite[Remark~4.12]{bursztyn.waldmann:2005b}. However, it might
happen that \eqref{eq:ExternalInnerProduct} is degenerate. Thus the
external tensor product is defined analogously to the internal tensor
product \eqref{eq:InternalTensorProduct} as the quotient
\begin{equation}
    \label{eq:ExternalTensorProduct}
    \mathcal{E}_1 \extensor \mathcal{E}_2
    =
    \mathcal{E}_1 \tensor \mathcal{E}_2
    \big/ \left(
        \mathcal{E}_1 \tensor \mathcal{E}_2
    \right)^\bot
\end{equation}
in order to get again a non-degenerate inner product. Needless to say,
the construction of $\extensor$ is functorial in a good sense
similarly to the internal tensor product.

Now we take $\mathcal{A}_\red$ as a right $\mathcal{A}_\red$-module
with its canonical completely positive inner product $\SP{a, a'} =
a^*a'$. Then we can form the external tensor product with
$\mathcal{B}_0 \big/ \mathcal{J}_\omega$ endowed with its pre Hilbert
space structure. Thus we consider
\begin{equation}
    \label{eq:TheE}
    \mathcal{E}
    =
    \left(\mathcal{B}_0 \big/ \mathcal{J}_\omega\right)
    \extensor
    \mathcal{A}_\red.
\end{equation}
The completely positive inner product \eqref{eq:ExternalInnerProduct}
becomes on factorizing representatives in $\left(\mathcal{B}_0 \big/
    \mathcal{J}_\omega\right) \tensor \mathcal{A}_\red$
\begin{equation}
    \label{eq:ExternalInnerProductExplicit}
    \SP{[b] \tensor a, [b'] \tensor a'}_\omega
    =
    \omega(b^*b') a^*a',
\end{equation}
where $b, b' \in \mathcal{B}_0$ and $a, a' \in
\mathcal{A}_\red$. Typically, the degeneracy space of
\eqref{eq:ExternalInnerProductExplicit} will be trivial already whence
the quotient \eqref{eq:ExternalTensorProduct} is unnecessary.
\begin{lemma}
    \label{lemma:External}
    Let $\mathcal{A}_\red$ and $\mathcal{B}$ be unital $^*$-algebras,
    $\mathcal{B}_0 \subseteq \mathcal{B}$ a $^*$-ideal, $\omega:
    \mathcal{B}_0 \longrightarrow \ring{C}$ a positive linear
    functional with Gel'fand ideal $\mathcal{J}_\omega$, and
    $\mathcal{E} = \left(\mathcal{B}_0 \big/ \mathcal{J}_\omega\right)
    \extensor \mathcal{A}_\red$.
    \begin{compactenum}
    \item \label{item:ExternalIsStarRepForA} The pre Hilbert right
        $\mathcal{A}_\red$-module $\mathcal{E}$ carries a
        $^*$-representation of $\mathcal{A} = \mathcal{B} \tensor
        \mathcal{A}_\red$ coming from the canonical $\mathcal{A}$-left
        module structure on $\left(\mathcal{B}_0 \big/
            \mathcal{J}_\omega\right) \tensor \mathcal{A}_\red$.
    \item \label{item:ExternalIsFullAgain} If
        $\image(\omega\at{\mathcal{B}_0 \cdot \mathcal{B}_0}) =
        \ring{C}$ then \eqref{eq:ExternalInnerProductExplicit} is
        full.
    \end{compactenum}
\end{lemma}
\begin{proof}
    The first statement is part of the functoriality of the external
    tensor product and in fact easy to verify. The second part is
    clear.
\end{proof}
\begin{remark}
    \label{remark:OurE}
    In our example, after the usual identification, we have
    \begin{equation}
        \label{eq:EIsInTensorProductMredG}
        \mathcal{E}
        = \Cinfty_0(G)[[\lambda]] \extensor \Cinfty(M_\red)[[\lambda]]
        \subseteq
        \left(\Cinfty_0(G) \tensor \Cinfty(M_\red)\right)[[\lambda]],
    \end{equation}
    with the inner product being precisely $\SP{\cdot, \cdot}_\red$
    from \eqref{eq:InnerProductExplicit}. In fact, $\mathcal{E}$ as
    constructed in \eqref{eq:TheE} needs not to be $\lambda$-adically
    complete in general but it will be dense in the right hand side of
    \eqref{eq:EIsInTensorProductMredG}. Note that $\omega$ fulfills
    the hypothesis of Lemma~\ref{lemma:External},
    \refitem{item:ExternalIsFullAgain}.
\end{remark}

We can now use the bimodule $\AEAred$ to induce $^*$-representations.
Thus let $\mathcal{D}$ be an auxiliary $^*$-algebra over $\ring{C}$
for the coefficients.
\begin{proposition}
    \label{proposition:RieffelInduction}
    We have a natural equivalence
    \begin{equation}
        \label{eq:RieffelEquivalenceOfFunctors}
        \mathsf{R}_{\mathcal{E}} (\cdot) \cong
        \left(\mathcal{B}_0\big/\mathcal{J}_\omega\right)
        \extensor \cdot
        :
        \Rep[\mathcal{D}](\mathcal{A}_\red)
        \longrightarrow
        \Rep[\mathcal{D}](\mathcal{A}).
    \end{equation}
\end{proposition}
\begin{proof}
    On objects, i.e. on a strongly non-degenerate $^*$-representation
    $\AredHD \in \Rep[\mathcal{D}](\mathcal{A}_\red)$ the Rieffel
    induction is given by
    \[
    \AEAred \itensor[\mathcal{A}_\red] \AredHD
    =
    \left(
        \left(\mathcal{B}_0 \big/\mathcal{J}_\omega\right)
        \extensor
        \mathcal{A}_\red
    \right)
    \itensor[\mathcal{A}_\red]
    \AredHD.
    \]
    This motivates to use the ``associativity'' of the tensor product
    to implement the natural equivalence
    \eqref{eq:RieffelEquivalenceOfFunctors}. Due to the presence of
    the two quotient procedures in $\extensor$ and
    $\itensor[\mathcal{A}_\red]$ we have to be slightly
    careful. Nevertheless, the $\ring{C}$-linear map defined by
    \[
    \mathsf{a}:
    \left(
        \left(\mathcal{B}_0 \big/\mathcal{J}_\omega\right)
        \tensor
        \mathcal{A}_\red
    \right)
    \tensor[\mathcal{A}_\red]
    \AredHD
    \ni
    ([b] \tensor a) \tensor \phi
    \; \mapsto \;
    [b] \tensor (a \cdot \phi)
    \in
    \left(\mathcal{B}_0 \big/\mathcal{J}_\omega\right)
    \tensor
    \AredHD
    \tag{$*$}
    \]
    turns out to be isometric with respect to the inner products on
    both sides. Thus it passes to the quotients and yields an
    isometric and now injective map
    \[
    \mathsf{a}:
    \left(
        \left(\mathcal{B}_0 \big/\mathcal{J}_\omega\right)
        \extensor
        \mathcal{A}_\red
    \right)
    \itensor[\mathcal{A}_\red]
    \AredHD
    \longrightarrow
    \left(\mathcal{B}_0 \big/\mathcal{J}_\omega\right)
    \extensor
    \AredHD.
    \tag{$**$}
    \]
    Since $\mathcal{A}_\red$ is unital and $\Unit_{\mathcal{A}_\red}
    \cdot \phi = \phi$ for all $\phi \in \AredHD$ by assumption, we
    see that ($*$) and hence also ($**$) is surjective. Thus
    $\mathsf{a}$ is unitary. It is now easy to check that $\mathsf{a}$
    is compatible with intertwiners and hence natural as claimed.
\end{proof}
\begin{remark}
    \label{remark:AtLast}
    From this proposition we arrive at the following picture for our
    example: up to the completion issues the Rieffel induction with
    $\Cinftycf(M_\red \times G)[[\lambda]]$ simply consists in
    tensoring the given $^*$-representation of
    $\Cinfty(M_\red)[[\lambda]]$ with the Schrödinger representation
    \eqref{eq:SchroedingerRep} on $\Cinfty_0(G)[[\lambda]]$. Note that
    once the $^*$-representation of $\Cinfty(M_\red)[[\lambda]]$ is
    specified we have everywhere very explicit formulas.
\end{remark}

%
%

\appendix

%
%

\section{Densities on principal bundles}
\label{sec:DensitiesPrincipalBundles}

In this appendix we collect some well-known basic facts on densities
on a principal bundle. The principal bundle will be denoted by $\pi: C
\circlearrowright G \longrightarrow M_\red$ as before. We follow the
tradition that the group acts from the right, $g$ acting via
$\mathsf{R}_g: C \longrightarrow C$. The corresponding left action, as
we used it throughout the main text, is then $\mathsf{L}_g =
\mathsf{R}_{g^{-1}}$.

We fix once and for all a normalization of the constant positive
density $|\D^N x|$ on the vector space $\lie{g}$. Moreover, we
consider a \emph{horizontal lift}
\begin{equation}
    \label{eq:HorizontalLift}
    ^\hor: \Secinfty(TM_\red) \longrightarrow \Secinfty(TC),
\end{equation}
which can e.g. be obtained from a principal connection. For a density
$\Omega \in \Secinfty(\Density T^*M_\red)$ we can define a new density
$\mu \in \Secinfty(\Density T^*C)$ as follows: for $c \in C$ we choose
a basis $X_1, \ldots, X_n \in T_{\pi(c)}M_\red$ and define
\begin{equation}
    \label{eq:LiftOmegatomu}
    \mu_c
    \left(
        X_1^\hor(c), \ldots, X_n^\hor(c), (e_1)_C, \ldots, (e_N)_C
    \right)
    =
    \Omega_{\pi(c)}(X_1, \ldots, X_n) |\D^N x|(e_1, \ldots, e_N).
\end{equation}
This yields indeed a smooth density $\mu$ on $C$ which has the
following properties:
\begin{proposition}
    \label{proposition:DensityLifting}
    Let $C \circlearrowright G \longrightarrow M_\red$ be a principal
    bundle.
    \begin{compactenum}
    \item \label{item:LiftIsSmoothAndIndependent} The definition
        \eqref{eq:LiftOmegatomu} yields a smooth, well-defined density
        $\mu \in \Secinfty(\Density T^*C)$ which is independent on the
        choice of the horizontal lift.
    \item \label{item:LiftPreservesPositivity} Let $c \in C$ then
        $\mu_c$ is positive iff $\Omega_{\pi(c)}$ is positive.
    \item \label{item:LIftModularWeight} For all $g \in G$ one has
        \begin{equation}
            \label{eq:gactsmuDeltagmu}
            \mathsf{R}^*_g \mu = \frac{1}{\Delta(g)} \mu.
        \end{equation}
    \item \label{item:LiftIsomorphism} The map
        \begin{equation}
            \label{eq:OmegaToMu}
            \Secinfty(\Density T^*M_\red)
            \ni \Omega \; \mapsto \; \mu \in
            \Secinfty(\Density T^*C)
        \end{equation}
        is a $\Cinfty(M_\red)$-module monomorphism which is surjective
        onto those densities satisfying \eqref{eq:gactsmuDeltagmu}.
    \end{compactenum}
\end{proposition}
\begin{proof}
    The first part is a simple verification that $\mu$ transforms
    correctly under a change of the bases. Moreover, since passing to
    another horizontal lift changes $X^\hor$ by vertical terms, it
    follows from this block-structure that $\mu$ does not depend on
    the choice of the horizontal lift. The second part is clear. For
    the third, note that the fundamental vector field $\xi_C$
    satisfies $\mathsf{R}^*_g \xi_C = (\Ad_g \xi)_C$. Then
    \eqref{eq:gactsmuDeltagmu} follows easily as we can choose an
    invariant horizontal lift, i.e. we have $\mathsf{R}^*_g X^\hor =
    X^\hor$ for all vector fields $X \in \Secinfty(TM)$. Finally,
    \eqref{eq:OmegaToMu} is clearly $\Cinfty(M_\red)$-linear (along
    $\pi^*$) and injective. Now chose $\Omega > 0$ and thus $\mu > 0$.
    If $\tilde{\mu}$ is a density with \eqref{eq:gactsmuDeltagmu} then
    $\tilde{\mu} = \pi^*u \mu$ with some $u \in \Cinfty(M_\red)$
    showing the surjectivity.
\end{proof}

We need some local expressions for $\mu$ in order to compute
integrations with respect to $\mu$. Thus let $U \subseteq M_\red$ be a
small enough open subset such that there exists a $G$-equivariant
diffeomorphism
\begin{equation}
    \label{eq:PhiTrivializesC}
    \Phi: U \times G \longrightarrow \pi^{-1}(U) \subseteq C,
\end{equation}
i.e. a trivialization. Since we trivialize $C$ as a \emph{right}
principal bundle, the fundamental vector field $\xi_{U \times G}$ on
$U \times G$ at $(p, g)$ is simply given by minus the \emph{left}
invariant vector field $\xi_{U \times G} = - T_e\mathsf{l}_g (\xi)$. Note
that the minus sign appears as we define the fundamental vector fields
with respect to the left action. However, in the density $|\D^nx|$
this does not matter anyway. Since we are free to choose the
horizontal lift we take $X^\hor(p, g) = X(p)$ for $X \in
\Secinfty(TU)$. Then the definition of $\mu$ just gives
\begin{equation}
    \label{eq:muIsDleftTensorOmega}
    \mu = \Omega \boxtimes \Dleft g ,
\end{equation}
i.e. the (external) tensor product of the \emph{left} invariant Haar
density and $\Omega$. Thus for $\phi \in \Cinfty_0(C)$ with support in
$\pi^{-1}(U)$ we get
\begin{equation}
    \label{eq:IntegralphiFubini}
    \int_C \phi \; \mu 
    = \int_{U \times G} \phi(p, g) \Omega(p) \Dleft g.
\end{equation}
A more global interpretation of this local formula is obtained as
follows: for $\phi \in \Cinfty_0(C)$ the integral
\begin{equation}
    \label{eq:intCphidg}
    \int_G \mathsf{R}^*_g \phi \Dleft g\At{c}
    = \int_G \phi(\mathsf{R}_g(c)) \Dleft g
\end{equation}
yields an invariant smooth function on $C$ since $\Dleft g$ is left
invariant. Thus it is of the form $\pi^*u$ with some function $u \in
\Cinfty_0(M_\red)$. Note that $u$ still has compact support.
Using the above local result and a partition of unity argument we see
that for this function $u$ we have
\begin{equation}
    \label{eq:IntMreduIsIntphiC}
    \int_{M_\red} u \; \Omega = \int_C \phi \mu.
\end{equation}
With some slight abuse of notation (omitting the $\pi^*$) we therefor
write
\begin{equation}
    \label{eq:intphiMu}
    \int_C \phi \; \mu 
    = \int_{M_\red} \left(\int_G \mathsf{R}_g^* \phi \Dleft g\right) \Omega.
\end{equation}

%
%

\begin{footnotesize}
    \renewcommand{\arraystretch}{0.5} 

\end{footnotesize}

\end{document}